\documentclass[12pt,twoside]{article}

\pdfoutput=1

\usepackage[latin1]{inputenc}
\usepackage{amsmath,amsthm,amssymb, mathrsfs, calrsfs} 
\usepackage{tikz}
\usepackage{graphicx}
\graphicspath{{./images/}}

\usepackage{pdfsync}
\usepackage{epsfig}
\usepackage{epstopdf}
\usepackage{mystyle}
\usepackage{url}
\usepackage{ifb}

\graphicspath{{./images/}}

\title{Piecewise rigid curve deformation \\ via a Finsler steepest descent}

\author{
	\begin{tabular}{c}
		Guillaume Charpiat\\[6mm]
		Stars Lab \\ 
		INRIA Sophia-Antipolis\footnote{\texttt{Guillaume.Charpiat@inria.fr}}
	\end{tabular} 
	\begin{tabular}{c}
		Giacomo Nardi, Gabriel Peyr\'e \\ Fran\c{c}ois-Xavier Vialard \\[2mm]
		Ceremade, \\ Universit\'e Paris-Dauphine\footnote{\texttt{\{nardi,peyre,vialard\}@ceremade.dauphine.fr}}\\
	\end{tabular}	
}
\date{}

\begin{document}

\maketitle

\begin{abstract}
This paper introduces a novel steepest descent flow in Banach spaces. This extends previous works on generalized gradient descent,  notably the work of Charpiat et al.~\cite{charpiat-generalized-gradient}, to the setting of Finsler metrics. Such a generalized gradient  allows one  to take into account a prior on  deformations  (e.g., piecewise rigid)  in order to favor some specific evolutions. We define a Finsler gradient descent method to minimize a functional defined on a Banach space and we prove a convergence theorem for such a method.  In particular, we show that the use of non-Hilbertian norms  on  Banach spaces is  useful  to study non-convex optimization problems where the geometry of the space might play a crucial role to avoid poor local minima.  We show some applications to the curve matching problem. In particular, we characterize piecewise rigid deformations on the space of curves and we study several models  to perform piecewise rigid evolution of curves. 


\end{abstract}
\section{Introduction}

This paper introduces a new descent method to minimize energies defined over Banach spaces. This descent makes use of a generalized gradient which corresponds to a descent direction for a Finsler geometry. We show applications of this method to the optimization over the space of curves, where this Finsler gradient allows one to construct piecewise regular curve evolutions.

\subsection{Previous Works} 
\label{sec-previous-works}

\paragraph{Energy minimization for curve evolution.}

The main motivation for this work is the design of novel shape optimization methods, with an emphasis toward curves evolutions. Shape optimization is a central topic in computer vision, and has been introduced to solve various problems such as image segmentation or shape matching. These problems are often solved by introducing an energy which is minimized over the space of curves.  The first variational method proposed to perform image segmentation through curve evolution is the snake model~\cite{Kass88Snakes}. This initial proposal has been formalized using intrinsic energies depending only on the geometry of the curves. A first class of energies corresponds to a weighted length of the curve, where the weight acts as an edge detector~\cite{caselles-active-contours,malladi-shape-modeling}. A second class of segmentation energies, pioneered by the Mumford-Shah model~\cite{MumfordShah89}, integrates a penalization both inside and outside the curve, see for instance~\cite{CV-01}. Shape registration requires to compute a matching between curves, which in turn can be solved by minimizing energies between pairs of curves. An elegant framework to design such energies uses distances over a space of measures or currents, see~\cite{Glaunes-matching} for a detailed description and applications in medical imaging.

Curve evolution for image processing is an intense area of research, and we refer for instance to the following recent works for applications in image segmentation~\cite{a,b,c,d} and matching~\cite{SotirasDP13,e,f}.

\paragraph{Shape spaces as Riemannian spaces.}

Minimizing these energies requires to define a suitable space of shapes and a notion of gradient with respect to the geometry of this space. The  mathematical study of spaces of curves has been largely investigated in the last years, see for instance~\cite{Younes-elastic-distance,Mennucci-CIME}. 
The set of curves is naturally modeled over a Riemannian manifold~\cite{MaMi}. This corresponds to using a Hilbertian metric on each tangent plane of the space of curves, i.e. the set of vector fields which deform infinitesimally a given curve. This Riemannian framework allows one to define geodesics which are shortest paths between two shapes~\cite{Younes-explicit-geodesic,Mennucci-Stiefel-manifolds}. Computing minimizing geodesics is useful to perform shape registration~\cite{Mennucci-filtering,Younes-04,2010.03.20}, tracking~\cite{Mennucci-filtering} and shape deformation~\cite{Kilian-shape-space}. The theoretical study of the existence of these geodesics depends on the Riemannian metric. For instance, a striking result~\cite{MaMi,Yezzi-Menn-2005b,Yezzi-Menn-2005a} is that the natural $L^2$-metric on the space of curves, that has been largely used in several applications in computer vision, is a degenerate Riemannian metric: any two  curves have distance equal to zero with respect to such a metric.

Beside the computation of minimizing geodesics, Riemannian metrics are also useful to define descent directions for shape optimization. Several recent works~\cite{MaMi,charpiat-new-metrics,Yezzi-Menn-2005a,Yezzi-Menn-2005b} point out that  the choice of the metric, which the gradient depends on,  notably affects the results of a gradient descent algorithm. Carefully designing the metric is thus crucial to reach better local minima of the energy. Modifying the descent flow can also be important for shape matching applications. A typical example of such  Riemannian metrics are Sobolev-type metrics~\cite{sundaramoorthi-sobolev-active,sundaramoorthi-2006,sundaramoorthi-new-possibilities, Yezzi-H2} which lead to smooth curve evolutions.

\paragraph{Shape spaces as Finslerian spaces.}

It is possible to extend the Riemannian framework by considering more general metrics on the tangent planes of the space of curves. Finsler spaces make use of Banach norms instead of Hilbertian norms~\cite{RF-geometry}. A few recent works~\cite{Mennucci-CIME,Yezzi-Menn-2005a} have studied the theoretical properties of Finslerian spaces of curves. To the best of our knowledge, with the notable exception of~\cite{charpiat-generalized-gradient}, which is discussed in detail in Section~\ref{sec-relation-charpiat}, no previous work has used Finslerian metrics for curve evolution.

\paragraph{Generalized gradient flow.}

Beyond shape optimization, the use of non-Euclidean geometries is linked to the study of generalized gradient flows. Optimization on manifolds requires the use of Riemannian gradients and Riemannian Hessians, see for instance~\cite{Opt-manifolds}. Second order schemes on manifolds can be used to accelerate shape optimization over Riemannian spaces of curves, see~\cite{Ring-opti-manifolds}. Optimization over Banach spaces requires the use of convex duality to define the associated gradient flow~\cite{Ulbrich,prox-banach-1,prox-banach-2}. It is possible to generalize these flows for metric spaces using implicit descent steps, we refer to~\cite{AGS} for an overview of the theoretical properties of the resulting flows.

\subsection{Motivation} 

The metrics defined over the tangent planes of the space of curves (e.g. an Hilbertian norm in the Riemannian case and a Banach norm in the Finsler case) have a major impact on the trajectory of the associated gradient descent. This choice thus allows one to favor specific evolutions. A first reason for introducing a problem-dependent metric is to enhance the performances of the optimization method. Energies minimized for shape optimization are non-convex, so a careful choice of the metric is helpful to avoid being trapped in a poor local minimum. A typical example is the curve registration problem, where reaching a non-global minimum makes the matching fail. A second reason is that, in some applications, one is actually interested in the whole descent trajectory, and not only in the local minimum computed by the algorithm. For the curve registration problem, the matching between the curves is obtained by tracking the bijection between the curves during the evolution. Taking into account desirable physical properties of the shapes, such as global or piecewise rigidity, is crucial to achieve state of the art results, see for instance~\cite{citeulike,chang08articulated,Shelton-2000}. In this article, we explore the use of Finsler gradient flows to encode piecewise rigid deformations of the curves.

\subsection{Contributions} 

Our first contribution is the definition of a novel generalized gradient flow, that we call Finsler descent, and the study of the convergence properties of this flow. This Finsler gradient is obtained from the $W^{1,2}$-gradient through the resolution of a constrained convex optimization problem. Our second contribution is the instantiation of this general framework to define piecewise rigid curve evolutions, without knowing in advance the location of the articulations. This contribution includes the definition of novel Finsler penalties to encode piecewise rigid and piecewise similarity evolutions. It also includes the theoretical analysis of the convergence of the flow for $BV^2$-regular curves. Our last contribution is the application of these piecewise regular evolutions to the problem of curve registration. This includes the definition of a discretized flow using finite elements, and a comparison of the performances of Riemannian and Finsler flows for articulated shapes registration.  The Matlab code to reproduce the numerical results of this article is available online\footnote{\url{https://www.ceremade.dauphine.fr/~peyre/codes/}}.

\subsection{Relationship with~\cite{charpiat-generalized-gradient}}
\label{sec-relation-charpiat}

Our work is partly inspired by the generalized gradient flow originally defined in~\cite{charpiat-generalized-gradient}. We use a different formulation for our Finsler gradient, and in particular consider a convex constrained formulation, which allows us to prove convergence results. An application to piecewise rigid evolutions is also proposed in~\cite{charpiat-generalized-gradient}, but it differs significantly from our method. In~\cite{charpiat-generalized-gradient}, piecewise rigid flows are obtained using a non-convex penalty, which poses both theoretical difficulties (definition of a suitable functional space to make the descent method well-defined) and numerical difficulties (computation of descent direction as the global minimizer of a non-convex energy). In our work  we prove a characterization of piecewise rigid deformations that enables the definition of a penalty  depending on the deformation (instead of instantaneous parameters as done in~\cite{charpiat-generalized-gradient}). Then, we generalize this penalty to the $BV^2$-framework obtaining a convex penalty for $BV^2$-piecewise rigid deformations.

\subsection{Paper Organization}

Section~\ref{F} defines the Finsler gradient and the associated steepest descent in Banach spaces, for which we prove a convergence theorem. Section~\ref{SC} introduces the space of $BV^2$-curves and studies its main properties, in particular its stability to reparametrization.
Section~\ref{PWR} characterizes $C^2$-piecewise rigid motions and defines a penalty in the case of  $BV^2$-regular motions. We apply this method in Section~\ref{CM} to the curve registration problem. We minimize a matching energy using the Finsler descent method for $BV^2$-piecewise rigid motions. Section~\ref{discretization} details the finite element discretization of the method. Section~\ref{examples} gives numerical illustrations of the Finsler descent method for curve matching. Section~\ref{similarity} refines the model introduced in Section~\ref{PWR} to improve the matching results by replacing piecewise rigid transforms with piecewise similarities.

\section{Finsler Descent Method in Banach Spaces}\label{F}

Let   $(H, \dotp{\cdot}{\cdot}_{H})$ be  a Hilbert space and let $E$ be a Fr\'echet differentiable energy defined on $H$.

We consider a  Banach space $(\Bb, \|.\|_\Bb)$  which is dense in $H$ and continuously embedded in $H$, and we consider the restriction of  $E$ to $\Bb$ (such a restriction will be also denoted by $E$). 

We aim to solve the following minimization problem
\begin{equation}\label{initial-eq}
	\underset{\GA\in \Bb}{\inf}\; E(\Ga)
\end{equation}
using a steepest descent method. We treat $\Bb$ as a manifold modeled on itself and denote  by $T_\GA \Bb$ the tangent space at $\GA \in \Bb$. In the following we suppose that at every point $\GA\in \Bb$, the space $T_\GA \Bb$ coincides with $\Bb$, although our descent method can be adapted to more general settings. 

For every $\GA\in \Bb$  we define an inner product  $\dotp{\cdot}{\cdot}_{H(\Ga)}$ that is continuous with respect to  $\Ga \in \Bb$, and we suppose that the norms $\|\cdot\|_H$ and $\|\cdot\|_{H(\GA)}$ are uniformly equivalent for every $\GA$  belonging to a ball of $\Bb$ (with respect to the norm on $\Bb$). This makes $H$ complete with respect to the norm $\|\cdot\|_{H(\GA)}$. 
Note that this inner product may be different from the inner product induced by $\dotp{\cdot}{\cdot}_H$ on $T_\GA \Bb$, and in particular it might depend on $\GA$. For instance in the case of Sobolev metrics for the space of curves we usually consider $H= W^{1,2}([0,1],\RR^2)$ and set $\Bb=T_\GA \Bb= W^{1,2}([0,1],\RR^2)$ equipped with the measure defined by the arclength of $\GA$ (see Remark \ref{rem-fins-sob}).

Since $E$ is Fr\'echet differentiable and $(H, \dotp{\cdot}{\cdot}_{H(\GA)})$ is a Hilbert space, by the Riesz representation theorem,  there exists a unique vector  
$v\in H$ such that
$$D E(\GA)(\Phi) = \dotp{v}{\Phi}_{H(\Ga)}\quad \forall \,\Phi\in T_\GA \Bb\,.$$
The vector $v$ represents  
the gradient of $E$ at $\GA$ with respect to the inner product $\dotp{\cdot}{\cdot}_{H(\Ga)}$, and it  is denoted by $v=\nabla_{H(\GA)} E(\GA)$.

\subsection{Finsler Gradient}

The Finsler gradient determines a descent direction by modifying $\nabla_{H(\GA)} E(\GA)$ with respect to a penalty $R_\GA$ that  depends on $\GA$. It is defined by minimizing $R_\GA$ under a constraint $\Ll_\GA$.

\begin{defn}[\BoxTitle{Finsler gradient}] \label{FinslerGrad}
For every $\GA \in \Bb$, let $R_\GA : T_\GA\Bb \rightarrow \R^+\cup\{+\infty\}$ be a function such that  $R_\GA \neq +\infty$ and  $\Ll_\GA \subset T_\GA\Bb$  a   set satisfying
\begin{equation}\label{cond-constraints}
	\Ll_\GA \subset \enscond{
		 	\Phi \in T_\GA\Bb
		}{ 
			\left\langle \nabla_{H(\GA)}E (\GA) , \Phi \right\rangle_{H(\GA)} \geq (1-\rho) \|\nabla_{H(\GA)}E (\GA)\|_{H(\GA)} \|\Phi \|_{H(\GA)}
		 } \,
\end{equation}
where $\rho \in (0,1)$ is a fixed parameter independent of $\GA$. This parameter is typically adapted to the targeted applications (see Section~\ref{examples}), and in particular to the energy $E$. If $R_\GA$ admits a minimum on $\Ll_\GA$ then  a Finsler  gradient for $E$ at $\GA$ with respect to $R_\GA$ is defined as: 
	\eql{\label{defgrad}
		\nabla_{R_\GA} E(\GA) \in {\rm argmin}\; \enscond{ R_\GA(\Phi) }{  \Phi\in \Ll_\GA } \,.
	}
Note that if $ \nabla_{H(\GA)}E (\GA)=0$ then $\Gamma$ is a critical point and any descent algorithm  stops. Note that $R_\Gamma$ is in general not equal to the Banach norm defined over the tangent space. This is important for some applications, such as the one considered in Section \ref{PWR} (piecewise rigid deformations).
\end{defn}	

\par The next theorem gives an existence result for the Finsler gradient which is proved by using the standard  direct method of  calculus of variations. 

\begin{thm}\label{general-existence} Let $T_\GA\Bb$ be a Banach space equipped with a topology $\mathcal{T}(T_\GA\Bb)$ such that every bounded sequence  in $T_\GA\Bb$ converges (up to a subsequence) with respect to the topology $\mathcal{T}(T_\GA\Bb)$.  Let $R_\GA$ be  coercive (i.e., $R_\GA(\Phi)\rightarrow +\infty$ as $\|\Phi\|_{T_\GA\Bb}\rightarrow +\infty$) and lower  semi-continuous with respect to the topology $\mathcal{T}(T_\GA\Bb)$ and we suppose that $\Ll_\GA$ is closed in $T_\GA\Bb$ with respect to the topology $\mathcal{T}(T_\GA\Bb)$. Then Problem~\eqref{defgrad} admits at least a solution.
\end{thm}

\begin{proof} As $R_\GA$ is coercive, every minimizing sequence is bounded in $T_\GA\Bb$ so  it converges (up to a subsequence) with respect to the topology $\mathcal{T}(T_\GA\Bb)$ toward an element of $\Ll_\GA$. Now, because of the lower semi-continuity of $R_\GA$, the theorem ensues.
\end{proof}

\par Such a result is the generalization of the usual existence theorem of calculus of variations on a reflexive Banach space. In fact (see Corollary 3.23 p. 71 in \cite{Brezis}), if  $T_\GA\Bb$ is reflexive, the existence of the Finsler gradient is guaranteed whenever  $\Ll_\GA$ is convex and closed with respect to the strong topology of $T_\GA\Bb$, and $R_\GA$ is  coercive, $R_\GA\neq+\infty$, convex, and lower semi-continuous  with respect to the strong topology of $T_\GA\Bb$. These hypotheses guarantee in particular an existence result if $T_\GA\Bb$ is a Hilbert space.
\par The previous theorem guarantees the existence of a minimum on non-reflexive Banach spaces. The key point is the existence of a suitable topology which guarantees compactness of minimizing sequences. We point out that, in general, such a topology is weaker than the strong topology of the Banach space. 
\par We point out that the applications studied in this work concern a minimization problem on $T_\GA\Bb=BV^2(\Circ,\RR^2)$. Such a space is not reflexive but the weak* topology of $BV^2(\Circ,\RR^2)$ satisfies the hypotheses of the previous theorem (see Appendix). Then, for some suitable set $\Ll_\GA$ and penalty $R_\GA$, the existence of the Finsler gradient is guaranteed.
\par The set $\Ll_\GA$ imposes a constraint on the direction of the Finsler gradient and more precisely on the angle between the Finsler and Hilbert gradient. It is crucial to guarantee the convergence of the descent method by the Zoutendijk theorem (see Theorem~\ref{convergence}). The parameter $\rho$ controls the deviation of the Finsler gradient with respect to $\nabla_{H(\GA)}E (\GA)$. This parameter can be tuned by the user to modify the geometry of the trajectory of the flow defined in Section~\ref{subsec-finsler-descent}. The impact of $\rho$ is studied by several numerical simulations in Section~\ref{sec-influ-rho}.

If  the hypotheses of Theorem \ref{general-existence} are verified then  the minimum in~\eqref{defgrad} exists, but in general it is not unique. A Finsler gradient is any minimum of the functional minimized in~\eqref{defgrad}.

Condition~\eqref{cond-constraints} implies
\begin{equation}\label{strictdirection}
	\left\langle \frac{\nabla_{H(\GA)}E (\GA)}{\|\nabla_{H(\GA)}E (\GA)\|_{H(\GA)}} , \frac{\nabla_{R_\GA}E (\GA) }{\|\nabla_{R_\GA}E (\GA) \|_{H(\GA)}}\right\rangle_{H(\GA)} \geq (1-\rho)>0 \quad\quad \foralls \GA\in \Bb.
 \end{equation}
This shows that the Finsler gradient is a valid descent direction, in the sense that
\eq{
	\frac{\d}{\d t} E(\Ga - t\nabla_{R_\GA}E (\GA))\Big|_{t=0} = 
	- \dotp{ \nabla_{H(\GA)}E (\GA) }{ \nabla_{R_\GA}E (\GA) }_{H(\GA)}
	< \, 0\,.
}

\begin{rem}[\BoxTitle{Relationship with~\cite{charpiat-generalized-gradient}}]
Our definition of Finsler gradient is partly inspired by the generalized gradient introduced  in Section 6.1 of~\cite{charpiat-generalized-gradient}. An important difference is that we introduce a constraint $\Ll_\GA$ whereas ~\cite{charpiat-generalized-gradient} defines the gradient as a minimum of $DE(\GA)(\Phi)+R_\GA(\Phi)$ on $T_\GA\Bb$. This is a crucial point because, as shown in the next section, this constraint guarantees the convergence of the descent method associated with the Finsler gradient toward a stationary point of $E$.  
\end{rem}

\begin{rem}[\BoxTitle{Relationship with Sobolev gradient}]\label{rem-fins-sob} We consider the spaces $\Bb=W^{1,2}([0,1],\RR^2)$, $H=L^2([0,1],\RR^2)$. More precisely, for every $\GA\in\Bb$, we set $T_\GA \Bb= W^{1,2}([0,1],\RR^2)$ and we denote by  $L^2(\GA)$ the space $L^2([0,1],\RR^2)$ equipped with the norm 
\eq{
	\norm{\Psi}_{L^2(\GA)}^2 = \int_0^1 |\Psi(s)|^2 |\GA'(s)| \d s.
}
 In order to make such a norm well-defined we suppose that  $|\GA'(s)|\neq 0$  for a.e. $s\in \Circ$.
This setting models smooth parametric planar curves and their deformations $\Psi$. Note that the space of curves is further detailed in Section~\ref{SC}. 

We introduce 
\eq{
	R_\GA(\Phi) = \norm{ D\Phi }_{L^2(\GA)}^2, \qquad  \foralls \Phi \in T_\GA \Bb,
}
\eql{\label{eq-constr-sobolgrad}
	\Ll_\Ga= \enscond{  
		\Phi\in T_\GA \Bb
	}{
		\norm{\nabla_{L^2(\GA)}E(\GA) - \Phi }_{L^{2}(\GA)} \leq \rho \norm{\nabla_{L^2(\GA)}E(\GA) }_{L^{2}(\GA)}
	}~
}
where we denote by $D\Phi$ the weak derivative of $\Phi$. Note that $\Ll_\GA$ satisfies condition~\eqref{cond-constraints}.
For a given differentiable energy $E$,~\eqref{defgrad} becomes
\begin{equation}	\label{example-sobolev-finsler}	
	\nabla_{R_\GA} E(\GA) \in \uargmin{ \Phi \in \Ll_\Ga } \norm{ D\Phi }_{L^2(\GA)}^2~.
\end{equation}
We remark that, comparing with Proposition 4 p. 17 in~\cite{charpiat-generalized-gradient}, the Finsler gradient~\eqref{example-sobolev-finsler}  represents a constrained version of the Sobolev gradient. Note also that in  Definition~\ref{FinslerGrad}, the penalty $R_\GA$ need not  be quadratic so that the negative Finsler gradient can be understood as a generalization of the Sobolev gradient.
\end{rem}

\subsection{Finsler Descent Method}
\label{subsec-finsler-descent}

In this section we consider the minimization problem~\eqref{initial-eq} of  an energy $E$ on  $\Bb$. Given some initialization $\GA_0 \in \Bb$, the Finsler gradient descent is defined as
\begin{equation}\label{sequence}
	\GA_{k+1} = \GA_k - \tau_k \nabla_{R_{\GA_k}} E(\GA_k)
\end{equation}
where  $\nabla_{R_{\GA_k}} E(\GA_k)$ is any minimizer of~\eqref{defgrad} and the step size $\tau=\tau_k>0$ is chosen in order to satisfy the Wolfe constraints 
\begin{equation}\label{Wolfe}
\left\{
\begin{array}{lll}
	E(\GA+\tau v)&\leq& E(\GA)+\alpha\tau \dotp{\nabla_{H}E(\GA)}{v}_{H}\\
	\dotp{ \nabla_{H}E(\GA+\tau v) }{ v }_{H} &\geq& \beta \dotp{\nabla_{H}E(\GA) }{ v }_{H}
\end{array}\right.
\end{equation}
for some  fixed $0<\alpha<\beta<1$ and with $v = -\nabla_{R_{\GA_k}} E(\GA_k)$, see for instance~\cite{Nocedal}, p.37.

We have the following result.

\begin{thm}\label{convergence} 
Let $E\in C^1(H, \R^+)$ be a non-negative energy. We suppose that 
there exists a constant $L>0$ such that 
\begin{equation}\label{grad-lip}
\|\nabla_{H} E (\GA_1) - \nabla_{H} E (\GA_2)\|_H \leq L \|\GA_1-\GA_2\|_H\quad \forall\, \GA_1, \GA_2 \in H\,.
\end{equation}
Then, for the sequence  $\{\GA_k\}_k$ (defined in~\eqref{sequence}),
$\|\nabla_{H} E(\Ga_{k})\|_H\rightarrow 0$. 

\end{thm}

\begin{proof} 
Since $\{ \GA_k \}$ is the sequence defined by the gradient descent satisfying the assumption of the Zoutendijk theorem (see~\cite{Nocedal}: Theorem 3.2 p.43) for the ambient norm on $H$, we have:
$$\sum_{k=0}^{\infty} \dotp{ 
 		\frac{\nabla_{H}E (\GA_k)}{\|\nabla_{H} E(\Ga_k)\|_H}
	}{ 
		\frac{\nabla_{R_{\GA_k}}E(\GA_k)}{\|\nabla_{R_{\GA_k}}E(\Ga_k)\|_H} 
	}_{H}^2 \,\|\nabla_{H} E(\Ga_k)\|_{H}^2\, < \, \infty\,.$$


As we have assumed that the  norms $\|\cdot\|_H$ and $\|\cdot\|_{H(\GA)}$ are equivalent on every bounded ball of $\Bb$, for $k$ large enough, the condition \eqref{strictdirection} implies :
$$\dotp{ 
 		\frac{\nabla_{H}E (\GA_{k})}{\|\nabla_{H} E(\Ga_{k})\|_H}
	}{ 
		\frac{\nabla_{R_{\GA_{k}}}E(\GA_{k})}{\|\nabla_{R_{\GA_{k}}}E(\Ga_{k})\|_H} 
	}_{H} \geq (1 - \rho)M >0\,
$$
with $M>0$. 
This follows by the fact that 
$$ \langle \nabla_{H}E (\GA_{k}), 
\nabla_{R_{\GA_{k}}} E \rangle_H = DE(\GA_{k})
(\nabla_{R_{\GA_{k}}}E) = \langle \nabla_{H(\GA)}E 
(\GA_{k}), \nabla_{R_{\GA_{k}}}E\rangle_{H(\GA)}$$
and the equivalence of the norms applied to  \eqref{cond-constraints}.

Therefore, we have in particular
$$\sum_{k=0}^{\infty}  \, \|\nabla_{H} E(\Ga_{k})\|_{H}^2\, < \, \infty\,,$$
and the result ensues.
\end{proof}

\begin{rem} 
[{\bf On the Zoutendijk theorem}] In the  previous proof we applied the Zoutendijk theorem in infinite dimensions which is not the case in ~\cite{Nocedal}. However, their proof can be straightforwardly generalized to the case of infinite dimensional  Hilbert spaces.
 \end{rem}

Note that the sequence defined by the Finsler descent method could diverge (for instance if $\nabla_H E(\GA)\rightarrow 0$ as $\|\GA\|_{\Bb}\rightarrow  +\infty$). However, if $E$ is coercive, its level sets are compact with respect to some weaker topology $\tau$ of $\Bb$, and the $H$-gradient is continuous with respect to such a  weak topology, then  the previous theorem guarantees  the convergence of the  Finsler descent  method toward a stationary point of the energy. In fact, as $E$ is coercive, we  have that $\{\GA_k\}$ is uniformly bounded in $\Bb$. Then, as   the level sets of $E$ are $\tau$-weakly compact,  $\{\GA_k\}$ $\tau$-weakly converges (up to a subsequence)
 to an element of $\Bb$. Because of the continuity property of $E$, such a point is a stationary point of $E$.

 \begin{rem}\label{grad_flow}An interesting problem would be to show that the Finsler gradient descent scheme admits a limit flow when the step size tends to zero, or to show that the machinery of gradient flows over metric spaces (see \cite{AGS}) can be adapted to our setting. We believe this is however not trivial and decided to leave this for future work. 
 \end{rem}

\section{Finsler Gradient in the Spaces of Curves}\label{SC}

This section specializes our method to a space of piecewise-regular curves. We target applications to piecewise rigid evolutions to solve a curve matching problem (see Section~\ref{CM}). Note that, in order to perform piecewise rigid evolutions, we are led to deal with curves whose first and second derivatives are not continuous. This leads us to consider  the setting of $BV^2$-functions. We refer the reader to Appendix  for the definition and the main properties of $BV$ and $BV^2$ functions.

\subsection{$BV^2$-curves}

In this section we define the space of $BV^2$-curves and introduce its main properties. This models closed, connected curves admitting a  $BV^2$-parameterization.
\par In the following, for every $\GA\in BV^2(\Circ, \RR^2)$, we denote by  $\d\GA$ the measure defined as 

$$
\d \GA(A)=\int_A |\GA'(s)| \d s \;, \quad \forall \, A\subset \Circ
$$
where $\GA'$ denotes the approximate derivative of $\GA$ (see for instance \cite{AFP}). In the following we identify $[0,1]$ with the unit circle $\Circ$.

\begin{defn}[{\bf $BV^2$-curves}]\label{BV2curves}
We set $\Bb = BV^2( \Circ, \R^2)$ equipped with the  $BV^2$-norm.
For any $\Ga \in \Bb$, we set $T_\GA \Bb = BV^2(\GA)$, the space $BV^2(\Circ,\RR^2)$ equipped with the measure  $\d\GA$.
In $BV^2(\GA)$, differentiation and integration are done with respect to the measure $\d\GA$. For every $\Phi \in T_\GA \Bb$, we have  in particular 

$$\frac{\d \Phi}{\d \GA}(s)\,=\, \underset{\varepsilon \rightarrow 0}{\lim} \, \frac{\Phi(s+\varepsilon)-\Phi(s)}{\d \GA((s-\varepsilon, s+\varepsilon))} \,,\quad \|\Phi\|_{L^1(\GA)} = \int_{\Circ} |\Phi(s)||\GA'(s)|\, \d s\,. $$
We also point out that $\frac{\d \Phi}{\d \GA}(s)= \Phi'(s)/|\GA'(s)|$ for a.e. $s\in \Circ$, which implies that such a derivative is Lebesgue-measurable. Remark that in order to make the previous derivation well defined we have to make a hypothesis  on the derivative. We refer to next section, in particular to \eqref{cond-derivative}, for a discussion about the necessity of such a condition.

The first and second variation are defined as
\begin{equation}\label{TV}
	 TV_\GA\left(\Phi\right) = \sup \enscond{
		\int_{\Circ}\Phi(s)\cdot \frac{\d g}{\d \GA}(s)\, \d\Ga(s) 
	}{
		g\in \mathrm{C}^1_c(\Circ, \RR^2),\|g\|_{L^\infty(\Circ,\R^2)}\leq 1
	} \,
\end{equation}
and 
 \begin{equation}
 TV^2_\GA(\Phi)=\sup \enscond{
 	\int_{\Circ}\Phi\cdot \frac{\d^2 g}{\d\GA^2}(s)\, \d \GA(s) 
	}{
		g\in \mathrm{C}^2_c({\Circ},\mathbb{R}^2 ),\|g\|_{L^{\infty}(\Circ,\mathbb{R}^2 )}\leq 1
	}\,
\end{equation}
for every $\Phi\in BV^2(\Circ,\RR^2)$. Now, as $\frac{\d g}{\d\GA}(s)\, \d\Ga(s)=g'(s)\,\d s$ we get $TV_\GA(\Phi) = \|\Phi'\|_{L^1(\Circ,\RR^2)}$.   The $BV^2$-norm on the tangent space is defined by 
$$\|\Phi\|_{BV^2(\GA)} = \|\Phi\|_{W^{1,1}(\GA)} + TV^2_\GA(\Phi)\,.$$
In a similar way we define $W^{2,1}(\GA)$. Every $\Phi \in T_\GA \Bb$ operates on a curve $\GA \in \Bb$ as
\eq{
	(\GA + \Phi)(s)= \GA(s) + \Phi(s)\,,\quad \foralls s\in\Circ.
}

\end{defn}

\begin{defn}[{\bf Tangent, normal, and curvature}]\label{curvature-dfn}
For every $\GA\in \mathcal{B}$ we define  the following vector
$$\nu_\GA(s)= \underset{r\rightarrow 0}{\lim}\, \frac{D \GA((s-r,s+r))}{|D \GA|((s-r,s+r))} \,$$
where $|D\GA|$ denotes the total variation of $\GA$ and $D\GA$ denotes the vector-valued measure associated with the total variation. Note that, as $\GA\in W^{1,1}(\Circ, \RR^2)$, $|D\GA|$ coincides with the measure $\d \GA$ (we recall that the total variation of a $ W^{1,1}$-function coincides with the $L^1$-norm of its derivative) and the limit defining $\nu_\GA$ exists  for $\d\GA$-a.e. $s\in \Circ$.
Moreover we have $\|\nu_\GA\|=1$ for $\d\GA$-a.e. $s\in\Circ$.


Now, $\GA\in W^{1,1}(\Circ, \RR^2)$, and we can suppose  that $|\GA'(s)|\neq 0$ for almost every $s\in \Circ$. This implies in particular that a subset of $\Circ$ is $\d\GA$-negligible if and only if it is Lebesgue-negligible.  Then,  the tangent and normal vectors to the curve at $\d\GA$-a.e. point $\Ga(s)$ are defined as 
\begin{equation}\label{def-tangente}
	\tgam(s) = \frac{\nu_\GA(s)}{\|\nu_\GA(s)\|}\quad \quad \n_\ga(s)= \tgam(s)^\bot
\end{equation}
\eq{
	\qwhereq (x,y)^\bot= (-y,x)\;,\quad \foralls (x,y)\in \RR^2.
}
\par  We point out that $\nu_\GA(s) = \GA'(s)/|\GA'(s)|$ for a.e. $s\in \Circ$ and  $\tgam\in BV(\Circ,\RR^2)$ with   $\tgam\cdot\tgam=1$ for a.e. $s\in \Circ$. Thus, by differentiating with respect to $\d \GA$, we get that the measure  $\tgam\cdot D_\GA\tgam $ is null ($D_\GA$ denotes here the vector-valued measure associated with the total variation $TV_\GA$).
Then, there exists a real measure $\rm{curv}_\GA$ such that
\begin{equation}\label{deriv-tang}
D_\GA\tgam = \ngam \,\rm{curv}_\GA\,.
\end{equation}
By the definition of $\ngam$ we also have
\begin{equation}\label{deriv-norm}
	D_\GA\ngam = -\tgam \,\rm{curv}_\GA\,.
\end{equation}
The measure $\rm{curv}_\GA$ is  called generalized curvature of $\GA$, and, in the case of a smooth curve,  it  coincides with the measure $\kappa_\GA \,\d s$ where $\kappa_\GA$ denotes the standard scalar curvature of $\GA$. 
\par From the properties of the total variation  (see for instance~\cite{AFP}) it follows that 
\begin{equation}\label{curvature}
	|\rm{curv}_\GA|(\Circ)\leq |D^2\GA|(\Circ)
\end{equation}
where $|\rm{curv}_\GA|(\Circ)$ denotes the total variation of the generalized curvature on the circle. 
\end{defn}

\begin{defn}[{\bf Projectors}]
We denote by $\Pi_\GA$ the  projection on the normal vector field $ \n_\ga$
\begin{equation}\label{normal-proj}
	\Pi_\GA (\Phi)(s) \;\;=\;\; \Big( \Phi(s) \cdot \n_\ga(s) \Big)\, \n_\ga(s),
\end{equation}
where $\cdot$ is the inner product in $\RR^2$. 
We denote by $\Sigma_\GA$ the  projection on the tangent vector  field $\tgam$
\begin{equation}\label{tangent-proj}
	\Sigma_\GA (\Phi)(s) \;\;=\;\; \Big( \Phi(s) \cdot \tgam(s) \Big)\, \tgam(s)\, .
\end{equation}
\end{defn}

\begin{defn}[{\bf Hilbertian structure}]\label{hilbert-struct}
The Banach space $\Bb=BV^2(\Circ,\RR^2)$ is continuously embedded in the  Hilbert space $H =W^{1,2}(\Circ,\RR^2)$. 

For every $\GA\in \Bb$, we define $W^{1,2}(\GA) = W^{1,2}(\Circ, \RR^2)$, where integration is done with respect to the measure $\d \GA$. In particular, if $\GA$ verifies $\underset{s\in\Circ}{{\rm essinf}}|\GA'(s)|> 0$, then the norms of $W^{1,2}(\Circ, \RR^2)$ and $W^{1,2}(\GA)$ are equivalent.  
This defines the following inner product on $T_\GA \Bb$
\eql{\label{eq-inner-prod-curve}
	\langle\Phi,\,\Psi\rangle_{W^{1,2}(\GA)} =\int_{\Circ} \Phi(s) \cdot \Psi(s)\; \d\Ga(s) + \int_{\Circ} \frac{\d\Phi}{\d \GA}(s) \cdot  \frac{\d\Psi}{\d \GA}(s)\; \d\Ga(s) \quad \forall \, \Phi, \Psi\in T_\GA \Bb\,.
}

\end{defn}


Finally, recall that for a Fr\'echet-differentiable energy $E$ on $H$, the $W^{1,2}(\GA)$-gradient of $E$ at $\Ga$ is defined as the unique deformation $\nabla_{W^{1,2}(\GA)} E(\Ga)$ satisfying :
$$D E(\Ga)(\Phi) \; = \; \langle\nabla_{W^{1,2}(\GA)} E(\Ga),\,\Phi\rangle_{W^{1,2}(\GA)}\,, \quad  \foralls \Phi\, \in\, T_\GA\Bb\, $$
where $D E(\Ga)(\Phi)$ is the directional derivative.

\subsection{Geometric Curves and Parameterizations}\label{repa}

For applications in computer vision, it is important that the developed method (e.g. a gradient descent flow to minimize an energy) only depends on the actual geometry of the planar curve, and not on its particular parametrization. We denote $[\GA] = \GA(\Circ)$ the geometric realization of the curve, i.e. the image of the parameterization in the plane.

If, for two curves $\GA_1,\GA_2 \in \Bb$ there exists a smooth invertible map $\phi:\Circ \to \Circ$ such that $\GA_2 = \GA_1 \circ \phi$, then $\GA_2$ is a reparameterization of $\GA_1$ and these parameterizations share the same image, i.e. $[\GA_1] = [\GA_2]$. This section shows, in some sense, the converse implication in the $BV^2$ framework, namely the existence of a reparameterization map between two curves sharing the same geometric realization. This result is important since it shows the geometric invariance of the developed Finsler gradient flow. 

Note however that this is clearly not possible without any constraint on the considered curve. For instance, there is no canonical parameterization of an eight-shaped curve in the plane. We will only consider injective curves $\GA \in \Bb$ satisfying the following additional property 
\begin{equation}\label{cond-derivative}
	0\notin \overline{\text{Conv}}(\GA'(s^+),\GA'(s^-)) \quad \foralls s\in \Circ\,.
\end{equation}
Here $\overline{\text{Conv}}$ denotes the closed convex envelope (a line segment) of the right and left limits $\GA'(s^+)$ and $\GA'(s^-))$ of the derivative of $\GA$ at $s$. We will show in the following that such a property gives a generalized definition of immersion for  $BV^2$-curves and implies that the support of the curve has no cusp points. 
\par We define the set of curves
\begin{equation}
	\Bb_0 = \{ \GA \in BV^2(\Circ,\Pl) 
	\,:\,
		\GA\;\mbox{is injective and satisfies}\;\eqref{cond-derivative}
	\}
\end{equation}
equipped with the $BV^2$-norm. 

Note that it is difficult to ensure that the iterates $\{\Gamma_k\}$ defined by \eqref{sequence} stay in $\Bb_0$, since $\Bb_0$ is not a linear space. As shown in Proposition \ref{openB0} below, $\Bb_0$ is an open set, so that one might need to use small step sizes $\tau_k$ to guarantee that $\Gamma_k \in \Bb_0$. This is however no acceptable, because it could contradict the constraints (2.8) and prevent the convergence of $\Gamma_k$ to a stationary point of $E$. This issue reflects the well known fact that during an evolution, a parametric curve can cross itself and become non-injective. 


We also note that, as pointed out in Definitions \ref{hilbert-struct} and \ref{BV2curves}, condition \eqref{cond-derivative} guarantees that the norms on $L^2(\Circ, \RR^2)$ and $L^2(\GA)$ and on $BV^2(\Circ, \RR^2)$ and $BV^2(\GA)$ are equivalent, so that the abstract setting described in Section \ref{F} is adapted to our case.

We first show that property \eqref{cond-derivative} implies local injectivity of the curve and that this local injectivity remains true in a neighborhood of the curve.

\begin{prop}\label{open_condition}   
Let $\Gamma_0 \in BV^2(\Circ,\Pl) $ and $t \in \Circ$ such that condition \eqref{cond-derivative} is satisfied. There exists $\varepsilon>0 \, ,  \gamma>0$ and $n \in \R^2$ a unit vector such that,  if $\|\Gamma_0 - \Gamma \|_{BV^2(\Circ,\Pl) } < \gamma$, then $\GA_0$ and $\GA$ are injective on $|s-t|<\varepsilon$, and 
$$|\langle \Gamma(s) -\GA(t) ,n \rangle| > \gamma|s-t|\quad \mbox{on} \quad |s - t| < \varepsilon\,.$$
\end{prop}

\begin{proof}
As $\GA_0'$ is a function of bounded variation, the left and right limits exist and are finite. Moreover, $\GA_0$ verifies \eqref{cond-derivative} at $t$  if and only if there exists a  unit vector $n$ such that $\langle \GA_0'(t^+), n \rangle > 0$ and $\langle \GA_0'(t^-), n \rangle > 0$. 
Let $\gamma_0 = \frac 1 2 \min\{\langle \GA_0'(t^-), n \rangle ,\langle \GA_0'(t^+), n \rangle\}$. 
By the fact that $\lim_{s \to t^-} \GA_0'(s)= \GA_0'(t^-)$ and $\lim_{s \to t^+} \GA_0'(s)= \GA_0'(t^+)$, there exists $\epsilon>0$
such that 
$$|\langle \Gamma_0(s) -\GA_0(t) ,n \rangle| > \gamma_0|s-t|\quad \mbox{if} \quad |s - t| < \varepsilon\,,$$
which proves the local injectivity for $\GA_0$.

Now, since $BV(\Circ,\RR^2)$ is continuously embedded in $L^\infty(\Circ,\RR^2)$, if
 $\|\Gamma_0 - \Gamma \|_{BV^2(\Circ,\Pl) } < \gamma_0/2$, then $\langle \GA'(t^+), n \rangle > \gamma_0/2$ and $\langle \GA'(t^-), n \rangle > \gamma_0/2$ which proves that $\GA$ verifies \eqref{cond-derivative}. Moreover we get 
$$|\langle \Gamma(s) -\GA(t) ,n \rangle| > \frac{\gamma_0}{2}|s-t|\quad \mbox{on} \quad |s - t| < \varepsilon\,,$$ 
which proves the local injectivity of $\GA$.
  The result ensues by setting  $\gamma = \gamma_0 / 2$.
\end{proof}

\begin{prop}\label{openB0}   
	$\Bb_0$ is an open set  of $\Bb=BV^2(\Circ, \RR^2)$. 
\end{prop}

\begin{proof}
If $\Lambda\in \Bb_0$ then 
\eq{
	m = \underset{s\in \Circ}{\rm{essmin}} \,| \Lambda'(s) |> 0\,.
}
Now, by Proposition \ref{open_condition}, if $$\|\GA - \Lambda\|_{BV(\Circ, \RR^2)}< m/4$$ then $\GA$ is locally injective and verifies \eqref{cond-derivative}. 
Moreover, as $\Circ$ is compact and \eqref{cond-derivative} is satisfied on $\Circ$, there exists $\varepsilon,\alpha>0$  such that
\begin{equation}\label{local-injective}
|\Lambda(s)-\Lambda(s')|\geq \alpha|s-s'|, \quad \forall\,s,s'\in\Circ\;\; \mbox{such that}\;\;|s-s'|\leq \varepsilon\,.
\end{equation}

Then, as $\Lambda$ is injective, if we take $\|\GA - \Lambda\|_{BV^2(\Circ,\RR^2)}< \beta(\Lambda)$ where 
$$\beta(\Lambda) = \frac{1}{2}\mbox{min}\,\left\{\alpha, \underset{s\in\Circ}{\inf}\;\underset{|s-s'|> \varepsilon}{\inf}\,\|\Lambda(s)-\Lambda(s')\|\right\}$$
then  $\GA$ is also globally injective.
\par Then
$$\left\{\GA\in \mathcal{B}_0\, \big| \;\|\GA - \Lambda\|_{BV^2(\Circ, \RR^2)}< {\rm{min}}\left\{\frac{m}{4}, \beta(\Lambda)\right\}\right\}\subset \Bb_0\, $$
which proves that $\Bb_0$ is an open set of $ BV^2(\Circ, \RR^2)$.
\end{proof}

The next proposition proves the existence of a reparameterization between two curves sharing the same image. 
\begin{prop}\label{reparameterization}({\bf Reparameterization})
For every  $\GA_1, \GA_2\in \Bb_0$ such that $[\GA_1]=[\GA_2]$,  there exists a homeomorphism  $\varphi\in BV^2(\Circ,\Circ)$ such that 
$$\GA_1=\GA_2\circ\varphi \,.$$
\end{prop}

\begin{proof}
For every $\GA_1, \GA_2\in\Bb_0$ we consider the arc-length parameterizations defined by
$$\varphi_{\GA_1}, \varphi_{\GA_2} : \Circ \rightarrow \Circ$$
$$\varphi_{\GA_1}(s) = \frac{1}{\mbox{Length}(\GA_1)}\int_{s_1}^s \,|\GA_1'(t)| \,dt \quad ,\;\varphi_{\GA_2}(s) = \frac{1}{\mbox{Length}(\GA_2)}\int_{s_2}^s \,|\GA_2'(t)|\,dt $$
with $s_1, s_2\in \Circ$ such that $\GA_1(s_1)=\GA_2(s_2)$.

Because of property~\eqref{cond-derivative} we can apply the inverse function theorem for Lipschitz functions (see Theorem 1 in~\cite{Clarke}) which allows to define  
$\varphi_{\GA_1}^{-1}, \varphi_{\GA_2}^{-1}\in BV^2(\Circ, \Circ)$. It follows that
$$(\GA_1 \circ \varphi_{\GA_1}^{-1} \circ \varphi_{\GA_2})(s) = \GA_2(s)\quad \foralls s\in \Circ.$$
\end{proof}

\subsection{Geometric Invariance}

For $BV^2$ curves, the geometric invariance of the developed methods should be understood as an invariance with respect to $BV^2$ reparameterizations. 

\begin{defn}
	Let $G_{BV^2}$ denote the set of homeomorphisms  $\varphi\in BV^2(\Circ,\Circ)$ such that $\varphi^{-1} \in BV^2(\Circ,\Circ)$. 
 Note that for every $\GA \in BV^2(\Circ, \RR^2)$ we have $\GA\circ \varphi \in BV(\Circ, \RR^2)$ for every $\varphi\in G_{BV^2}$. In fact, as every $BV^2$-function  is Lipschitz-continuous, by the chain-rule for $BV$-function, we get $\GA\circ \varphi \in BV(\Circ, \RR^2)$. Moreover, $(\GA\circ \varphi)'= \GA'(\varphi)\varphi' \in BV(\Circ, \RR^2)$ because $BV$ is a Banach algebra (one can check that $\GA'(\varphi)\in BV(\Circ, \RR^2)$ by performing the change of variables $t=\varphi(s)$ in the definition of total variation).   
\end{defn}

To ensure this invariance, we consider energies $E$ and penalties $R_\GA$ such that
$$
	E(\GA\circ\varphi) = E(\GA)\;\;, \quad R_{\GA\circ\varphi}(\Phi\circ\varphi) = R_{\GA}(\Phi)\quad\quad  \foralls \GA \in \Bb_0,\,\foralls \phi \in G_{BV^2} ,\,\foralls\Phi \in T_\GA\Bb\,.
$$
This implies that 
$$\nabla_{R_{\GA\circ\varphi}}E(\GA\circ\varphi)(\Phi\circ\varphi)= \nabla_{R_{\GA}}E(\GA)(\Phi)\circ\varphi$$
so that the descent scheme~\eqref{sequence} does not depend on the parameterization of $\GA$. Finally, as 
$$(\GA - \tau\nabla_{R_{\GA}} E(\GA))\circ \varphi = \GA\circ \varphi  - \tau\nabla_{R_{\GA\circ \varphi }} E(\GA\circ \varphi ), $$
for $\tau=\tau_k$, the descent step in~\eqref{sequence} is also invariant under reparameterization. 

This shows that the Finsler gradient flow can actually be defined over the quotient space
$\Bb/G_{BV^2}$. To avoid unnecessary technicalities, we decided not to use this framework and develop our analysis in the setting of the vector space $\Bb$.

Another consequence of this invariance is that, as long as the evolution~\eqref{sequence} is in $\Bb_0$, the flow does not depend on the choice of the parameterization. However, as already noted in Section~\ref{repa}, it might happen that the sequence leaves $\Bb_0$, in which case different choices of parameterizations of an initial geometric realization can lead to different evolutions.

\section{Piecewise Rigidity}\label{PWR}

This section defines a penalty $R_\GA$ that favors piecewise rigid $BV^2$ deformations of $\BVD$-curves.  For the sake of clarity we present the construction of this penalty in two steps. We first characterize in Section~\ref{GRD} $C^2$-global rigid deformations for  smooth curves. Then, in Section~\ref{BV2M}, we introduce a  penalty that favors  piecewise rigid $BV^2$ deformations for curves belonging to $\Bb$.

\subsection{Rigid Curve Deformations}\label{GRD}

A smooth curve evolution $\GA_t\in C^1(\RR, \Bb)$ reads
\eql{\label{eq-continuous-flow}
	\foralls t \in \RR, \quad \pd{\GA_t(s)}{t} = \Phi_t(s)
	\qwhereq
	\Phi_t \in T_{\GA(t)} \Bb\,.
}
 We further assume in this section that $\GA_t$ is a $C^2$ curve. This evolution is said to be globally rigid if it preserves the pairwise distances between points along the curves, i.e. 
\eql{\label{eq-rigid-flow}
	\foralls t \in \RR,\quad \foralls (s,s') \in \Circ \times \Circ, \quad
	|\GA_t(s)-\GA_t(s')|=|\GA_0(s)-\GA_0(s')|.
} 
The following proposition shows that the set of instantaneous motions $\Phi_t$ giving rise to a rigid evolution  is, at each time, a linear sub-space of dimension 3 of $T_{\GA_t} \Bb$.

\begin{prop}\label{rigid-ch}
	The evolution~\eqref{eq-continuous-flow} satisfies~\eqref{eq-rigid-flow} if and only if $\Phi_t \in \Rr_{\GA_t}$ for all $t \in \RR$,  where  
	\eql{\label{eq-rigid-vector-space}
		\Rr_\GA = \enscond{ \Phi \in T_\Ga \Bb }{
                        \exists a \in \RR, \exists b \in \R^2,\;\;
			\foralls s \in \Circ, \; \Phi(s) =  a \, \rot{\GA(s)}  + b
		}\,.
	}
\end{prop}
\begin{proof}
Recall that the group of distance preserving transformations on $\R^d$ is the Euclidean group $E(d) = \R^d \rtimes  O_d(\R)$ and that any element of $E(d)$ is uniquely defined by the image of $d+1$ points in $\R^d$ which are affinely independent. Therefore, provided that the curve $\GA$ has at least three non-collinear points, $\Phi_t$ is the restriction of $g_t \in E(d)$, a path on $E(d)$ which is uniquely defined. In addition, $g_t$ and $\Phi_t$ have the same smoothness. Thus the result follows from the fact that the Lie algebra of $\R^d \rtimes  O_d(\R)$ is $\R^d \rtimes A(d)$, where $A(d)$ denotes the set of antisymmetric matrices.
The degenerate cases such as when the curve is contained in a line or a point are similar.
\end{proof}

Note that for numerical simulations, one replaces the continuous PDE~\eqref{eq-continuous-flow} by a flow discretized at times $t_k = k \tau$ for some step size $\tau>0$ and $k \in \NN$, 
\eq{
	\GA_{k+1} = \GA_k + \tau \Phi_k
	\qwhereq
	\Phi_k \in T_{\GA_k} \Bb.
} 
This is for instance the case of a gradient flow such as~\eqref{sequence} where $\Phi_k = - \nabla_{R_{\GA_k}} E(\GA_k)$. In this discretized setting, imposing $\Phi_k \in \Rr_{\GA_k}$ only guarantees that rigidity~\eqref{eq-rigid-flow} holds approximately and for small enough times $t_k$. 

The following proposition describes this set of tangent fields in an intrinsic manner (using only derivatives along the curve $\GA$), and is pivotal to extend $\Rr_\GA$ to piecewise-rigid tangent fields.

\begin{prop}\label{charactrigid}
	For a $C^2$-curve $\GA$, one has $\Phi \in \Rr_\GA$ if and only if $\Phi$ is $C^2$ and satisfies $L_\GA(\Phi) = 0$ and $H_\GA(\Phi) = 0$, where $L_\GA$ and $H_\GA$ are the following linear operators
\begin{equation}\label{eqrigid}
 	L_\GA(\Phi) \;=\;   \frac{\d\Phi}{\dgs}\, \cdot\, \tgam  
	\qandq  
	H_\GA(\Phi) \;=\;  \frac{\d^2\Phi}{\dgs^2} \, \cdot\, \ngam \,.
\end{equation}
From a geometric point of view, $L_\GA(\Phi)$ takes into account the length changes and $H_\GA(\Phi)$ the curvature changes.
\end{prop}

\begin{proof}
Using the parameterization of $\Ga$, any such deformation $\Phi$ satisfies 
\begin{equation}
	\label{eq:rigidg}
	\exists\, !\; (a,b) \in \R\times\R^2, \quad \foralls s \in [0,1], \quad \Phi(s) \; = \;  a\rot{\Ga(s)}  + b \,.
\end{equation}
By differentiation with respect to \! $s$, this is equivalent to 
\eq{ 
	\exists\, !\; a \in \R,\quad \foralls s \in [0,1], \quad \frac{\d\Phi}{\d s}(s) \; = \; a|\GA'(s)|\; \ngam(s)
}
which can be rewritten as $\frac{\d\Phi}{\d\GA} (s) =  a \ngam(s) $ by differentiating with respect to \! the length element $\d\GA = \norg\,\d s$,
or simply as $ \frac{\d\Phi}{\d s}(s)  = a \ngam(s) $ 
by considering an arc-length parameterization. 
This is equivalent to 
\eq{
	\exists\, !\; a \in \R,\quad \foralls s \in [0,1], \quad 
	\begin{cases} 
	\displaystyle \;\; \frac{\d\Phi}{\dgs}\, \cdot\, \tgam(s) \; = \; 0 \vspace{2mm} \\
	\displaystyle \;\; \frac{\d\Phi}{\dgs}\, \cdot\, \ngam(s) \; = \; a
	\end{cases}
}
which is  equivalent to
\eq{
	\begin{cases}
		\displaystyle \;\; \frac{\d\Phi}{\dgs}\, \cdot\, \tgam \; = \; 0 \vspace{2mm} \\
		\displaystyle \;\; \frac{\d}{\dgs} \left( \frac{\d\Phi}{\dgs} \, \cdot\, \ngam \right) \; =\; 0
	\end{cases}
}
and, using that 
\eq{ 
	\frac{\d}{\dgs} \left( \frac{\d\Phi}{\dgs} \, \cdot\, \ngam \right) \; =\;  
	\frac{\d^2\Phi}{\dgs^2} \, \cdot\, \ngam \; - \frac{\d\Phi}{\dgs} \, \cdot\, \kappa_\Ga \,\tgam , 
}
where $\kappa_\Ga$ is the curvature of $\GA$, we obtain the desired characterization.
\end{proof}

\subsection{Piecewise rigid $BV^2$ deformations }\label{BV2M}

This section extends the globally rigid evolution considered in the previous section to piecewise-rigid evolution.
\par In the smooth case considered in the previous section, this corresponds to imposing that an instantaneous deformation $\Phi \in T_\GA \Bb$  satisfies~\eqref{eqrigid} piecewisely for possibly different pairs $(a,b)$ on each piece.
To generate a piecewise-smooth Finsler gradient $\Phi = \nabla_{R_\GA} E(\GA)$ (as defined in~\eqref{defgrad}) that is a piecewise rigid deformation, one should design a penalty $R_\GA$ that satisfies this property. This is equivalent to imposing  
$L_\GA(\Phi) = 0$ and $H_\GA(\Phi) = 0$ for all $s \in [0,1]$ except for a finite number of points (the articulations between the pieces). In particular, note that $L_\GA(\Phi)$ is undefined at these points, while $H_\GA(\Phi)$ is the sum of Dirac measures concentrated at the  articulation points (due to the variations of $a$). This suggests that, in the smooth case, we can favor piecewise rigidity by minimizing $\left\| H_\GA(\Phi) \right\|_{L^1(\Ga)}$ under the constraint $L_\GA(\Phi) = 0 \;\; a.e.$, so that  we control the jumps of the second derivative without setting in advance the articulation points. Note also that the minimization of the $L^1$-norm favors sparsity and, in contrast to the $L^2$-norm, it enables the emergence of Dirac measures.
\par In order to extend such an idea to the $BV^2$-framework we remind that 
$$\left\| H_\GA(\Phi) \right\|_{L^1(\Ga)}=TV_\GA\left(\frac{\d \Phi}{\d\GA}\cdot \ngam\right)\quad \forall \,\Phi\in C^{2}(\Circ,\RR^2),\, \GA \in C^2(\Circ, \RR^2)$$
which defines a suitable penalty in the $BV^2$-setting.
Moreover, since we are interested in piecewise rigid motions, we  deal with curves that could be not  $C^1$ at some points $s$. It is useful to introduce the following operators
\begin{align}\label{eq-operator-L}
	L_\GA^+(\Phi)(s) &=	\underset{\underset{t\in(s,s+\epsilon)}{t\rightarrow s}}{\lim} 
		\frac{\d\Phi}{\dgs}(t) \cdot\, \tgam(t)\,, \\
	L_\GA^-(\Phi)(s) &=\underset{\underset{t\in(s-\epsilon,s)}{t\rightarrow s}}{\lim} 
		\frac{\d\Phi}{\dgs}(t) \cdot\, \tgam(t)\,.
\end{align}
Of course if $\GA$ and $\Phi$ are $C^1$ at $s$ we have $L_\GA^+(\Phi)(s)=L_\GA^-(\Phi)(s)=L_\GA(\Phi)(s)$. 
The next definition introduces a penalty for piecewise rigid evolution in $\Bb$.

\begin{defn}[\BoxTitle{$\BVD$ Piecewise-rigid penalty}]\label{defn1}
	For $\Ga \in \Bb$ and $\Phi \in T_\GA\Bb = \BVDga$, we define
	\begin{equation}\label{eq-piecewise-rigid-penalty}	
		R_\Ga(\Phi) =
			TV_\GA\left(\frac{\d \Phi}{\d\GA}\cdot \ngam\right)
			 + \iota_{\Cc_\Ga}(\Phi)
	\end{equation}
where $\iota_{\Cc_\Ga}$ is the indicator  function of $\Cc_\Ga$ 
$$
\displaystyle{\iota_{\Cc_\Ga}(\Phi)=
\left\{
\begin{array} {ll}
0 &\text{if } \Phi \in \Cc_\Ga\\
+\infty &\text{otherwise }  
\end{array}\right.}
\,.$$
Note that~\eqref{TV} is the total variation of $f$ with respect to the measure $\d\GA$. We remind  that $TV_\GA(f)=|Df|(\Circ)$ for every $f\in L^1(\Circ,\RR^2)$.
\par The set $\Cc_\Ga$ is defined as follows
\begin{equation}\label{eq-constr-C-Ga}
	\Cc_\Ga = \enscond{\Phi\in T_\GA\Bb }{L^+_\GA(\Phi)=L^-_\GA(\Phi)=0 }\,.
\end{equation}
\end{defn}
In order to define the Finsler gradient  we consider a constraint on the normal component of the deformation field.

\begin{defn}[\BoxTitle{Deviation constraint}]\label{defn-dev-constr} For $\GA \in \Bb$, we define
\eql{\label{eq-dev-constr}
	\Ll_\GA =\enscond{
		 \Phi\in T_\Ga\Bb
	}{ 
		\normbig{ \Pi_\GA( \nablad E(\Ga) -\Phi )}_{W^{1,2}(\Ga)} \leq \rho \normbig{\Pi_\GA(\nablad E(\Ga))  }_{W^{1,2}(\Ga)}
	}\,.
}
\end{defn}

Here, $\rho\in (0,1)$ is called the rigidification parameter, and controls the deviation of the Finsler gradient from the $W^{1,2}$ gradient. $\Pi_\GA$ is the projector introduced in equation~\eqref{normal-proj}.
\par We point out that in the applications studied in this paper we consider an intrinsic energy $E$ (i.e., it does not depend on reparameterization). In this case  the $W^{1,2}$-gradient of $E$ is normal to the curve, so that  $\Ll_\GA$ satisfies condition~\eqref{cond-constraints} in the case of an intrinsic energy, and  it can be used to define a valid Finsler gradient.

 Using these specific instantiations for $R_\GA$ and $\Ll_\GA$, Definition~\ref{FinslerGrad} reads in this $\BVD$ framework as follows.

\begin{defn}[\BoxTitle{$\BVD$ Piecewise-rigid Finsler gradient}] We define
\begin{equation}\label{grad-rigid}
	\nabla_{R_\GA} E(\GA) \in \uargmin{\Phi \in \Ll_\GA}\;  
   R_\GA(\Phi).
\end{equation}
\end{defn}

The following result  ensures the existence of a Finsler gradient. To prove  it we consider the space $\Bb$ equipped with the weak* toplogy.

\begin{thm}\label{existence1} 
	Problem~\eqref{grad-rigid}  admits at least a solution.
	
\end{thm}

In order to prove the theorem, we need the following lemmas. They guarantee in particular the compactness of minimizing sequences with respect to the $BV^2$-weak* topology. The proof relies on the evaluation of a bilinear form which is degenerate if the curve is a circle, so we treat the case of the circle separately. 

\begin{lem}\label{lem-compacity}
Let  $\GA\in \mathcal{B}$ be an injective curve.  We suppose that $\GA$ is different from a circle. Then there exists a constant $C(\GA)$ depending on $\GA$ such that 
\begin{equation}\label{compactnessR-1}
	 	\norm{ \Phi }_{BV^{2}(\GA)} \leq 
		C(\GA)\left( (1+\rho) \|\Pi_\GA(\nabla_{W^{1,2}(\Ga)}E(\GA))\|_{W^{1,2}(\Ga)} + R_\GA(\Phi) \right) \quad \forall\,\Phi \in \Ll_\GA \cap \Cc_\GA
\end{equation}
where $\Pi_\GA$ is the operator defined in \eqref{normal-proj}.	
\end{lem}

\begin{proof} 
The proof is essentially based on   estimation~\eqref{bound-first-der} giving a bound for the $L^\infty$-norms of the deformation $\Phi$ and its first derivative. We also remark that, as $\Phi\in \Ll_\GA$, we have
\begin{equation}\label{l1-bound}
\norm{ \Pi_\GA(\Phi) }_{
L^{2}(\Ga)} \leq (1+\rho)\|\Pi_\GA(\nabla_{W^{1,2}(\Ga)}E(\GA))\|_{W^{1,2}(\Ga)}\,.
\end{equation}
In the following we denote by $l(\GA)$ the length of the curve $\GA$. 

{\bf Bound on the first derivative}. In this section  we prove the following   estimate for the $L^\infty$-norms of  $\frac{\d\Phi}{\d\GA}\cdot \ngam $ and $\Phi$:
\begin{equation}\label{bound-first-der}
	\mbox{max}\left\{
		\left\|\Phi\right\|_{L^{\infty}(\GA)}\,;\,
		\left\|\frac{\d \Phi}{\d\GA}\cdot\ngam\right\|_{L^{\infty}(\GA)}
	\right\}
	\leq C_0(\GA)\left((1+\rho)\|\Pi_\GA(\nabla_{W^{1,2}(\Ga)}E(\GA))\|_{W^{1,2}(\Ga)}+ R_\GA(\Phi)\right)\,
\end{equation}
where $C_0(\GA)$ depends on  $\GA$. 
\par  Let $s_0\in \Circ$, we can write
$$\frac{\d\Phi}{\d\GA}\cdot \ngam = u + a$$
where $u\in BV(\GA)$ such that $u(s_0)=0$  and $a=\frac{\d\Phi}{\d\GA}(s_0)\cdot \ngam(s_0)\in \R$. As $\Phi\in\Cc_\GA$ we have $L_\GA^+(\Phi)=L_\GA^-(\Phi)=0$, which implies
$$\frac{\d\Phi}{\d\GA}=\left(\frac{\d\Phi}{\d\GA}\cdot \ngam\right)\ngam$$
and
\begin{equation}\label{writing-phi}
\Phi(s)= \Phi(s_0) + a[\GA(s)-\GA(s_0)]^\bot+ \int_{s_0}^s u\ngam \d \GA(s) \,\quad \forall \, s\in \Circ\,.
\end{equation}
Now, by projecting on the normal to $\GA$, we can write 
\begin{equation}\label{decomp-xi}\Pi_\GA(\Phi)=\Pi_\GA(\Phi(s_0)+ a[\GA(s)-\GA(s_0)]^\bot)+\Pi_\GA\left(\int_{s_0}^s u\ngam \d \GA(s)\right)\,.
\end{equation}
In particular, by the properties of good representatives for  $BV$-functions of one variable (see~\cite{AFP} p. 19), we have 
$$
|u(s)|=|u(s)-u(s_0)|\leq TV_\GA(u)\quad \forall s\in \Circ
$$
which implies that
\begin{equation}\label{first-bv0}
\left\|\int_{s_0}^s u\ngam \d \GA(s)\right\|_{L^\infty(\GA)}\leq
l(\GA) TV_\GA(u)= l(\GA) R_\GA(\Phi)\,
 \end{equation}
 and
 \begin{equation}\label{first-bv0-l2}
\left\|\int_{s_0}^s u\ngam \d \GA(s)\right\|_{L^2(\GA)}\leq
l(\GA)^{3/2} R_\GA(\Phi)\,
 \end{equation}
Thus, by~\eqref{l1-bound},~\eqref{first-bv0-l2}, and~\eqref{decomp-xi} it follows that
\begin{equation}\label{bound-xi}
\|\Pi_\GA(\Phi(s_0)+ a[\GA(s)-\GA(s_0)]^\bot)\|_{L^2(\GA)}\leq (1+\rho)\|\Pi_\GA(\nabla_{W^{1,2}(\Ga)}E(\GA))\|_{W^{1,2}(\Ga)}+ l(\GA)^{3/2}R_\GA(\Phi)\,.
\end{equation}
\par We remark now that $\|\Pi_\GA(\Phi(s_0)+ a[\GA(s)-\GA(s_0)]^\bot)\|_{L^2(\GA)}^2$ can be written as   
\begin{equation}\label{bilinear}
\|\Pi_\GA(\Phi(s_0)+a[\GA(s)-\GA(s_0)]^\bot)\|_{L^2(\GA)}^2=(|\Phi(s_0)|, a)\cdot A\left(\frac{\Phi(s_0)}{|\Phi(s_0)|},s_0\right)
\begin{pmatrix}
|\Phi(s_0)|\\
a\\
\end{pmatrix}
\end{equation}
where, for any $e\in \Circ\subset \RR^2$ and $s_0\in \Circ$, the matrix $A(e,s_0)$ is defined by
\begin{equation}
\begin{pmatrix}
\int_{\Circ} \left(e\cdot\ngam\right)^2\d \GA(s)&\int_{\Circ}\left([\GA(s)-\GA(s_0)]^\bot\cdot\ngam\right)\left(e\cdot\ngam\right) \d \GA(s)\\
\int_{\Circ}\left([\GA(s)-\GA(s_0)]^\bot\cdot\ngam\right)\left(e\cdot\ngam\right) \d \GA(s)&\int_{\Circ} \left([\GA(s)-\GA(s_0)]^\bot\cdot\ngam\right)^2\d \GA(s)\\
\end{pmatrix}
.
\end{equation}

\par Note that the bilinear form defined by $A(e,s_0)$ is degenerate if and only if the determinant of $A(e,s_0)$ is zero which means that there exists  $\alpha\in\R$ such that 
$(e- \alpha[\GA(s)-\GA(s_0)]^\bot)\cdot \ngam=0$ for every $s\in \Circ$. Note that this  implies that $\GA$ is either a circle or a line. Now, as we work with closed injective curves $\GA$ is different from a line. Then, because of the hypothesis on $\GA$, we get that for every $s_0\in \Circ$ the   bilinear form associated with $A(e,s_0)$  is not degenerate.

\par In particular the determinant of $A$ is positive which means that the bilinear form  is positive-definite.  This implies  that its smallest eigenvalue is positive and in particular, by a straightforward calculation, it can be written as $\lambda\left(e,s_0\right)$ where   $\lambda:\Circ\rightarrow \R$ is a positive continuous function. Then, we have
\begin{equation}\label{case1}
\underset{e,s_0\in \Circ}{\inf}\,\lambda(e, s_0)(|\Phi(s_0)|^2+a^2) \leq \lambda\left(\frac{\Phi(s_0)}{|\Phi(s_0)|},s_0\right)(|\Phi(s_0)|^2+a^2)\leq \|\Pi_\GA(\Phi(s_0)+ a\GA(s)^\bot)\|_{L^2(\GA)}^2\,
\end{equation}
where the infimum of $\lambda$ on $\Circ\times\Circ$ is a positive constant depending only on $\GA$ and denoted by $\lambda_\GA$. 
\par The previous relationship and~\eqref{bound-xi} prove that, for every $s_0\in \Circ$, we have 
\begin{equation}\label{norme-infty}
	\mbox{max}\left\{|\Phi(s_0)|,a\right\}
	\leq C_0(\GA)\left((1+\rho)\|\Pi_\GA(\nabla_{W^{1,2}(\Ga)}E(\GA))\|_{W^{1,2}(\Ga)}+ R_\GA(\Phi)\right)\,
\end{equation}
where $C_0(\GA)=\mbox{max}\{1/\lambda_\GA, l(\GA)^{3/2}/\lambda_\GA\}$ depends only on $\GA$.

Then, because of the arbitrary choice of $s_0$ and the definition of $a$ ($a=\frac{\d\Phi}{\d\GA}(s_0)\cdot \ngam(s_0)$),~\eqref{norme-infty}      implies~\eqref{bound-first-der}. In particular~\eqref{bound-first-der}  gives a bound for the $W^{1,1}(\GA)$-norm of $\Phi$. 

{\bf Bound on the second variation}. 
We have
$$ TV^2_\GA(\Phi) = TV_\GA\left(\frac{\d\Phi}{\d\GA}\right)\,.$$ 
Now,  $\frac{\d \Phi}{\d\GA}=\left(\frac{\d \Phi}{\d\GA}\cdot \ngam\right)\ngam\in BV(\GA)$ and, by the generalization of the  product rule  to  $BV$-functions (see Theorem 3.96, Example 3.97, and Remark 3.98 in~\cite{AFP}), we get 
$$
TV_\GA\left(\frac{\d \Phi}{\d\GA}\right)\leq 2\left(TV_\GA\left(\frac{\d \Phi}{\d\GA}\cdot \ngam\right) + \left\|\frac{\d \Phi}{\d\GA}\cdot\ngam\right\|_{L^\infty(\GA)}TV_\GA(\ngam)\right) \,.
$$
The constant 2 in the previous inequality comes from the calculation of the total variation on the intersection of the jump sets of $\left(\frac{\d \Phi}{\d\GA}\cdot \ngam\right)$ and $\ngam$ (see Example 3.97 in~\cite{AFP}).
Note also that $TV_\GA(\ngam)=|D \ngam|(\Circ)$. 
\par Then, by~\eqref{deriv-norm} and~\eqref{curvature}, we get 
\begin{equation}\label{estimate-second-var0}
\begin{array}{ll}
\displaystyle{TV_\GA\left(\frac{\d \Phi}{\d\GA}\right)}&\displaystyle{\leq 2\left(TV_\GA\left(\frac{\d \Phi}{\d\GA}\cdot \ngam\right) + |\rm{curv}_\GA|(\Circ)\left\| \frac{\d \Phi}{\d\GA}\cdot\ngam\right\|_{L^\infty(\GA)}\right)}\\
&\displaystyle{\leq  2\left(TV_\GA\left(\frac{\d \Phi}{\d\GA}\cdot \ngam\right) + |D^2\GA|(\Circ)\left\| \frac{\d \Phi}{\d\GA}\cdot\ngam\right\|_{L^\infty(\GA)}\right)}\,\\
\end{array}
\end{equation}
which implies that 
\begin{equation}\label{estimate-second-var}
TV_\GA^2\left(\Phi\right)\leq C_1(\GA)\left(R_\GA(\Phi) + \left\|\frac{\d \Phi}{\d\GA}\cdot\ngam\right\|_{L^\infty(\GA)}\right)\,.\\
\end{equation}
where $C_1(\GA)$ is a constant depending on $\GA$.

The Lemma follows from~\eqref{bound-first-der} and~\eqref{estimate-second-var}.
\end{proof}

The next lemma gives a similar result in the case where $\GA$ is a circle.

\begin{lem}\label{lem-compacity-circle}
Let  $\GA\in \mathcal{B}$ be a circle with radius $r$.  Then there exists a constant $C(r)$ depending on $r$ such that 
	\begin{equation}\label{compactnessR}
	 	\norm{ \Phi }_{BV^2(\GA)} \leq 
		C(r)\left( (1+\rho) \|\Pi_\GA(\nabla_{W^{1,2}(\Ga)}E(\GA))\|_{W^{1,2}(\Ga)} +
R_\GA(\Phi)\right)\, 
	\end{equation}
	for every $\Phi \in \Ll_\GA \cap \Cc_\GA$ such that $\Phi(s_0)\cdot\tgam(s_0) =0$ for some $s_0\in\Circ$.
\end{lem}

\begin{proof}
 The proof is based on the same arguments used to prove the previous lemma. We denote by $r$ the radius of the circle.
\par As $\Phi(s_0)\cdot\tgam(s_0)=0$, 
by the properties of good representatives for  $BV$-functions of one variable (see~\cite{AFP} p. 19), we have 
\begin{equation}\label{estim00}
|\Phi\cdot\tgam|=|\Phi(s)\cdot\tgam(s)-\Phi(s_0)\cdot\tgam(s_0)|\leq TV_\GA(\Phi\cdot\tgam)\quad \forall s\in \Circ\,.
\end{equation}
Now, as $L_\GA^+(\Phi)=L_\GA^-(\Phi)=0$ and the curvature is equal to $1/r$ at each point, we get 
$$\frac{\d (\Phi\cdot\tgam)}{\d\GA} = \Phi\cdot\frac{\ngam}{r}$$
and from~\eqref{estim00} it follows 
\begin{equation}\label{bound-tang} \|\Phi\cdot\tgam\|_{L^{\infty}(\GA)}\leq \frac{ \|\Phi\cdot\ngam\|_{L^{1}(\GA)}}{r}\,.
\end{equation}
\par Now, as $\Phi\in \Ll_\GA$, we have  
\begin{equation}\label{bound-norm}
\|\Phi\cdot\ngam\|_{L^{2}(\GA)}\leq (1+\rho) \|\Pi_\GA(\nabla_{W^{1,2}(\Ga)}E(\GA))\|_{W^{1,2}(\Ga)}
\end{equation}
and, from~\eqref{bound-tang} and~\eqref{bound-norm}, it follows
\begin{equation}\label{estim0}
 \|\Phi\|_{L^{1}(\GA)}\leq \sqrt{2\pi r}(2\pi +1)(1+\rho) \|\Pi_\GA(\nabla_{W^{1,2}(\Ga)}E(\GA))\|_{W^{1,2}(\Ga)}\,.
 \end{equation}
\par  Concerning the first derivative we remark that, as $\Phi$ is periodic, the mean value of its first derivative is equal to zero. Then, by Poincar\'e's inequality (see Theorem 3.44 in~\cite{AFP}), we have 
\begin{equation}\label{estim1} \left\|\frac{\d \Phi}{\d\GA}\right\|_{L^{1}(\GA)}\leq C_0(r) TV_\GA\left(\frac{\d \Phi}{\d\GA}\right) 
\end{equation}
where $C_0(r)$ is a constant depending on $r$. Moreover, by integrating by parts the integrals of the definition of second variation, we get 
\begin{equation}\label{estim2}
TV^2_\GA(\Phi) =  TV_\GA\left(\frac{\d\Phi}{\d\GA}\right)\,.
\end{equation}
So, in order to prove the lemma it suffices to prove a bound for  $TV_\GA\left(\frac{\d\Phi}{\d\GA}\right)$. 
\par Now, as $\frac{\d \Phi}{\d\GA}=\left(\frac{\d \Phi}{\d\GA}\cdot \ngam\right)\ngam$,  by the generalization of the  product rule  to  $BV$-functions (see Theorem 3.96, Example 3.97, and Remark 3.98 in~\cite{AFP}), we get 
\begin{equation}\label{estim30}
TV_\GA\left(\frac{\d \Phi}{\d\GA}\right)\leq \left(1+\frac{1}{r}\right) TV_\GA\left(\frac{\d \Phi}{\d\GA}\cdot \ngam\right)= \left(1+\frac{1}{r}\right) R_\GA(\Phi) \,,
\end{equation}
where we used the fact that $\ngam$ has no jumps (see Example 3.97 in~\cite{AFP}).

The lemma follows from~\eqref{estim0},~\eqref{estim1},~\eqref{estim2}, and~\eqref{estim30}.
\end{proof}

We can now prove Theorem~\ref{existence1}.

\begin{proof} The proof is based on Lemma~\ref{lem-compacity} and Lemma~\ref{lem-compacity-circle}, so we distinguish two cases: $\GA$ is a circle and it is not.
\par We suppose that $\GA$ is different from a circle. Let $\{\Phi_h\}\subset \Ll_\GA \cap \Cc_\GA$ be a minimizing sequence of  $R_\GA$. We can also suppose $\underset{h}{\sup}\,	R_\GA(\Phi_h)< +\infty$.
From Lemma~\ref{lem-compacity} it follows that 
$$
\underset{h}{\sup}\,	\norm{ \Phi_h }_{BV^{2}(\GA)} \leq 
		C(\GA)\left( (1+\rho) \|\Pi_\GA(\nabla_{W^{1,2}(\Ga)}E(\GA))\|_{W^{1,2}(\Ga)} +
\underset{h}{\sup}\,R_\GA(\Phi_h)\right)\,
$$   
where $C(\GA)$ depends only on $\GA$.  This gives a uniform bound for the $BV^2(\GA)$-norms of $\Phi_h$ and  implies that  $\{\Phi_h\}$  converges (up to a subsequence) toward some  $\Phi \in BV^2(\GA)$ with respect to the $BV^2(\GA)$-weak* topology (see Theorem 3.23 in~\cite{AFP}). 
\par In particular $\Phi_h \rightarrow \Phi$  in $W^{1,1}(\GA)$ which proves that  $\Phi \in \Cc_\GA $, and, by the lower semi-continuity of the $L^2$-norm, we also get  $\Phi \in \Ll_\GA$.
\par Now, as $R_\GA$ is lower semi-continuous with respect to the $BV^2(\GA)$-weak* topology, the theorem ensues.
\par In the case where $\GA$ is a circle with radius $r$, for every minimizing sequence $\{\Phi_h\}\subset \Ll_\GA \cap \Cc_\GA$, we consider the sequence 
\begin{equation}\label{psi-proof}
\Psi_h(s)=\Phi_h(s) - (\Phi_h(s_0)\cdot \tgam(s_0))\tgam(s) \,
\end{equation}
for some $s_0\in \Circ$.
We remark that $\{\Psi_h \} \subset  \Ll_\GA$. 
 Moreover 
\begin{equation}\label{phi-psi}
\frac{\d\Psi_h}{\d\GA}(s)=\frac{\d\Phi_h}{\d\GA}(s)-\left(\frac{\Phi_h(s_0)\cdot \tgam(s_0)}{r}\right)\ngam(s)
\end{equation}
which implies that, for every $h$,  $\Psi_h\in \Cc_\GA$ and
\begin{equation}\label{equal-TV}
R_\GA(\Psi_h)=R_\GA(\Phi_h)\,.
\end{equation}
Then the sequence $\{\Psi_h\}$ is a minimizing sequence of Problem~\eqref{grad-rigid} such that $\Psi(s_0)\cdot\tgam(s_0)=0$. We can also suppose $\underset{h}{\sup}\,	R_\GA(\Psi_h)< +\infty$.

Then, by Lemma~\ref{lem-compacity-circle} we get
$$	 \underset{h}{\sup}\,	\norm{ \Psi_h }_{BV^{2}(\GA)} \leq 
		C(r)\left( (1+\rho) \|\Pi_\GA(\nabla_{W^{1,2}(\Ga)}E(\GA))\|_{W^{1,2}(\Ga)} +
\underset{h}{\sup}\,R_\GA(\Psi_h)\right)\,
$$ 
where $C(r)$ depends only on $r$.

This proves a uniform bound for  $\|\Psi_h\|_{BV^2(\GA)}$ which implies that the minimizing sequence $\{\Psi_h\}$ converges (up to a subsequence) with respect to the $BV^2$-weak* topology.  Then we can conclude as in the previous case.
\end{proof}

We point out that, as showed in the previous proof, when $\GA$ is a circle the Finsler gradient is defined up to a tangential translation. This was actually expected because such a tangential translation is a rotation of the circle.
 
We have defined a penalty for piecewise rigid $BV^2$ deformations for curves belonging to $\Bb$. In the next section we use the Finsler descent method with respect to such a penalty to solve curve matching problems.

\section{Application to Curve Matching}\label{CM}
	
This section shows an application of the Finsler descent method  to the curve matching problem. 

\subsection{The Curve Matching Problem}

Given two curves $\Ga_0$ and $\La$ in $\Bb$, the curve matching problem (also known as the registration problem) seeks for an (exact or approximate) bijection between their geometric realizations $[\GA_0]$ and $[\La]$ (as defined in Section~\ref{repa}). One thus looks for a matching (or correspondence) $f : [\GA_0] \rightarrow \RR^2$ such that $f( [\Ga_0] )$ is equal or close to $[\La]$.

There exist a variety of algorithms to compute a matching with desirable properties, that are reviewed in Section~\ref{sec-previous-works}. A simple class of methods consists in minimizing an intrinsic  energy $E(\GA)$ (i.e., $E$ only depends on $[\GA]$), and to track the points of the curve, thus establishing a matching  during the minimization flow. We suppose that $E(\GA)>0$ if $[\GA] \neq [\La]$ and $E(\La)=0$, so that the set of global minimizers of $E$ is exactly $[\La]$. This is for instance the case if $E(\GA)$ is a distance between $[\GA]$ and $[\La]$. A gradient descent method (such as~\eqref{sequence}) defines a set of iterates $\GA_k$, so that $\GA_0$ is the curve to be matched to $\La$. The iterates $\GA_k$ (or at least a sub-sequence) converge to $\Ga_\infty$, and the matching is simply defined as 
\eq{
	\foralls s \in \Circ, \quad f(\GA_0(s)) = \GA_\infty(s).
}
If the descent method succeeds in finding a global minimizer of $E$, then $f$ is an exact matching, 
i.e. $f([\GA_0]) = [\La]$. 
This is however not always the case, and the iterates $\GA_k$ can converge to a local minimum. It is thus important to define a suitable notion of gradient to improve the performance of the method. The next sections describe the use of the Finsler gradient to produce piecewise rigid matching.

\subsection{Matching Energy}

The matching accuracy  depends on  the energy and on the kind of descent used to define the flow. 
In this paper we are interested in studying the Finsler descent method rather than designing novel energies. For the numerical examples, we consider an energy based on reproducing kernel Hilbert space (r.k.h.s.) theory~\cite{rkhs, rkhs2}.  These energies  have been introduced for curve matching in~\cite{currents-matching, Glaunes-matching}. For an overview on other types of energies we refer the reader to  the bibliography presented in Section~\ref{sec-previous-works}.

We consider a positive-definite kernel $k$ in the sense of the r.k.h.s theory~\cite{rkhs,rkhs2}. Following~\cite{currents-matching}, we  define a distance between curves as
\begin{equation}\label{DistDef}
	{\rm dist}(\Ga,\La)^2 = Z(\Ga,\Ga) + Z(\La,\La) - 2 Z(\Ga,\La)\;, \quad 
	\foralls \Ga, \La \in \Bb
\end{equation}
where 
\begin{equation}\label{int-H}
  	Z(\Ga,\La) =  \int_{\Circ} \int_{\Circ} \n_{\Ga}(s) \cdot \n_{\La}(t) \; 
	k\left( \Ga(s) , \La(t) \right)\;\d \Ga(s) \d \La(t)\,.
\end{equation}
As the kernel $k$ is positive-definite in the sense of r.k.h.s. theory, it can be shown that $\rm{dist}$ defined in~\eqref{DistDef} is a distance between the geometric realizations $([\Ga], [\La])$ (up to change in orientation) of the curves. In our numerical tests, we define $k$ as a sum of two Gaussian kernels with standard deviation $\sigma>0$ and $\delta>0$
\begin{equation}\label{kernel}
	k(v, w) =  e^{-\frac{\norm{v-w}^2}{2\si^2}} + e^{-\frac{\norm{v-w}^2}{2\delta^2} }, 
	\quad \quad \foralls v,\,w\in \RR^2, 
\end{equation}
which can be shown to be a positive-definite kernel. We use a sum of Gaussian kernels to better capture features at different scales in the curves to be matched. This has been shown to be quite efficient in practice in a different context in~\cite{FX-ref}.  This energy takes into account the orientation of   the normals along the shape in order to stress the difference between the interior and the exterior of closed shapes. Remark that, to obtain interesting numerical results, both $\GA$ and $\La$ have to be parameterized with the same orientation (clockwise or counter-clockwise). 



Given a target curve $\La\in \Bb$, we consider the following energy 
\begin{equation}\label{energy}
	E : \Bb \rightarrow \RR, \; \quad E(\GA) = \frac{1}{2} {\rm dist}(\GA,\La)^2 \,.
\end{equation} 
Remark that, as $\rm{dist}$ is a distance then $[\La]$ is equal to the set of global minimizers of $E$.

We consider $W^{1,2}(\Circ,\RR^2)$ as ambient space, so that  we have 
\begin{equation}\label{grad-sob}
\nabla_{W^{1,2}} E = K \nabla_{L^2} E\,,
\end{equation}
where  $K$ denotes the inverse of the isomorphism between $W^{1,2}$ and its dual. Then, it suffices to compute the $L^2 $-gradient.

The gradient of $E$ at $\GA$ with respect to  $L^2(\GA)$-topology is given by the following proposition.

\begin{prop}
The gradient of $E$ at $\GA$ with respect to the $L^2(\GA)$ scalar product is given by
\begin{equation}\label{gradient-cont-l2}
	\nabla_{L^{2}(\Ga)} E(\GA)(s) = 
	\n_\GA(s) \Big[\int_{\Circ} 
		\n_\GA(t) \cdot \nabla_1 k(\GA(s),\GA(t)) )\d \GA(t)  - \int_{\Circ} 
		\n_{\La}(t) \cdot \nabla_1 k(\GA(s),\La(t)) 
		\d \La(t) \Big] 
\end{equation}
for every $s\in \Circ$, where $\nabla_1 k$ represents the derivative with respect to the first variable.

For every deformation $\Phi$, the $L^2$ gradient of $E$ at $\GA$ satisfies 
\eq{
\begin{array}{ll}
 \langle \nabla_{L^2(\GA)} E(\GA), \Phi\rangle_{L^2(\GA)}= &\displaystyle{\int_{\Circ}  \n_\GA(s) \cdot \Phi(s) \int_{\Circ} \n_\GA(t) \cdot \nabla_1 k(\GA(s),\GA(t)) d\GA(t)\, d\GA(s)}\\
 & \displaystyle{- \int_{\Circ}  \n_\GA(s) \cdot \Phi(s) \int_{\Circ} \n_{\La}(t) \cdot \nabla_1 k(\GA(s),\La(t)) \d \La(t)\,  \,d\GA(s)\, .}\\
 \end{array}
}

\end{prop}

\begin{proof} 
In order to prove~\eqref{gradient-cont-l2} we calculate the gradient for $Z(\GA,\La)$ with respect to $\GA$. We rewrite $Z$ as 
\eq{
	Z(\GA,\La) =  \int_{\Circ} \int_{\Circ} \GA'(s) \cdot \La'(t) \; k\left( \GA(s) , \La(t) \right)\; \d t\; \d s\,,
}
and we consider a smooth variation of the curve $\GA$, denoted by $\delta \GA$. Then, for $h$ small, we have
\eq{
	\begin{array}{ll}
	\displaystyle{I(h)=\frac{Z(\GA+ h\delta\GA, \La)- Z(\GA,\La)}{h}} = 
	&\displaystyle{    \int_{\Circ}\int_{\Circ} (\GA'(s) \cdot  \La'(t))(\nabla_1 k(\GA(s),\La(t))\cdot \delta \GA(s)) \, \d t\,  \,\d s }\\
	& \displaystyle{+ \int_{\Circ} \int_{\Circ} \delta\GA'(s) \cdot \La'(t) \; k\left( \GA(s) , \La(t) \right)\; \d t\; \d s\ + o(h)}\\
	\end{array}
}
and integrating by parts we obtain
\eq{
	\begin{array}{ll}
	\displaystyle{I(h)}&=\displaystyle{  \int_{\Circ}\int_{\Circ} (\GA'(s) \cdot  \La'(t))(\nabla_1 k(\GA(s),\La(t))\cdot \delta \GA(s)) \, \d t\,  \,\d s }\\
	& \displaystyle{-\int_{\Circ}\int_{\Circ} (\delta\GA(s) \cdot  \La'(t))(\nabla_1 k(\GA(s),\La(t))\cdot  \GA'(s)) \, \d t\,  \,\d s + o(h)}\\
	\end{array}
}
which can be written as
\begin{equation}\label{final-derivation} 
	I(h)=\displaystyle{\int_{\Circ}\int_{\Circ}[ \nabla_1 k(\GA(s),\La(t))^t(\delta\GA(s) \otimes  
	\GA'(s) - \GA'(s)\otimes\delta\GA(s)  )  \La'(t) ]\, \d t\,  \,\d s + o(h)}
\end{equation}

\eq{
	\qwhereq
	v \otimes w= v w^t, \quad\quad\foralls v, w \in \RR^2.
}
Now, writing $\delta \GA(s)$ with respect to the basis $\{\tgam(s), \ngam(s)\}$ and reminding that $\GA'(s)=|\GA'(s)|\tgam(s)$, we can show that the matrix
\eq{
	M(s)=\delta\GA(s) \otimes  \GA'(s) -\GA'(s)\otimes\delta\GA(s) = |\GA'(s)|(\delta\GA(s)\cdot\ngam(s)) (\ngam(s) \otimes  \tgam(s) -\tgam\otimes\ngam(s)) \\
}
acts as 
\begin{equation}\label{final-derivation2}
	 M(s)(v) = - |\GA'(s)|(\delta\GA(s)\cdot\ngam(s)) \rot{v}, \quad \foralls v\in \RR^2.
\end{equation}
Then, by~\eqref{final-derivation} and~\eqref{final-derivation2}, we obtain
\eq{
\begin{array}{ll}
\displaystyle{I(h)}&=\displaystyle{ - \int_{\Circ}\int_{\Circ} (\delta\GA(s)\cdot\ngam(s)) ( \nabla_1 k(\GA(s),\La(t))\cdot   \rot{\La'(t)})\, \d t\,  \,|\GA'(s)|\d s + o(h)}\\
\end{array}.
}
Finally, as $h\rightarrow 0$, we obtain the $L^2(\GA)$-gradient of $Z(\GA,\La)$  is given by
\eq{
	- \n_{\GA}(s)\int_{\Circ} \n_{\La}(t) \cdot \nabla_1 k(\GA(s),\La(t)) \d \La(t)
}
that represents  the second term in~\eqref{gradient-cont-l2}. For the first term we need to apply the same argument to calculate the gradient of $Z(\GA,\GA)$.
\end{proof}

\subsection{Matching Flow}
In this section, we use $H = W^{1,2}(\Circ,\R^2)$.
In order to minimize $E$ on $\Bb$ we consider the scheme~\eqref{sequence}, that defines $\{\GA_k\}$ for $k>0$ as 
\begin{equation}\label{descent}
	\GA_{k+1} = \GA_k - \tau_k \nabla_{R_{\GA_k}} E(\GA_k ) 
\end{equation}
where $\GA_0$ is the input curve to be matched to $\La$, 
 $\nabla_{R_{\GA_k}} E$ is defined by~\eqref{grad-rigid} (using $H = W^{1,2}(\Circ,\R^2)$) and  $\tau_k > 0$ is a step size, that satisfies the Wolfe rule~\eqref{Wolfe}. According to Theorem \ref{convergence}, following proposition proves the convergence of the method:

\begin{prop} 
The $W^{1,2}-$gradient of the energy functional $E$ is $W^{1,2}-$ Lipschitz on every set of curves of bounded length.
\end{prop}

\begin{proof}
We remark that we choose $W^{1,2}(\Circ,\RR^2)$ as ambient space and, moreover, we have 
\begin{equation}\label{grad-sob}
\nabla_{W^{1,2}} E = K \nabla_{L^2} E\,,
\end{equation}
where we have denoted $K$ the inverse of the isomorphism between $W^{1,2}$ and its dual. Then, it suffices to prove the proposition for the $L^2 $-gradient. For the sake of clarity, we separate the proof in several steps.

{\bf Continuity of the energy and the gradient.} 
By the dominated convergence theorem,  $E$ is continuous on  $W^{1,2}(\Circ,\R^2)$.
Note that
\begin{align*}
	 \langle \nabla_{L^2(\GA)} E(\GA), \Phi\rangle_{L^2(\GA)} = &\int_{\Circ}  \rot{\GA'(s)} \cdot \Phi(s) \int_{\Circ}  \rot{\GA'(t)} \cdot \nabla_1 k(\GA(s),\GA(t)) )\d t\, \d s\\
	  - &\int_{\Circ} \rot{\GA'(s)} \cdot \Phi(s) \int_{\Circ}  \rot{\La'(t)} \cdot \nabla_1 k(\GA(s),\La(t)) \d t\,  \,\d s
\end{align*}
where $\rot{(x,y)}=(-y,x)$ for every $(x,y)\in \RR^2$.

Then $E$ is  non-negative  and $C^1$ with respect to the $W^{1,2}(\Circ, \RR^2)$-ambient topology.\vspace{0.2cm} 

{\bf Condition \eqref{grad-lip}.} 
We detail the proof for the term of the gradient depending on both $\GA$ and $\La$. For the other term the proof is similar. For every couple of curves $(\Ga,\La)$, we introduce the following function
\eq{
	\Ii(\GA,\La)(s) = \int_{\Circ} \n_{\La}(t) \cdot \nabla_1 k(\GA(s),\La(t))\; \d \La(t)=\int_{\Circ} \rot{\La'(t)} \cdot \nabla_1 k(\GA(s),\La(t))\; \d t\,.
}

It suffices just to prove that there exists $L>0$ such that
\eq{
	\norm{ \rot{\GA'_1} \Ii(\GA_1,\La) -  \rot{\GA'_2} \Ii(\GA_2,\La) }_{L^2(\Circ,\RR^2)} 
	\leq L \norm{\GA_1-\GA_2}_{W^{1,2}(\Circ,\RR^2)} 
}
for every couple of curves $(\GA_1, \GA_2) \in BV^2(\Circ,\RR^2)$. We have
\eq{
	\norm{ \rot{\GA'_1} \Ii(\GA_1,\La) -  \rot{\GA'_2} \Ii(\GA_2,\La) }_{L^2(\Circ,\RR^2)} = 
	\norm{ \GA'_1\Ii(\GA_1,\La) - \GA'_2\Ii(\GA_2,\La) }_{L^2(\Circ,\RR^2)}
}
and
\begin{align}\label{lip1}
	\|\GA'_1\Ii(\GA_1,\La) - \GA'_2\Ii(\GA_2,\La)\|_{L^2(\Circ,\RR^2)} \leq &  \|\GA'_1\Ii(\GA_1,\La) - \GA'_1\Ii(\GA_2,\La)\|_{L^2(\Circ,\RR^2)}  \\
	+ & \|\GA'_1\Ii(\GA_2,\La) - \GA'_2\Ii(\GA_2,\La)\|_{L^2(\Circ,\RR^2)}\,.
\end{align}

Note that
\begin{equation}\label{boundI}
	\norm{\Ii(\GA,\La)}_{L^\infty(\Circ,\RR^2)}\leq  
	\alpha \norm{\La'}_{L^1(\Circ,\RR^2)}\,,
\end{equation}
where $\alpha = \sup_{x,y \in \R^2} |\nabla_1 k(x,y)|$.
Now, we have
\begin{align*}
	& \|\GA'_1[\Ii(\GA_1,\La)-\Ii(\GA_2,\La)]\|^2_{L^2(\Circ,\RR^2)} \leq \norm{\GA'_1}_{L^1(\Circ,\RR^2)}^2 \|\Ii(\GA_1,\La)-\Ii(\GA_2,\La)\|^2_{L^\infty(\Circ,\RR^2)}  \\
	& \qquad \leq  \norm{\GA'_1}_{L^1(\Circ,\RR^2)}^2 \norm{\La'}^2_{L^1(\Circ,\RR^2)} \sup_{s\in \Circ}  \int_{\Circ}|\nabla_1k(\GA_1(s),\La(t))-\nabla_1k(\GA_2(s),\La(t))|^2\;  \d t \\
	&\qquad \leq  \norm{\GA'_1}_{L^1(\Circ,\RR^2)}^2 \norm{\La'}^2_{L^1(\Circ,\RR^2)} \frac{\sigma^2+\delta^2}{\sigma^2\delta^2}\sup_{s\in \Circ} | \GA_1(s) - \GA_2(s)|^2\,,
\end{align*}
where we used  the fact that $r e^{-r^2}$ is 1-Lipschitz continuous (given by a straightforward derivative calculation). Then, as $W^{1,2}(\Circ,\RR^2)$ is continuously embedded in $L^\infty(\Circ,\RR^2)$, we get
\begin{equation}\label{lip2}
	\norm{ \GA'_1[\Ii(\GA_1,\La)-\Ii(\GA_2,\La)] }^2_{L^2(\Circ,\RR^2)} 
	\leq 
	C_1 \norm{\La'}^2_{L^1(\Circ,\RR^2)} \norm{\GA_1-\GA_2}^2_{W^{1,2}(\Circ,\RR^2)} 
\end{equation}
where $C_1 =  \norm{\GA'_1}_{L^1(\Circ,\RR^2)}^2 C_0^2/\sigma^2$ ($C_0$ denotes here  the constant of the embedding of $W^{1,2}(\Circ,\RR^2)$ in $L^\infty(\Circ,\RR^2)$ so that $\|\GA\|_{L^\infty(\Circ,\RR^2)}\leq C_0 \|\GA\|_{W^{1,2}(\Circ,\RR^2)}$). 

Moreover, by~\eqref{boundI}, we have  
$$\|\GA'_1\Ii(\GA_2,\La) - \GA'_2\Ii(\GA_2,\La)\|^2_{L^2(\Circ,\RR^2)}  
 	 \leq \alpha^2  \norm{\La'}_{L^1(\Circ,\RR^2)}^2 \|\GA'_1 - \GA'_2\|_{L^2(\Circ,\RR^2)}^2$$
which implies 
\begin{equation}\label{lip3}
  \displaystyle{\|\GA'_1\Ii(\GA_2,\La) - \GA'_2\Ii(\GA_2,\La)\|^2_{L^2(\Circ,\RR^2)}\leq  C_2\|\GA_1 - \GA_2\|_{W^{1,2}(\Circ,\RR^2)}^2}
\end{equation}
where $C_2=\alpha^2  \norm{\La'}_{L^1(\Circ,\RR^2)}^2$. Then, by~\eqref{lip1},~\eqref{lip2} and~\eqref{lip3}, the $W^{1,2}$-gradient of the energy verifies \eqref{grad-lip} on every 
set of curves of bounded length. This guarantees actually that the constant $C_1$ is uniformly bounded and we can define the Lipschitz constant. 
\end{proof}

Therefore, the application of Theorem \ref{convergence} gives
\begin{cor} 
Under the assumption that the lengths of $\GA_k$ are bounded, every accumulation point of $\{\GA_k\}$ in $\Bb = \BVDcirc$ is a  critical point of $E$.
\end{cor}

\begin{rem}
We were not able to relax the boundedness assumption although this seems rather plausible under the assumptions that the initial and target curves are in $BV^2(\Circ,\R^2)$.
The result of the corollary is however relatively weak in the sense that it is difficult to  check numerically the convergence in $BV^2(\Circ,\R^2)$.
\end{rem}

\section{Discretization}
\label{discretization}

This section discretizes Problem~\eqref{grad-rigid} using finite elements in order to calculate numerically the Finsler gradient flow. We define a $n$-dimensional sub-space $\Bb_n \subset \Bb$ of piecewise linear curves. The embedding $\Bb_n \subset \Bb$ defines a natural finite dimensional Riemannian and Finsler structure on $\Bb_n$ inherited from the ones of $\Bb$. This allows us to apply our Finsler gradient flow in finite dimension to approximate the original infinite dimensional Finsler flow.

\subsection{Finite Elements Spaces}

\paragraph{Notations.}

In the following, to ease the notation, we identify $\RR^2$ with $\CC$ and $\Circ$ with $[0,1]$ using periodic boundary conditions.  
The canonical inner produced on $\CC^n$ is
\begin{equation}\label{scalarC}
	\dotp{\tilde{f}}{\tilde{g}}_{\mathbb{C}^n} = 
	\sum_{i=1}^n \dotp{\tilde{f}_i}{\tilde{g}_i}
	=
	\sum_{i=1}^n {\rm Real}(\tilde{f}_i \,\overline{\tilde{g}_i}), 
	\,\quad \foralls \tilde{f}, \tilde{g} \in\CC^n \,,
\end{equation}
where we denote by $\overline{\tilde{g}_i}$ the conjugate of $\tilde{g}_i$.

\paragraph{Piecewise affine finite elements.}

We consider the  space $\PP_{1,n}$ of the finite elements on $[0,1]$ (with periodic boundary conditions) of order one with $n$ equispaced nodes. 
A basis of $\PP_{1,n}$ is defined as 
\begin{align*}
	\xi_i(s) &= \max\left\{0, 1- n\left| s- \frac{i}{n} \right|\right\} \quad  s\in [0,1],\;  \foralls i= 1, ..., n-1 \\
	\xi_n(s) &= \max\left\{0, 1- n\left|s \right|\right\} + \;\max\left\{0, 1- n\left| s- 1 \right|\right\},  \quad s\in [0,1].
\end{align*}
Every  $f\in \PP_{1,n}$ can be written as 
\begin{equation}\label{tangent-discrete}
	f=\sum_{i=1}^n\tilde{ f}_i \, \xi_i\,,\quad \tilde{f}_i\in\CC\,
\end{equation}
with $\tilde{f}_i=f(i/n)\in\CC$ for every $i$. We denote by $\tilde{f}=(\tilde{f}_1,..., \tilde{f}_n)\in \CC^n$ the coordinates of $f$ with respect to  the basis $\{\xi_i\}_{i=1,...,n}$. 
Remark that there exists a bijection between $\PP_{1,n}$ and $\CC^n$, defined by the following operator
\begin{equation}\label{bijection}
	P_1\; : \tilde{f}=(\tilde{f}_1,..., \tilde{f}_n)\in \CC^n 
	\;\;\mapsto\;\; 
	P_1(\tilde{f})=  \; f\in \PP_{1,n} \mbox{\;\;s.t.\;\;} f=\sum_{i=1}^n\tilde{f}_i \, \xi_i\,.
\end{equation}
The forward and backward finite differences operators are defined as
\begin{equation}\label{delta}
\begin{array}{ll}
	\Delta^+: \CC^n\rightarrow  \CC^n \;,& \quad \Delta^+( \tilde{f})_i = n(\tilde{f}_{i+1}-\tilde{f}_i)\,,\\
	\Delta^-: \CC^n\rightarrow  \CC^n \;, &\quad \Delta^-( \tilde{f})_i = n(\tilde{f}_i -\tilde{f}_{i-1})\,,
\end{array}
\end{equation} 

\paragraph{Piecewise constant finite elements.}

\newcommand{\myI}{\mathbb{I}}

For every $f\in \PP_{1,n}$,~\eqref{tangent-discrete} implies that  first derivative $\frac{\d f}{\d s}$ belongs to $\PP_{0,n} \subset  BV([0,1], \RR^2)$, where $\PP_{0,n}$ is the class of the piecewise constant functions with $n$ equispaced  nodes. A basis of $\PP_{0,n}$ is defined by 
\begin{align*}
\zeta_i(s) =\myI{}_{[\frac{i}{n},\frac{i+1}{n}]}(s) \quad  \foralls i= 1, ..., n-1 \,,\quad 
\zeta_n(s) =\myI{}_{[0,\frac{1}{n}]}(s) \,,
\end{align*}
where $\myI_A$ is the characteristic function of a set $A$, and with $s\in [0,1]$. Then, the first derivative of $f$ can be written as 
\begin{equation}\label{first-derivative-approx}
\frac{\d f}{\d s} = \sum_{i=1}^n \Delta^+(\tilde{f})_i\zeta_i\,.
\end{equation}
We finally define the following bijection between $\PP_{0,n}$ and $\CC^n$:
\begin{equation}\label{bijection2}
P_0\; : \; \tilde{f}=(\tilde{f}_1,..., \tilde{f}_n)\in \CC^n\;\;\mapsto\;\; P_0(\tilde{f})= f\in \PP_{0,n}   \mbox{\;\;s.t.\;\;} f=\sum_{i=1}^n\tilde{ f}_i\zeta_i\,.
\end{equation}

\subsection{Finite Element Spaces of Curves} 

\paragraph{Discretized curves.} 

The discrete space of curves is defined as $\Bb_n= \PP_{1,n} \subset \Bb$ and every curve $\GA\in\Bb_n$ can be written as
\begin{equation}\label{gamma-coord}
	\GA=\sum_{i=1}^n \tGa_i\xi_i\,,\quad \tGa_i=\GA(i/n)\in \CC\,
\end{equation}
where the vector $\tGa=P_1^{-1}(\GA) = (\tGa_1, ..., \tGa_n)\in \CC^n$ contains the coefficients of $\GA$ in the finite element basis. By~\eqref{first-derivative-approx} the tangent and normal vectors~\eqref{def-tangente} to $\GA\in \Bb_n$ are computed as 
\begin{equation}\label{tan-norm}
\tgam= \sum_{i=1}^n\frac{\Delta^+ (\tGa)_i}{|\Delta^+ (\tGa)_i|}\zeta_i  \;,\;\; \ngam(i)= \rot{\tgam(i)}~,
\end{equation}
where $\rot{(x,y)}=(-y,x)$ for all $(x,y) \in \RR^2$.
In particular we have 
\begin{equation}\label{arc-length}
	\frac{\d\GA}{\d s} = \sum_{i=1}^n \Delta^+(\tGa)_i\zeta_i\,.
\end{equation}

\paragraph{Discretized tangent spaces.}  

For every  $\GA\in\Bb_n$, the discrete tangent space to $\Bb_n$ at $\GA$ is defined as  $T_{\GA} \Bb_n = \Bb_n$ equipped with the  inner product $\dotp{\cdot}{\cdot}_{H^1(\GA)}$. Every vector field $\Phi\in T_\GA \Bb_n$ can be written as
\begin{equation}\label{representation-phi}
\Phi=\sum_{i=1}^n \tPhi_i\xi_i\,,\quad \tPhi_i=\Phi(i/n)\in \CC\,
\end{equation}
where $\tPhi=(\tPhi_1, ..., \tPhi_n)\in \CC^n$ are the coordinates of $\Phi$ with respect to the basis of $\PP_{1,n}$. 

By identifying every vector field $\Phi \in T_\GA\Bb_n$ with its coordinates $\tPhi$, the tangent space can be identified with  $\CC^n$. 
In particular we have 
\begin{equation}\label{arc-length-2}
	\frac{\d\Phi}{\d \Gamma} = \sum_{i=1}^n \frac{\Delta^+(\tilde{\Phi})_i}{|\Delta^+(\tilde{\GA})_i|}\zeta_i\,.
\end{equation}
Moreover, $\CC^n$ can be equipped with the following Riemannian metric:
 \begin{defn}[\BoxTitle{Discrete inner product}]
 We define $\ell^2(\tilde{\GA})$ and  $h^1(\tilde{\GA})$ as the set $\CC^n$ equipped with the following inner products respectively
 \begin{equation}\label{scal-prod-coord} 
	\dotp{\tPhi}{\tilde{ \Psi}}_{\ell^2(\tilde{\GA})} = 
	\dotp{P_1(\tPhi)}{P_1(\tPsi)}_{L^2(\GA)}  \,,
\end{equation} 
\begin{equation}\label{scal-prod-coord-sob} 
	\dotp{\tPhi}{\tilde{ \Psi}}_{h^1(\tilde{\GA})} = 
	\dotp{P_1(\tPhi)}{P_1(\tPsi)}_{H^1(\GA)}  \,. 
\end{equation} 
 \end{defn}

We now give the explicit formulas for the products~\eqref{scal-prod-coord} and \eqref{scal-prod-coord-sob}, which are  useful for computational purposes. 

Proposition~\eqref{prod-scla-cont} details the relationship between the product~\eqref{scal-prod-coord} and the canonical inner product on $\CC^n$ defined by~\eqref{scalarC}. For this purpose, we define  the mass matrix $M_{\tGa} \in \RR^{n \times n}$  as
\begin{equation}\label{rigidity}
	M_{\tGa} = \sum_{i=1}^n |\Delta^+(\tGa)_i| M^i\,
	\qwhereq
	M^i_{h,j}=\int_{i/n}^{(i+1)/n}\xi_h\xi_j\,.
\end{equation} 
The elements of the matrices $M^i \in \RR^{n \times n}$ for $i=1,...,n$ are equal to zero excepted for the following block:  
\eq{
\begin{pmatrix}

M^i_{i,i}& M^i_{i,i+1}\\
M^i_{i+1,i}& M^i_{i+1,i+1}\\
\end{pmatrix}
=
\frac{1}{6n}
\begin{pmatrix}
2&1\\
1&2\\

\end{pmatrix}, 
}
where the indices $i-1$ and $i+1$ should be understood modulo $n$.

\begin{prop}\label{prod-scla-cont} For all $\tPsi, \tPhi$ in $\CC^n$, one has
\begin{equation}\label{discreteprod}
 	\dotp{\tPhi}{\tPsi}_{\ell^2(\tilde{\GA})}
	= 
	\dotp{\tPhi}{M_{\tGa}\tPsi}_{\mathbb{C}^n}, 
\end{equation}
where $M_{\tGa}$ is the mass matrix defined in~\eqref{rigidity}.
\end{prop}

\begin{proof}  Denoting $\Phi=P_1(\tPhi)$ and $\Psi=P_1(\tPsi)$,
\eqref{representation-phi} and~\eqref{arc-length} imply that
\eq{
	\dotp{\Phi}{\Psi}_{L^2(\GA)}=\int_0^1 \,\Phi\cdot \Psi d\,\GA(s) 
	= 
	\sum_{i=1}^n |\Delta^+(\tGa)_i|\int_{i/n}^{(i+1)/n} 
	\left( \,\sum_{j=1}^n \tPhi_j \xi_j \cdot\sum_{h=1}^n \tPsi_h\xi_h\right) \d s\,.
}
Then, since \eqref{scal-prod-coord}, we have 
\begin{equation}
	\dotp{\tPhi}{\tPsi}_{\ell^2(\tilde{\GA})} 
	= 
	\sum_{i=1}^n |\Delta^+( \tGa)_i| \dotp{\tPhi}{M^i\tPsi}_{\mathbb{C}^n}
	= 
	\dotp{\tPhi}{M_{\tGa} \tPsi}_{\mathbb{C}^n}
\end{equation}
where $M_{\tGa}$ is the  mass matrix~\eqref{rigidity}.
\end{proof}

The next proposition  details the relationship between the product~\eqref{scal-prod-coord-sob} and the canonical inner product on $\CC^n$. To this end, we introduce  the  matrix $N_{\tGa} \in \RR^{n \times n}$  defined by
\begin{equation}\label{rigidity-2}
	N_{\tGa} = \sum_{i=1}^n |\Delta^+(\tGa)_i| N^i\,
	\qwhereq
	N^i_{h,j}=\frac{1}{|\Delta^+(\tGa)_j||\Delta^+(\tGa)_h|}\int_{i/n}^{(i+1)/n}\frac{\d\xi_h}{\d s}\cdot\frac{\d\xi_j}{\d s}\,.
\end{equation} 
The elements of the matrices $N^i \in \RR^{n \times n}$ for $i=1,...,n$ are equal to zero except for the following block:  
\eq{
\begin{pmatrix}

N^i_{i,i}& N^i_{i,i+1}\\
N^i_{i+1,i}& N^i_{i+1,i+1}\\
\end{pmatrix}
=
n
\begin{pmatrix}
\displaystyle{\frac{1}{|\Delta^+(\tGa)_i|^2}}&\displaystyle{-\frac{1}{|\Delta^+(\tGa)_i||\Delta^+(\tGa)_{i+1}|}}\\
\displaystyle{-\frac{1}{|\Delta^+(\tGa)_i||\Delta^+(\tGa)_{i+1}|}}&\displaystyle{\frac{1}{|\Delta^+(\tGa)_{i+1}|^2}}\\

\end{pmatrix}\,.
}
 

\begin{prop}\label{prod-scla-cont-sob} For all $\tPsi, \tPhi$ in $\CC^n$, one has
\begin{equation}\label{discreteprod-sob}
 	\dotp{\tPhi}{\tPsi}_{h^1(\tilde{\GA})}
	= 
	\dotp{\tPhi}{U_{\tGa}\tPsi}_{\mathbb{C}^n}, 
\end{equation}
where $U_{\tGa}$ is the matrix defined by 
\begin{equation}\label{rigidity-mass-h1}
U_{\tGa} = M_{\tGa} +  N_{\tGa} \,,
\end{equation}
where $M_{\tGa}$, $N_{\tGa}$  are the   matrix \eqref{rigidity} and \eqref{rigidity-2} respectively. 
We point out that, since $U_{\tilde\Gamma}$ is a  matrix of an inner product in a basis,  it is always invertible.
\end{prop}

\begin{proof}  Denoting $\Phi=P_1(\tPhi)$ and $\Psi=P_1(\tPsi)$, \eqref{representation-phi} implies 
\eq{
	\langle \frac{\d \Phi}{\d \GA}, \frac{\d \Psi}{\d \GA}\rangle_{L^2(\GA)}=
	\sum_{i=1}^n |\Delta^+(\tGa)_i|\int_{i/n}^{(i+1)/n} 
	\left( \,\sum_{j=1}^n \frac{\tPhi_j}{|\Delta^+(\tGa)_j|} \frac{\d \zeta_j}{\d s} \cdot\sum_{h=1}^n \frac{\tPsi_h}{|\Delta^+(\tGa)_h|}\frac{\d \zeta_h}{\d s}\right) \d s\,.
}
Then, by previous proposition, we have 
$$
	\dotp{\tPhi}{\tPsi}_{h^1(\tilde{\GA})} 
	= \dotp{\tPhi}{M_{\tGa} \tPsi}_{\mathbb{C}^n}	 
	+ \dotp{\tPhi}{N_{\tGa}  \tPsi}_{\mathbb{C}^n}
$$
where $M_{\tGa}$, $N_{\tGa}$  are the   matrices \eqref{rigidity} and \eqref{rigidity-2} respectively. 
 
\end{proof}

\subsection{Discrete Finsler Flow}

The initial optimization~\eqref{initial-eq} is discretized by restricting the minimization to the space $\Bb_n$, which corresponds to the following finite dimensional optimization
\eql{\label{eq-min-discr}
	\umin{\tGa \in \CC^n} \tE(\tGa)~,
}
where $\tilde E(\tGa)$ approximates $E(P_1(\tGa))$.

The discrete Finsler gradient is obtained in a similar way by restricting the optimization~\eqref{defgrad} to $\Bb_n$
\eql{\label{eq-finsler-grad-discr}
	\nabla_{\tR_{\tGa}} \tE(\tGa) \in 
	\uargmin{\tPhi \in \tilde\Ll_{\tGa}} \tR_{\tGa}(\tPhi)~, 
}
where the discrete penalty  reads
\eql{\label{defL-discrete}
	\tR_{\tGa}(\tPhi) = R_{P_1(\tGa)}(P_1(\tPhi))
}
and, as discrete constraint, we set 
\eql{
	\label{eq-expression-Ll}
	\tilde{\Ll}_{\tGa} = \enscond{ \tPhi \in \CC^n }{
				\normbig{\tPi_{\tGa}( \nabla_{h^1(\tilde{\GA})} \tE(\tilde{\ga}) -\tPhi) }_{h^1(\tilde{\ga})} 
				\leq 
				\rho \normbig{\tPi_{\tGa}( \nabla_{h^1(\tGa)} \tE(\tilde{\ga})) }_{h^1(\tilde{\ga})} 
			} 
\,.}
The Finsler flow discretizing the original one~\eqref{subsec-finsler-descent} reads
\eql{\label{eq-finsler-flow-discr}
	\tGa_{k+1} = \tGa_k - \tau_k \nabla_{\tR_{\tGa_k}} \tE(\tGa_k).
}
where $\tau_k>0$ is chosen following the Wolfe rule~\eqref{Wolfe}.

The following sections detail how to compute this flow for the particular case of the curve matching energy introduced in Section~\ref{CM}. 

\subsection{Discrete Energy} 

\paragraph{Exact energy for piecewise affine curves. }

For curves $\GA=P_1(\tGa)$ and $\La=P_1(\tLa)$ in $\Bb_n$, the energy $E(\GA)$ defined in~\eqref{energy} can be computed as
\eq{
	E(\GA) = 
		\frac{1}{2}\Zz(\GA,\GA) - 
		\Zz(\Ga,\La) + 
		\frac{1}{2}\Zz(\La,\La)
}
where
\eq{
	\Zz(\Ga, \La) =
 	\sum_{i=1}^n \sum_{j=1}^n \dotp{ \Delta^+(\tGa)_i }{ \Delta^+( \tLa)_j } 
	T(\tGa,\tLa)_{i,j}
}
\eq{
	\qwhereq
	T(\tGa,\tLa)_{i,j} = \int_{\frac{i-1}{n}}^{ \frac{i}{n}} \int_{\frac{j-1}{n}}^{\frac{j}{n}} k(\GA(s), \La(t))\, \d \GA(s) \d \La(t)\,.
}

\paragraph{Approximate energy for piecewise affine curves. }

In general there is no closed form expression for the operator $T$, so that, to enable a direct computation of the energy and its gradient, we use a first order approximation with a trapezoidal quadrature formula
\eq{
	\tilde T(\tGa,\tLa)_{i,j} = 
	\frac{1}{4} \big(
		k(\tilde\GA_i,\tLa_j) + k(\tilde\GA_{i+1},\tLa_j) + 
		k(\tilde\GA_i,\tLa_{j+1}) + k(\tilde\GA_{i+1},\tLa_{j+1})
	\big)\,.
}
One thus has the approximation
\eq{
	\tilde T(\tGa,\tLa)_{i,j} =  T(\tGa,\tLa)_{i,j} + O(1/n^2). 
}
This defines the discrete energy $\tE$ on $\CC^n$ as 
\eql{\label{eq-energy-matching-discr}
	\tE(\tGa) = 
	\frac{1}{2}\tilde{\Zz}(\tGa,\tGa) - 
	\tilde{\Zz}(\tGa,\tLa) + 
	\frac{1}{2}\tilde{\Zz}(\tLa,\tLa)
}
\eq{
	\qwhereq
	\tilde{\Zz}(\tGa,\tLa) =
	\sum_{i=1}^n \sum_{j=1}^n \dotp{\Delta^+ (\tGa)_i }{ \Delta^+( \tLa)_j } \tilde T(\tGa,\tLa)_{i,j}
}

\paragraph{Discrete $h^1$-gradient. }

The following proposition gives the formula to calculate the gradient of $\tE$ with respect to inner product~\eqref{scal-prod-coord-sob}.
\begin{prop}
The gradient of $\tE$ at $\tGa$ with respect to the metric defined by the  inner product~\eqref{scal-prod-coord-sob} is
\eq{
	\nabla_{h^1(\tilde{\GA})} \tE(\tGa) = U_{\tGa}^{-1}\nabla \tE(\tGa)\,
}
where $U_{\tGa}$ is the  matrix~\eqref{rigidity-mass-h1} and $\nabla \tE$ the  gradient of $\tE$ for the canonical inner product of $\CC^{n}$~\eqref{scalarC}, which is given by

\begin{equation}\label{discretegradient}
\nabla \tE(\tGa)_i = \nabla \tilde{\Zz}(\tGa, \tGa)_i -\nabla \tilde{\Zz}(\tGa, \tLa)_i
\end{equation}
where 
$$
\begin{array}{ll}
 \nabla \tilde{\Zz}(\tGa, \tLa)_i &\displaystyle{ = 
 	\frac{1}{4} \sum_{j=1}^n  (\tGa_{i+1}-\tGa_{i-1}) (\tLa_{j+1}-\tLa_j)[\nabla_1 k(\tGa_i,\tLa_j)+ \nabla_1 k(\tGa_i,\tLa_{j+1})]} \\
		& \displaystyle{+ \sum_{j=1}^n  (\tLa_{j+1}-\tLa_j)[ T(\tGa,\tLa)_{i-1,j} - T(\tGa,\tLa)_{i,j} ]}\,.
		\end{array}
		$$

\end{prop}

\begin{proof}
The  gradient~\eqref{discretegradient} of $\tE$ for the canonical inner product of $\CC^{n}$  can be computed by a straightforward calculation. For every $\tPhi\in  \CC^n$ we have the following expression for the derivative of $\tE$

\eq{
	D \tE (\tGa)(\tPhi)= \dotp{\tPhi}{\nabla_{h^1(\tilde{\GA})} \tE(\tGa)}_{h^1(\tilde{\GA})}
	= 
	\dotp{\tPhi}{\nabla\tE(\tGa)}_{\mathbb{C}^n}
}
and, by~\eqref{discreteprod}, we get 
\eq{
	\nabla_{h^1(\tilde{\GA})} \tE(\tGa) = U_{\tGa}^{-1}\nabla \tE(\tGa)\,.
}
\end{proof}

\subsection{Discrete Piecewise-rigid Curve Matching}

This section first describes in a general setting the discrete Finsler gradient over finite-element spaces, then specializes it to the piecewise rigid penalty for the matching problem, and lastly gives the explicit formula of the corresponding functionals to be minimized numerically. 

\paragraph{Discrete piecewise-rigid penalty.}

To re-write conveniently the discrete Finsler gradient optimization~\eqref{eq-finsler-grad-discr}, we introduce the following finite-dimensional operators.

\begin{defn}[\BoxTitle{Discrete operators}]\label{dfn-discr-op} For all $\Ga=P_1(\tGa), \Phi=P_1(\tPhi)$ we define
	\begin{align*}
		\tV_{\tGa}(\tPhi) &= TV_{\GA}\left(\frac{\d \Phi}{\d \GA}\cdot {\bf n}_{\Ga} \right), \qquad \tV_{\tGa}: \CC^n\rightarrow \RR   \\
		\tL_{\tGa}(\tPhi) &=P_0^{-1}(L_{\GA}^+(\Phi)),  
			\qquad \tL_{\tGa}: \CC^n\rightarrow \RR^n \\
		\tPi_{\tGa}(\tPhi) &=P_0^{-1}(\Pi_{\GA}(\Phi)), 
			\qquad \tPi_{\tGa}: \CC^n\rightarrow \CC^n
	\end{align*}
\end{defn}

The following proposition uses these discrete operators to compute the discrete Finsler penalty and constraint defined in~\eqref{defL-discrete}. 

\begin{prop}
We set $\tR_{\tGa}(\tPhi)= R_{P_1(\Ga)}(P_1(\Phi))$. One has 
\eql{\label{eq-discr-pr-penalty}
	\tR_{\tGa}(\tPhi) = \tV_{\tGa}(\tPhi) + \iota_{\tilde\Cc_{\tGa}}(\tPhi)
	\qwhereq
	\tilde\Cc_{\tGa} = \enscond{ \tPhi \in \CC^n }{ \tL_{\tGa}(\tPhi)=0 }.
}
\end{prop}
\begin{proof} 
Denoting $\Ga=P_1(\tGa), \Phi=P_1(\tPhi)$, by~\eqref{defL-discrete}, we have	
$$\tR_{\tGa}(\tPhi) = R_{\Ga}(\Phi)= TV_{\GA}\left(\frac{\d \Phi}{\d \GA}\cdot {\bf n}_{\Ga} \right)  + \iota_{\Cc_{\Ga}}(\Phi) = \tV_{\tGa}(\tPhi)  + \iota_{\tilde\Cc_{\tGa}}(\tPhi)$$
where 
$$\tilde\Cc_{\tGa}=\enscond{ \tPhi \in \CC^n }{ \tL_{\tGa}(\tPhi)=0 }.$$

\end{proof}

The following proposition gives explicit formulae for the discrete operators introduced in Definition~\ref{dfn-discr-op}.

\begin{prop}\label{constraints-discrete} For every $\tGa, \tPhi\in\CC^n$, we consider $\GA=P_1(\tGa)\in\Bb_n$, $\Phi=P_1(\tPhi)\in T_\GA \Bb_n$. One has 
\begin{align}\label{eq-expression-tL}
	\tL_{\tGa}(\tPhi)_i &= \dotp{\frac{\Delta^+(\tPhi)_i}{|\Delta^+(\tGa)_i|} }{ \frac{\Delta^+ (\tGa)_i}{|\Delta^+ (\tGa)_i|} }\, , \\
	\label{eq-expression-tPi}
	\tPi_{\tGa}(\tPhi)(s) &= \sum_{i=1}^n \langle \tPhi_i\xi_i(s) + \tPhi_{i+1} \xi_{i
+1}(s),  (\tilde{\bf n}_\ga)_i \rangle(\tilde{\bf n}_\ga)_i \zeta_i(s) \, ,
	  \\
	\label{eq-expression-tV}
		\tV_{\tGa}(\tPhi) & 
		= \norm{\Delta^- ( \tilde H_{\tGa}(\tPhi) ) }_{\ell^1} 
		= 
		\sum_{i=1}^n
		\left| 
			\tilde H_{\tGa}(\tPhi)_i
			-
			\tilde H_{\tGa}(\tPhi)_{i-1}
		\right|\, , \\
	\mbox{where}\;\;\tilde H_{\tGa}(\tPhi)_i & := \dotp{ \frac{\Delta^+(\tPhi)_i}{|\Delta^+(\tGa)_i|} }{ (\tilde{\bf n}_\ga)_i }
\end{align}
and where $\tilde{\bf n}_\ga$ denotes the vector of the coordinates of ${\bf n}_\ga$ with respect to the basis of $\mathbb{P}_0$.
\end{prop}

\begin{proof} 
\textbf{(Proof of~\eqref{eq-expression-tL})} Using~\eqref{first-derivative-approx}  the first derivative of $\Phi$ can be written (with respect to the basis of $\PP_{0,n}$) as 
\eq{
	\frac{\d \Phi}{\d \GA} = \sum_{i=1}^n \frac{\Delta^+(\tPhi)_i}{|\Delta^+(\tGa)_i|} \zeta_i\,
}
which implies that
\eq{
	L_{\GA}^+(\Phi)= \frac{\d \Phi}{\d \GA}\cdot \tgam = \sum_{i=1}^n \langle\frac{\Delta^+(\tPhi)_i}{|\Delta^+(\tGa)_i|} ,\frac{\Delta^+ (\tGa)_i}{|\Delta^+ (\tGa)_i|} \rangle\zeta_i\,.
}
Then, by the definitions of $L_{\GA}^{+(-)}$,  conditions $L_{\Ga}^{+(-)}(\Phi)=0$ become
\eq{
	\langle\frac{\Delta^+(\tPhi)_i}{|\Delta^+(\tGa)_i|} ,\frac{\Delta^+ (\tGa)_i}{|\Delta^+ (\tGa)_i|} \rangle=0 
	\quad \quad \foralls i=1,...,n,
}
which is equivalent to $\tL_{\tGa}(\tPhi)=0$.\\
\textbf{(Proof of~\eqref{eq-expression-tPi})} By~\eqref{representation-phi} and~\eqref{tan-norm}, we get 
\eq{
\begin{array}{ll}
\Pi_\GA(\Phi)(s) & = \displaystyle{\langle \sum_{i=1}^n \tPhi_i\xi_i(s), \sum_{i=1}^n (\tilde{\bf n}_\ga)_i  \zeta_i(s) \rangle  \sum_{i=1}^n (\tilde{\bf n}_\ga)_i  \zeta_i(s)\,}\\
& =\displaystyle{  \sum_{i=1}^n \langle \tPhi_i \xi_i(s) + \tPhi_{i+1} \xi_{i+1}(s) ,  (\tilde{\bf n}_\ga)_i \rangle(\tilde{\bf n}_\ga)_i \zeta_i(s)}
\end{array}
}
which proves the result. \\
\textbf{(Proof of~\eqref{eq-expression-tV})} By~\eqref{first-derivative-approx} and~\eqref{tan-norm}, we get
\begin{align*}
  TV_\GA\left( \frac{\d \Phi}{\d \GA}\cdot {\bf n}_\GA\right)
  &=  
  TV_\GA\left( \sum_{i=1}^n \langle\frac{\Delta^+(\tPhi)_i}{|\Delta^+(\tGa)_i|},(\tilde{\bf n}_\ga)_i \rangle\zeta_i \right)\\
  &= \sum_{i=1}^n
		\left| 
			\tilde H_{\tGa}(\tPhi)_i
			-
			\tilde H_{\tGa}(\tPhi)_{i-1}
		\right|
\end{align*}
where we used the fact that the  total variation for piecewise constant functions coincides with the sum of jumps sizes.
\end{proof}

\subsection{Calculation of the Discrete Finsler Gradient}
\label{calculation-finsler}

One can minimize the matching energy $\tE$ defined in~\eqref{eq-energy-matching-discr} using the Finsler flow $\{\tGa_k\}$ of~\eqref{eq-finsler-flow-discr}. This requires computing at each step $k$ the Finsler gradient~\eqref{eq-finsler-grad-discr} for the piecewise-rigid penalty $\tR_{\tGa}$ defined in~\eqref{eq-discr-pr-penalty}. Solving~\eqref{eq-finsler-grad-discr} at each step in turn requires the resolution of a finite dimensional convex problem, and the functional to be minimized is explicitly given with closed form formula in Proposition~\ref{constraints-discrete}.

Several convex optimization algorithms can be used to solve~\eqref{eq-finsler-grad-discr}. A convenient method consists in recasting the problem into a second order cone program by introducing additional auxiliary variables $(\tilde{\Phi}, \tilde{S},\tilde{Y},\tilde{T})$ as follow
\eq{
	\underset{ (\tilde{\Phi}, \tilde{S},\tilde{Y},\tilde{T}) \in \CC^{2n}\times\RR^{2n} }{ {\rm Min} }
		\dotp{\tilde{Y} }{ {\bf 1}}_{\mathbb{C}^n} 	
		\qwhereq {\bf 1}=(1,...,1) \in \RR^n
}
where the minimum is taken under the following set of affine and conic constraints
\begin{align*}
	& -\tilde{Y}_i\leq \Delta^-( \tilde{H}_{\tilde \Gamma} (\tilde \Phi) )_i  \leq \tilde{Y}_i \;,\quad \quad  \foralls i=1,...,n \\
	& \tL_{\tGa}(\tilde{\Phi}) =0 \\
	& \tilde{S} =  U_{\tGa}^{1/2}(\tPi_{\tilde{\Ga}}(\tilde{\Phi}) - \tPi_{\tilde{\Ga}}(\nabla_{h^1(\tilde{\GA})} \tE(\tilde{\Ga}))) \\
	& \langle \tilde{T}, {\bf 1}\rangle_{\mathbb{C}^n} \leq \rho^2 \normbig{\tPi_{\tGa}( \nabla_{h^1(\tilde{\GA})} \tE(\tilde{\ga})) }_{h^1(\tilde{\ga})}^2  \\
	& (\tilde{S}_i,\tilde{T}_i) \in 
		\enscond{ (s,t) \in \CC \times \RR }{ |s|^2 \leq t } ,\quad \foralls i=1,...,n.
\end{align*}
We point out that the variable $\tilde{S}$ is defined by the mass matrix $U_{\tGa}$ \eqref{rigidity-mass-h1} because of the relationship \eqref{discreteprod-sob}. For the numerical simulation, we use an interior point solver, see~\cite{convex-optimization}. These interior points algorithms are powerful methods to solve medium scale SOCP problems and work remarkably well for $n$ up to several thousands, which is typically the case for the curve matching problem.

\section{Numerical Examples}\label{examples}

In this section we give some numerical examples to point out the properties of the piecewise rigid Finsler evolution. 

It should be noted that the resulting sequence $\{\Gamma_k\}$ depends on the choice of the step sizes $\{\tau_k\}$, which is left to the user and should only comply with the Wolfe conditions \eqref{Wolfe}.

Numerically, we observe in practice  that choosing small enough (according to the Wolfe condition) step sizes $\tau_k$ always provides consistent evolutions. This phenomenon is related to the existence of a limiting gradient flow (as highlighted in Remark \ref{grad_flow}), and the depicted evolutions are intended to show an approximation of this flow.

For the numerical examples shown in this section and in Section \ref{NE}, we used a fixed finite element discretization as detailed in Section \ref{discretization} (with $n=1280$). 
	This corresponds to 
 	imposing a fixed common parameterization of the discretized curves generated by the iterations. Note however that applications
	to more complicated imaging problems might require re-parameterizing the curves from time to time during the iterations
	of the gradient descent \eqref{sequence}. This is important when deading with complicated shapes since the parameterization 
	might become ill-conditionned, 
	which can deteriorate the numerical accuracy of the scheme. 

\subsection{Influence of $\rho$} 
\label{sec-influ-rho}

To exemplify the main properties of the piecewise rigid Finsler flow, we introduce a synthetic example where we replace in the definition~\eqref{grad-rigid} of the Finsler gradient $\nabla_{R_\GA} E(\GA)$ (more precisely in the definition~\eqref{eq-dev-constr} of the constraint $\Ll_\Ga$) the  gradient $\nablad E(\Ga)$ by the vector field $F(\Ga) \in T_\Ga \Bb$ defined as
\eql{\label{field-const}
	F(\Ga) : s \in \Circ \mapsto -\pa{ 5\GA_1(s), 1000(\Ga_2(s)-1/2)^2} \in \RR^2
}
\eq{
	\qwhereq
	\Ga(s) = (\Ga_1(s),\Ga_2(s)) \in \RR^2.
}
The initial flow associated to this vector field reads
\eql{\label{eq-synthetic-flow}
	\Ga_{k+1} = \Ga_k - \tau_k F( \Ga_k )
}
for some small enough time steps $\tau_k>0$. Such a flow is represented in Figure~\ref{-F} where $\tau_k = 0.0005$ for every $k$.

\begin{figure}[!h]
\begin{center}
\includegraphics[width=.3\linewidth]{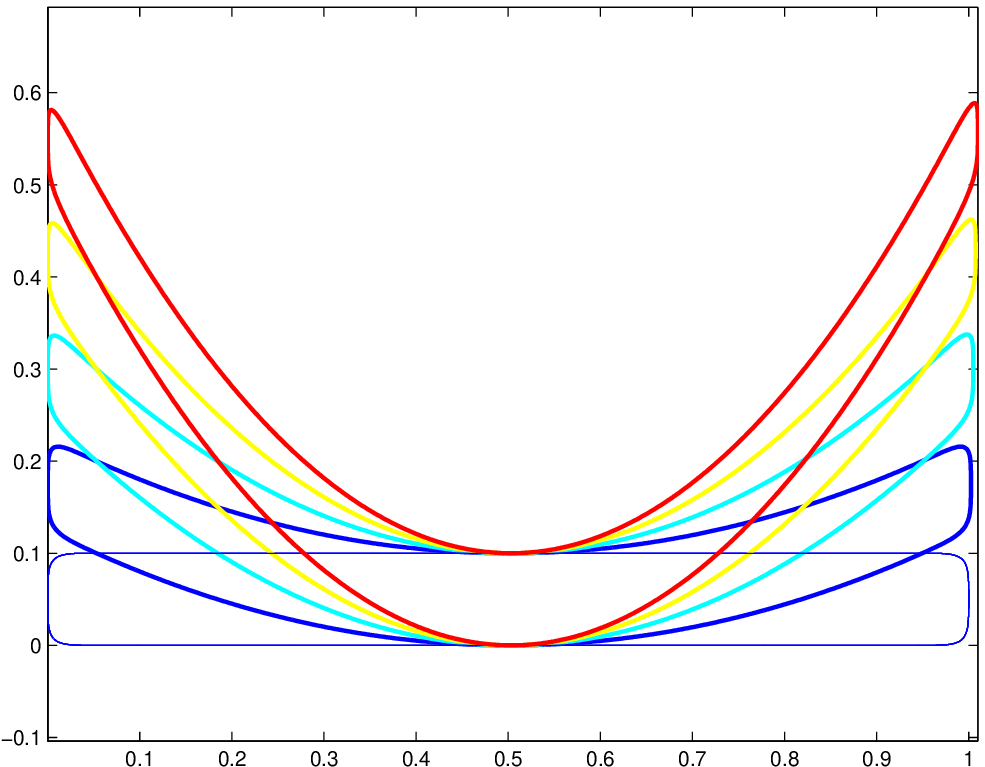}
\end{center}
\caption{\label{-F} Evolution generated by $-F$. }
\end{figure}

Figure~\ref{fig1} shows the impact of the parameter $\rho$ on this evolution. As $\rho$ increases, the evolution becomes increasingly piecewise rigid. For $\rho$ large enough, it is globally rigid, i.e. satisfies~\eqref{eq-rigid-flow} and  $\nabla_{R_{\Ga_k}} E(\Ga_k) \in \Rr_{\Ga_k}$ for all $k$, where $\Rr_{\Ga}$ is defined in~\eqref{eq-rigid-vector-space}. 

\begin{figure}[!h]
\begin{center}
\begin{tabular}{@{}c@{\hspace{5mm}}c@{}}
\includegraphics[width=.3\linewidth]{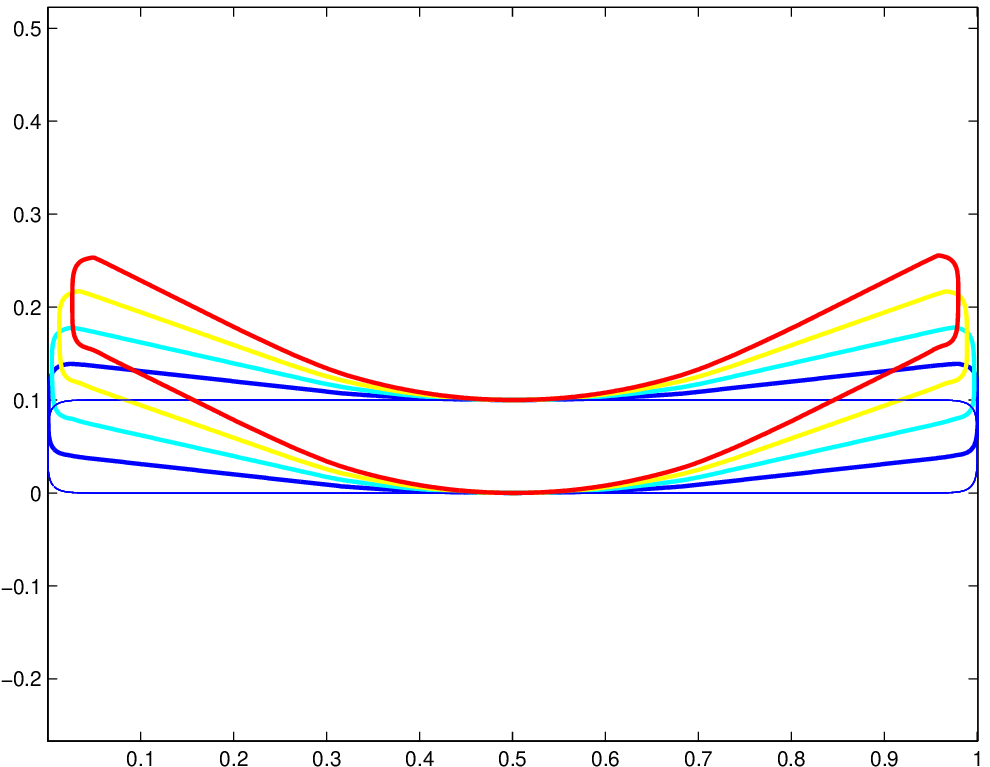}&
\includegraphics[width=.3\linewidth]{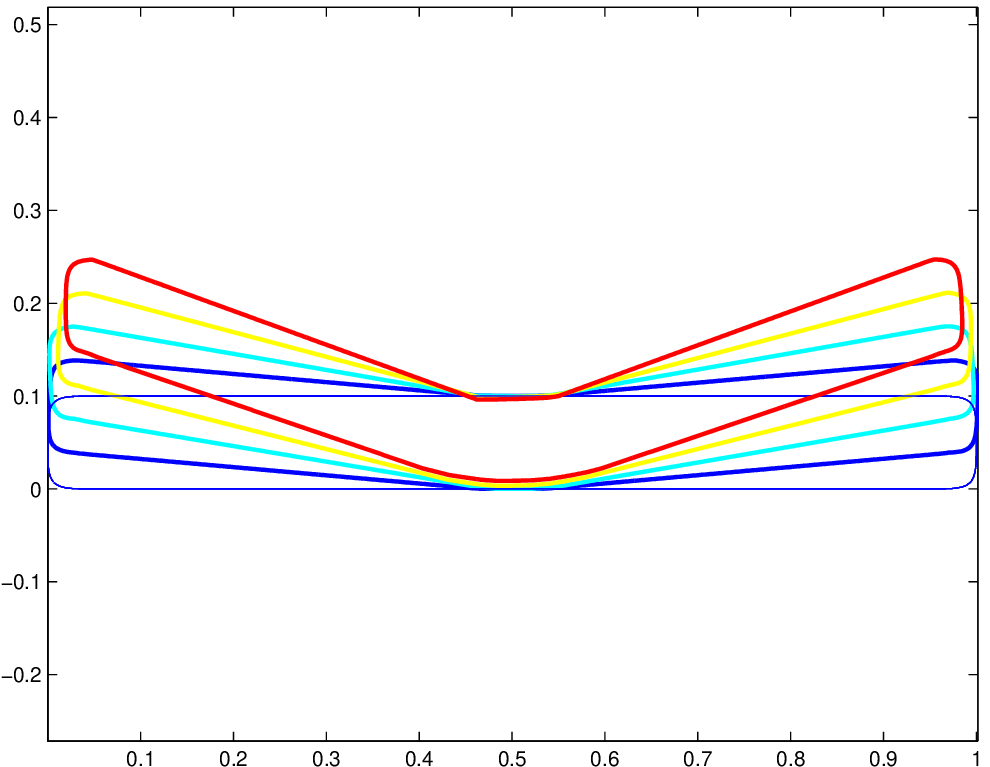}\\
$\rho=0.1$ & $\rho= 0.3$ \\
\includegraphics[width=.3\linewidth]{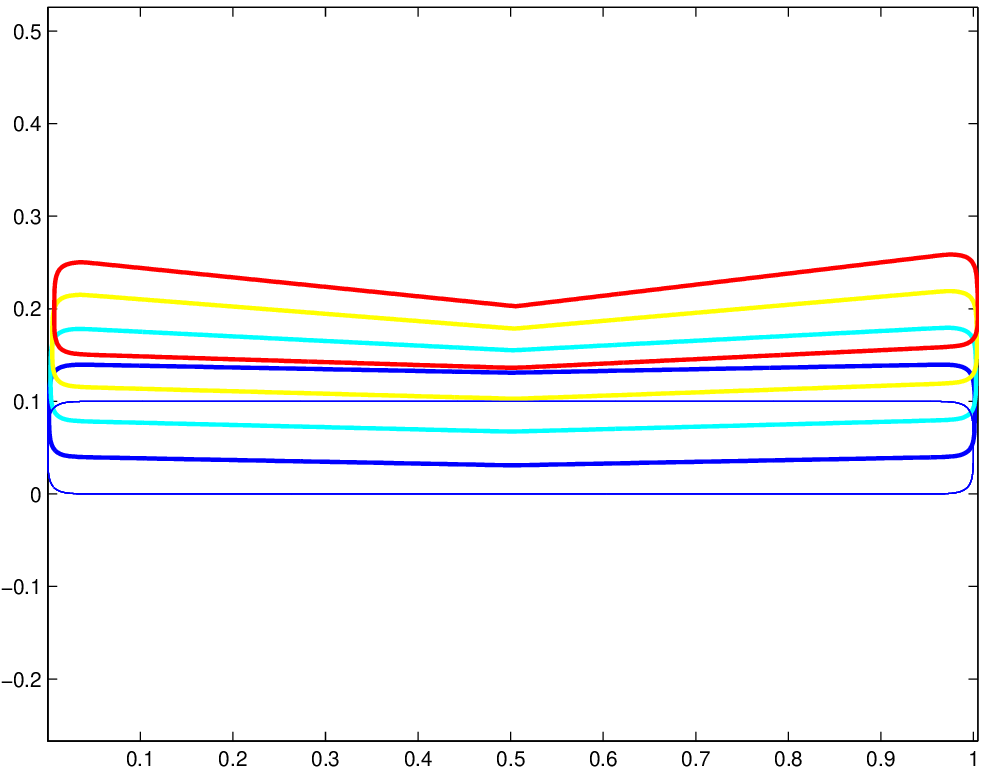}&
\includegraphics[width=.3\linewidth]{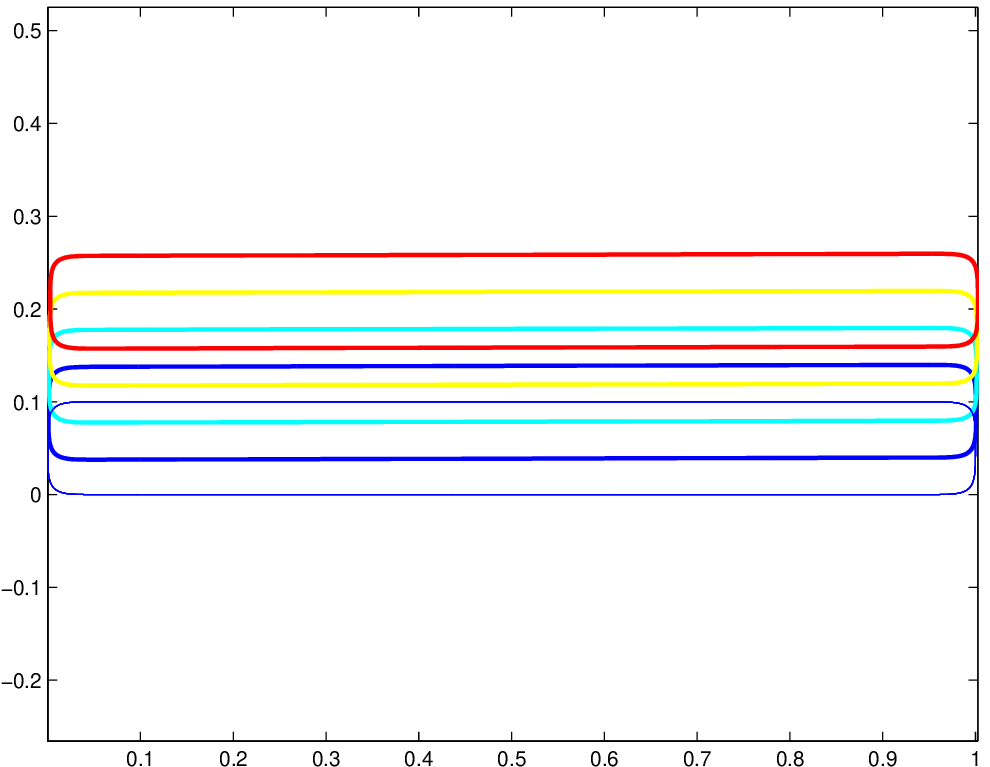} \\
  $\rho= 0.7$  & $\rho=0.9$
\end{tabular}
\end{center}
\caption{\label{fig1} Evolution for different values of $\rho$. }
\end{figure}

\subsection{Curve Registration} 
\label{sec-numerics-registration}

We now give an example of application of the Finsler flow to the curve matching problem described in Section~\ref{CM}. Figure~\ref{evolutions0} compares the results of the piecewise-rigid Finsler gradient with the Sobolev Riemannian gradient detailed in Remark~\eqref{rem-fins-sob} which is very similar to the one introduced in~\cite{sundaramoorthi-sobolev-active,charpiat-generalized-gradient}.

In order to obtain good matching results, it is important to select the parameters $(\tau,\si,\de)$ (see~\eqref{kernel} and~\eqref{eq-synthetic-flow}) in accordance to the typical size of the features of the curves to be matched.  For each method, we have manually tuned the parameters $(\tau,\si,\de)$ in order to achieve the best matching results. Choosing a large value of $\si$ and a smaller value for $\delta$ is useful to capture shapes with features at different scales, which is the case in our examples. 
 
\begin{figure}[h]
\centering
\begin{tabular}{@{}c@{\hspace{1mm}}c@{\hspace{1mm}}c@{\hspace{1mm}}c@{\hspace{1mm}}c@{}}
\includegraphics[height=.19\linewidth, angle=270]{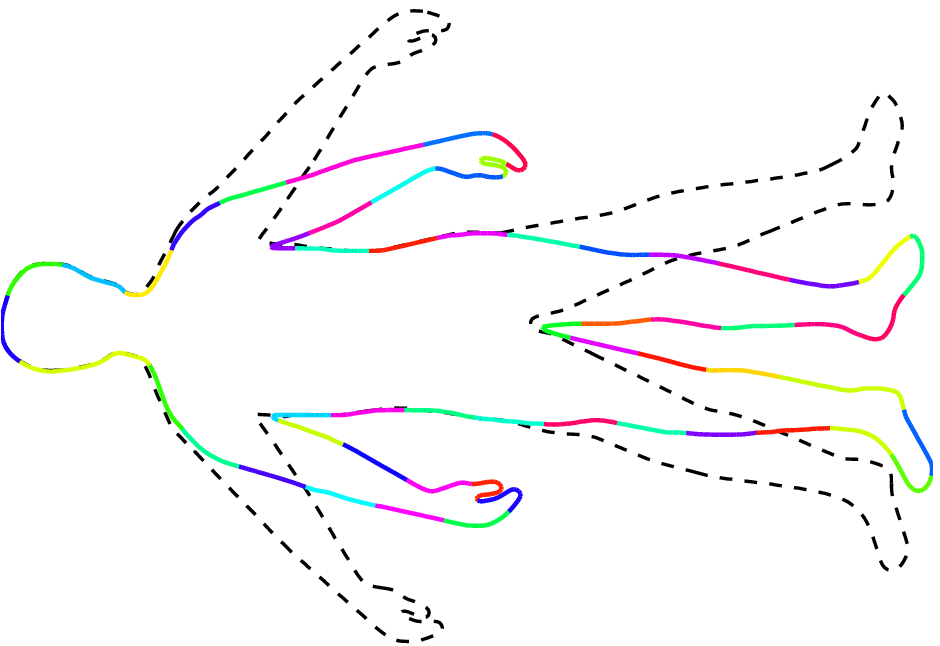}&
\includegraphics[height=.19\linewidth, angle=270]{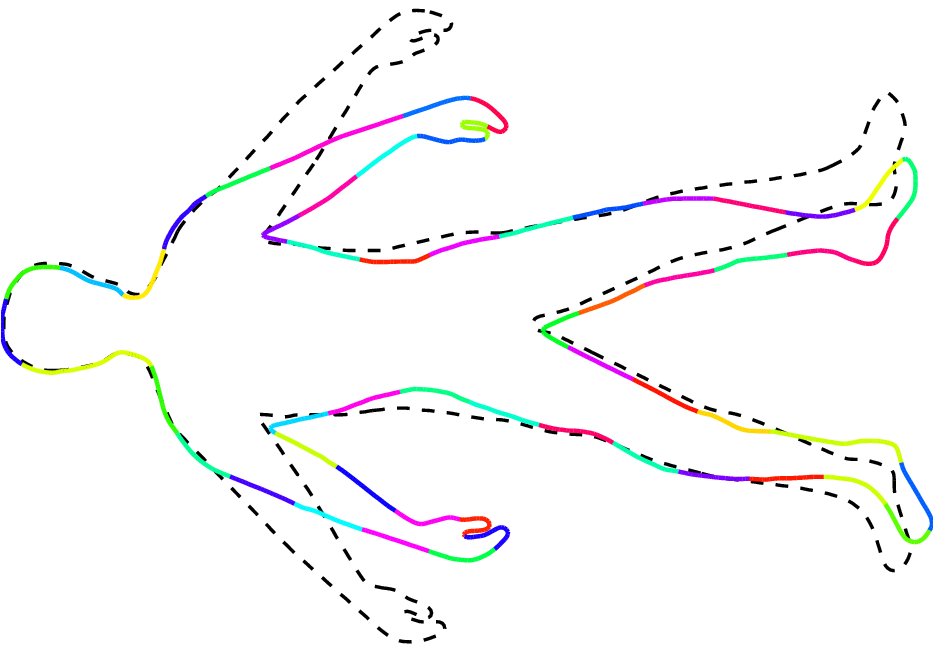}&
\includegraphics[height=.19\linewidth, angle=270]{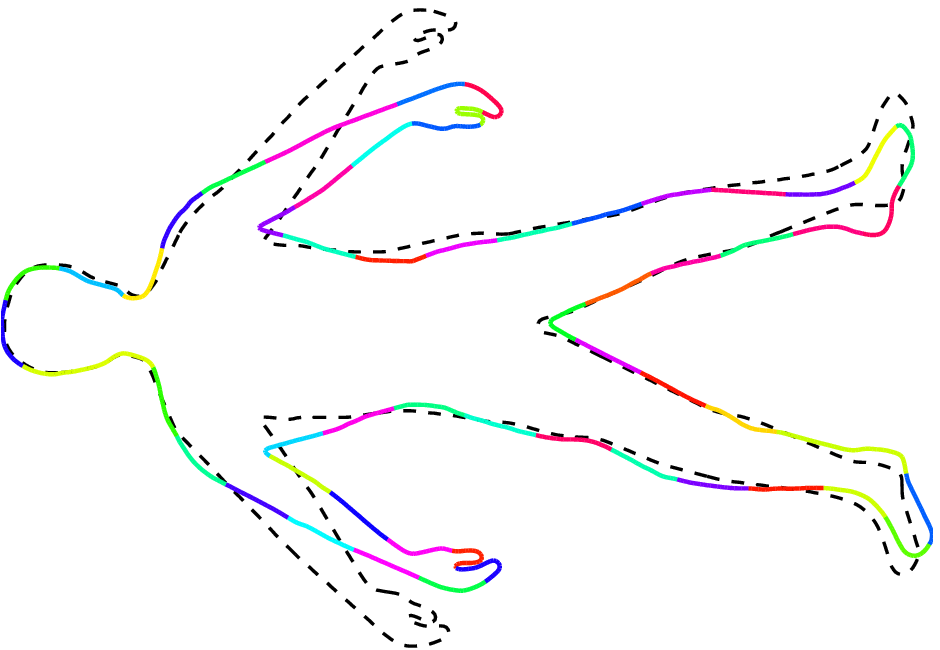}&
\includegraphics[height=.19\linewidth, angle=270]{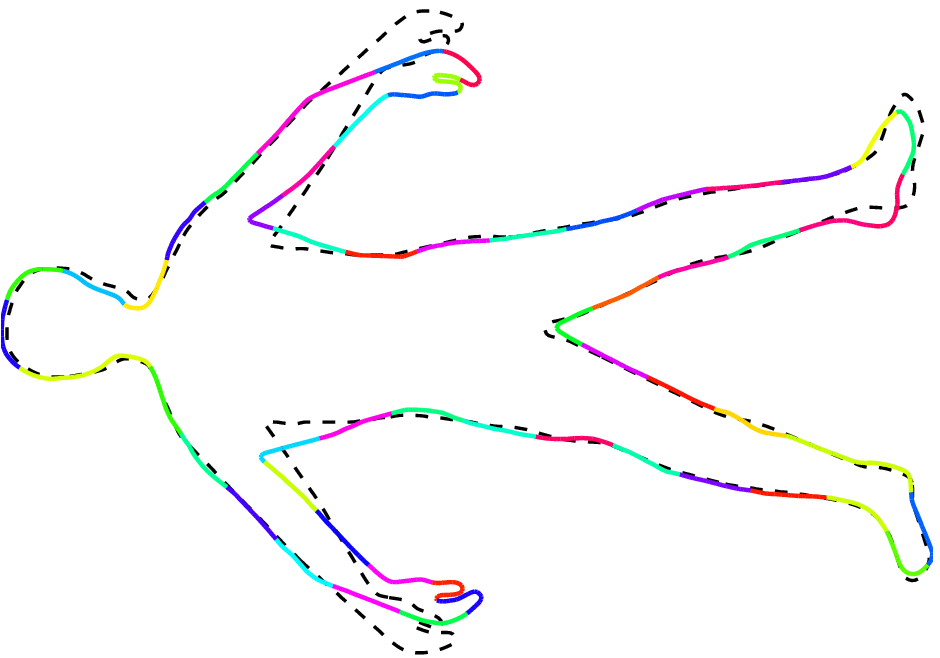}&
\includegraphics[height=.19\linewidth, angle=270]{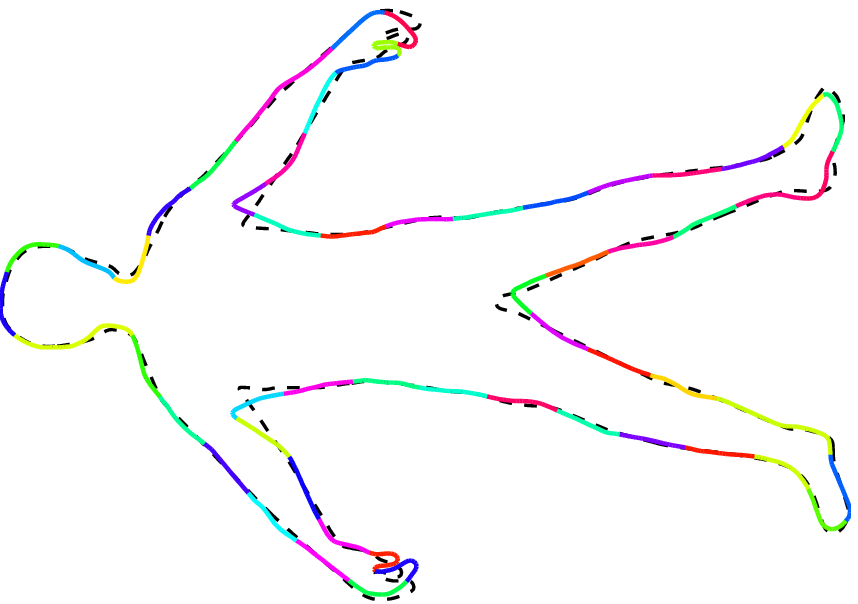}\\
\includegraphics[height=.19\linewidth, angle=270]{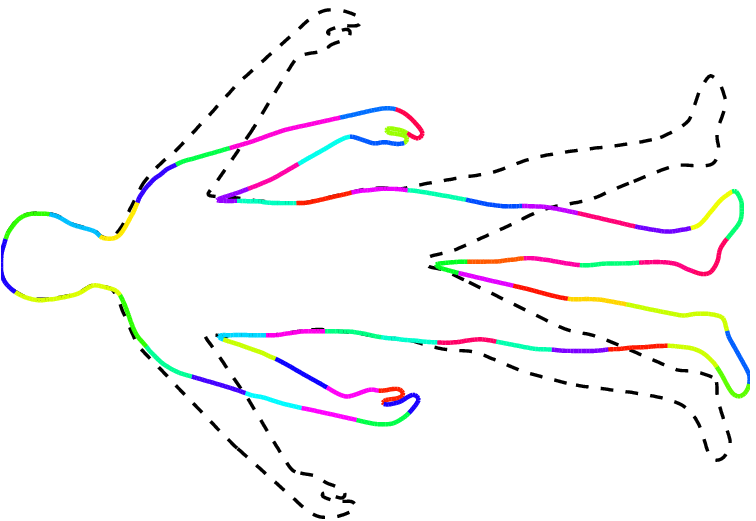}&
\includegraphics[height=.19\linewidth, angle=270]{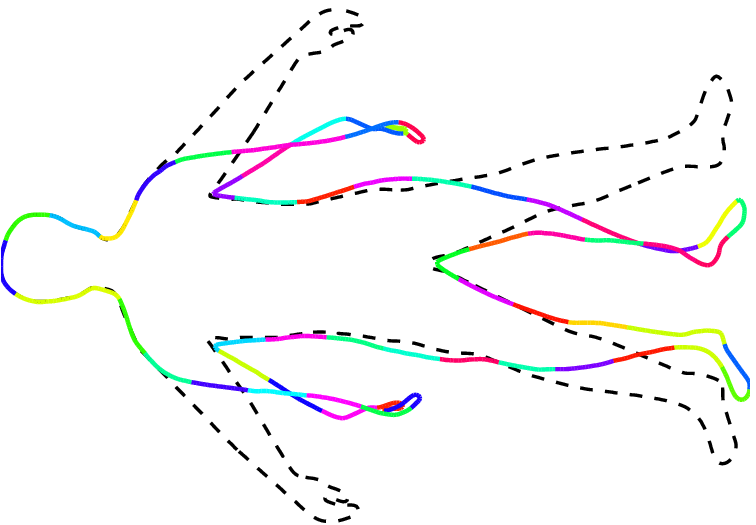}&
\includegraphics[height=.19\linewidth, angle=270]{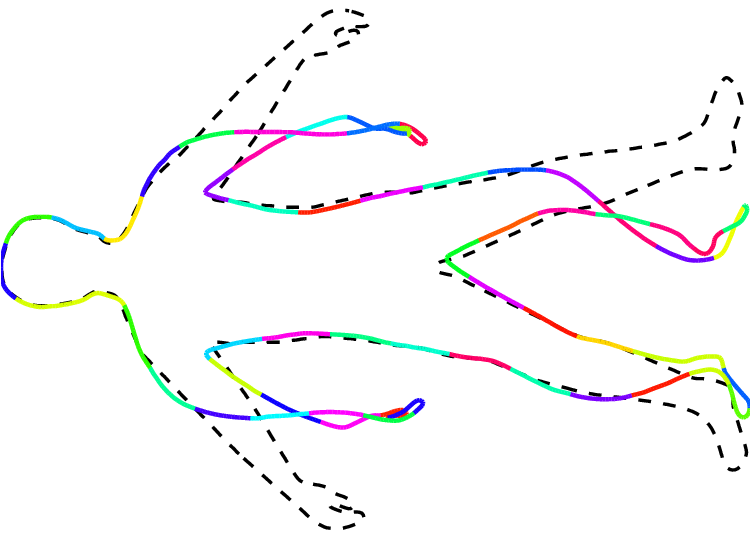}&
\includegraphics[height=.19\linewidth, angle=270]{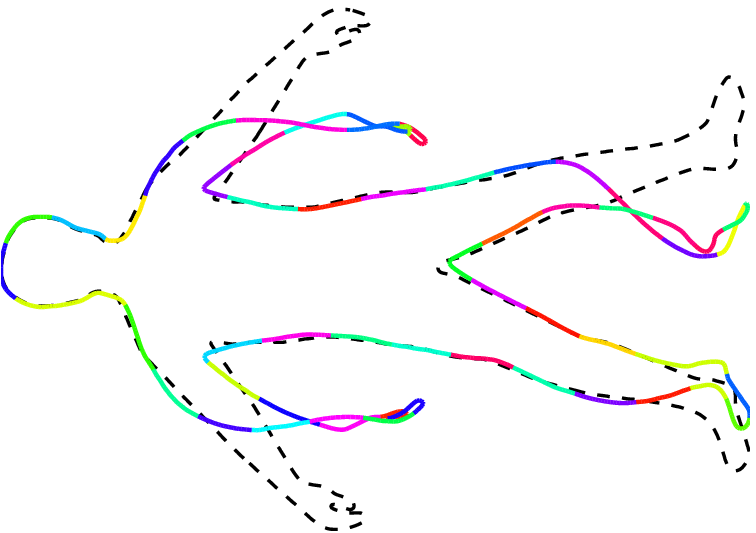}&
\includegraphics[height=.19\linewidth, angle=270]{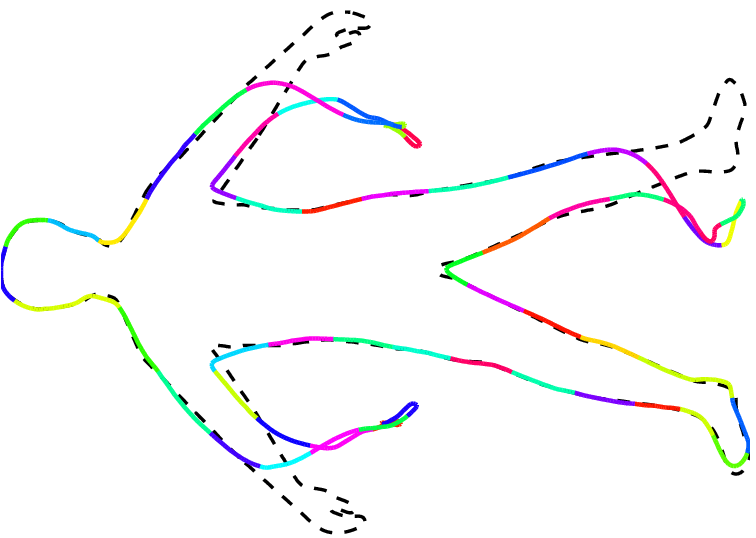}
\end{tabular}
\caption{\label{evolutions0} Finsler evolution (top) with $\rho=0.8$ and  Sobolev evolution (bottom) for different step $k$ of the flow. Each image displays the target curve $\La$ (dash line) and the current curve $\Ga_k$ (solid line). The energy is computed using $\sigma=0.8$, $\delta=0.04$. }
\vspace{0.3cm}
\end{figure}

The piecewise rigid gradient is particularly efficient in this setting where the curves to be matched are naturally obtained by approximate articulations, which are well approximated by piecewise rigid deformations. Note however that our method does not necessitate a prior segmentation of the shape into disjoint areas undergoing rigid motions, i.e. the location of the articulations does not need to be known beforehand. The piecewise rigid matching is obtained solely by minimizing the distance energy $E(\Ga)$ to the target curve $\La$.  

The Finsler gradient thus allows to avoid poor local minima and perform an overall good global matching. In contrast the Sobolev gradient flow is trapped in a poor local minimum and the matching has failed. Note, however, that the matching achieved by the Finsler gradient is not perfect. Some local defects near corners are mostly due to the strong constraint $L_\Ga(\Phi)=0$ which enforces the exact conservation of the length of the curve. This constraint is alleviated in Section~\ref{similarity}, which presents a piecewise similarity Finsler gradient which leads to better matching results. 

\section{$BV^2$-piecewise Similarity Motions}\label{similarity}

In order to improve the matching results we augment our model by allowing the curve to shrink or to lengthen during the evolution. The idea consists in considering evolution by piecewise similarity motions instead of the rigid deformations considered in Section~\ref{PWR}.  

\subsection{Similarity Curve Deformations}

We extend the rigid deformations considered in Section~\ref{GRD} to smooth evolutions $t \mapsto \Ga_t$ following a PDE~\eqref{eq-continuous-flow} that includes also a global scaling of the space. This evolution is said to obey a similarity transform if there exists a smooth function $\la : \RR \rightarrow \RR^+$ such that
\eql{\label{eq-similarity-flow}
	\foralls (s,s') \in \Circ \times \Circ, \quad
	\norm{\GA_t(s)-\GA_t(s')} = \la(t) \norm{\GA_0(s)-\GA_0(s')}.
} 
The following proposition states that the set of instantaneous motions $\Phi_t$ giving rise to a similarity evolution  is, at each time, a linear sub-space of dimension 4 of $T_{\Ga_t} \Bb$.

\begin{prop}
	The evolution~\eqref{eq-continuous-flow} satisfies~\eqref{eq-similarity-flow} if and only if, for all $t \in \RR$, $\Phi_t \in \Ss_{\GA_t}$ where  
	\eql{\label{eq:similarity}
		\Ss_\GA = \enscond{ \Phi \in T_\Ga \Bb }{
			\foralls s \in \Circ, \; \Phi(s) =  A \GA(s)  +  b:
			\;\text{for}\; 
			A \in \Ss_{2 \times 2}, \: b \in \R^2
		}
	}
	\eq{
		\qwhereq
		\Ss_{2 \times 2} = \enscond{
		\begin{pmatrix}
					\al& -\be\\
					\be&\al\\
			\end{pmatrix} \in \RR^{2 \times 2}
		}
		{ (\al,\be) \in \RR^2 }
	\,.}
\end{prop}
\begin{proof}
Using the fact that the Lie algebra of the group of similarities is $\R^2 \rtimes \Ss_{2 \times 2}$, we obtain the desired result following the proof of Proposition \ref{rigid-ch}. 
\end{proof}

Analogously to Proposition~\ref{charactrigid}, the next proposition characterizes in an intrinsic manner the set~$\Ss_\Ga$.

\begin{prop}\label{similitude}
	For a $C^2$-curve $\Ga$, one has $\Phi \in \Ss_\GA$ if and only if $\Phi$ is $C^2$ and satisfies 
\begin{equation}\label{similarity0}
		\frac{\d K_\GA(\Phi)}{\dgs} = 0
		\end{equation}
	where we have introduced the following linear operator 
	\eq{	
		K_\GA(\Phi)=\left(\frac{\d \Phi}{\d \GA}\cdot \tgam, \frac{\d \Phi}{\d \GA}\cdot \ngam \right), 
		\quad \quad\foralls \Phi \in T_\GA\Bb\,.
	}
\end{prop}

\begin{proof}
Given a curve $\GA\in C^2(\Circ,\RR^2)$, every   deformation $\Phi$ of $\Gamma$ which is the restriction, to the curve $\Gamma$,  of an instantaneous similarity motion, can be written as
\eq{
	\Phi(s) \; = \;  A \GA(s)  + b, \quad \quad \foralls s\in \Circ 
}
for some matrix $A=\begin{pmatrix}
\al& -\be\\
\be&\al\\
\end{pmatrix} $ and a vector $b$.
Now, differentiating with respect to $\d \GA$ we obtain
\eq{
	\frac{\d\Phi}{\dgs}(s)= A\tgam
}
which is equivalent to  
\begin{equation}\label{similarity-eqz}
	\frac{\d \Phi}{\dgs}  \cdot\, \tgam \; = \; 
	 \alpha 
	\qandq
	\frac{\d \Phi}{\dgs}  \cdot\, \ngam \; = \; 
	 \beta
\end{equation}
for every $s\in \Circ$. Remark that similarity  is the only affine motion verifying~\eqref{similarity-eqz} and this is due to the form of the matrix $A$. In fact, if $\frac{\d \Phi}{\dgs} $ verifies $\frac{\d \Phi}{\dgs}  \cdot\, \tgam = \alpha$ and $\frac{\d \Phi}{\dgs}  \cdot\, \ngam= \beta$ then
$$\frac{\d \Phi}{\dgs} = \alpha \tgam + \beta \ngam = \alpha \tgam + \beta \tgam^\bot = A\tgam\,.$$ 
In particular, if $\alpha = 0$ then $\Phi$ is a rigid motion and~\eqref{similarity-eqz} coincides with the characterization proved in Proposition~\ref{charactrigid}.

Then, differentiating again with respect to $d \GA(s)$, we have
\eq{
	\frac{\d}{\dgs}\left(\frac{\d \Phi}{\dgs}  \cdot\, \tgam\right) \; = \; 0
	\qandq
	\frac{\d}{\dgs}\left(\frac{\d \Phi}{\dgs}  \cdot\, \ngam\right) \; = \; 0
}
which is equivalent to~\eqref{similarity0}.
\end{proof}

\subsection{Piecewise Similarity Deformations}

Similarly to the Finsler penalty introduced in~\ref{BV2M}, we define a penalty that favors piecewise similarity transformations  by minimizing  the $L^1$-norm of the first derivative of $K_\GA$. To control the piecewise rigid transformation part, we relax the equality constraint $L_\GA(\Phi)=0$ defined as $\Cc_\Ga$ in~\eqref{eq-constr-C-Ga} to  a constraint $\Cc_\Ga^\la$ on the $L^1$ norm of $L_\GA(\Phi)$.

\begin{defn}[\BoxTitle{Piecewise-similarity penalty}]\label{defn2} For $\la \geq 0$ and $\Ga \in \Bb$, we define for all $\Phi \in T_\Ga\Bb$
\begin{equation}\label{eq-piecewise-similarity-penalty}
	R_\Ga^\la(\Phi) =
	TV_\GA(K_\GA(\Phi) )
	+ \iota_{\Cc_\Ga^\la}
	\qwhereq
	\Cc_\Ga^\la = \enscond{ \Phi \in T_\GA \Bb }{
			\norm{L_\Ga(\Phi)}_{L^1(\Ga)}\leq \la }
\end{equation}
where $L_\Ga$ is either $L_\Ga^+$ or $L_\Ga^-$ as defined in~\eqref{eq-operator-L} and $TV_\GA$ is defined in~\eqref{TV}.
\end{defn}

The piecewise similarity Finsler gradient $\nabla_{R_\Ga^\la}E(\Ga)$ is defined by minimizing~\eqref{defgrad} with the penalty $R_\Ga^\la$ defined in~\eqref{eq-piecewise-similarity-penalty} with the constraint set $\Ll_\Ga$ defined in~\eqref{eq-dev-constr}.  The following proposition shows that, as $\la$ tends to 0, the set of piecewise similarity Finsler gradients tends to the set of piecewise-rigid Finsler gradients. 

\begin{prop}
	One has $R_\Ga^0 = R_\Ga$ where $R_\Ga$ is defined in~\eqref{eq-piecewise-rigid-penalty}.
\end{prop}
\begin{proof}
One has $\Cc^0_\Ga = \Cc_\Ga$.
If $R_\Ga^0(\Phi) \neq +\infty$, one has $L_\GA^{+}(\Phi)=L_\GA^{-}(\Phi)=0$ a.e., so that in this case 
\eq{
	TV_\GA( K_\Ga(\Phi) )= 
	TV_\GA\left(\frac{\d \Phi}{\d \GA(s)}\cdot \ngam \right)\,.}
\end{proof}
The following theorem extends Theorem~\ref{existence1} to the piecewise similarity penalty and ensures existence of the corresponding Finsler gradient.

\begin{thm} 
The function $R_\GA^\la$ defined in~\eqref{eq-piecewise-similarity-penalty} admits at least a minimum on $\Ll_\GA$. 
\end{thm}

\begin{proof} 
	It suffices to adapt the proof of Theorem~\ref{existence1} by  using the new constraint on the $L^1$-norm of $L_\Ga(\Phi)$.

\end{proof}

\subsection{Numerical examples}\label{NE}

We now show some numerical examples for the piecewise similarity Finsler gradient. The computation is performed with the discretization detailed in Section~\eqref{discretization}, which is extended in a straightforward manner to handle the piecewise similarity model.

\begin{figure}[!h]
\centering
\begin{tabular}{@{}c@{\hspace{5mm}}c@{}}
\includegraphics[width=.3\linewidth]{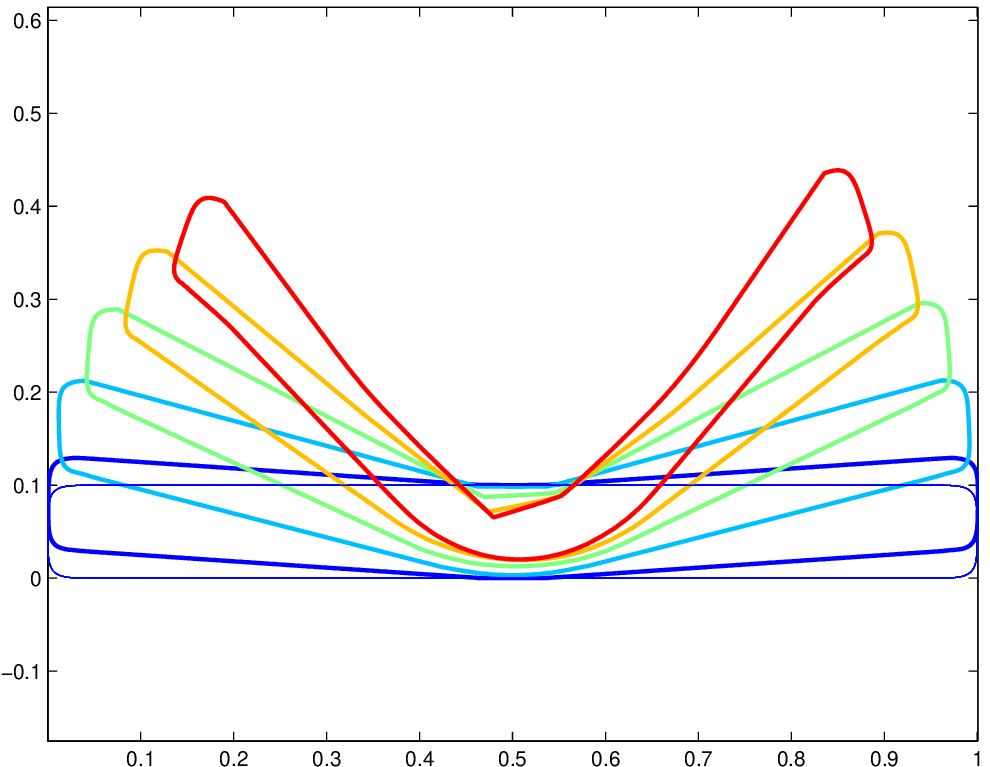}&
\includegraphics[width=.3\linewidth]{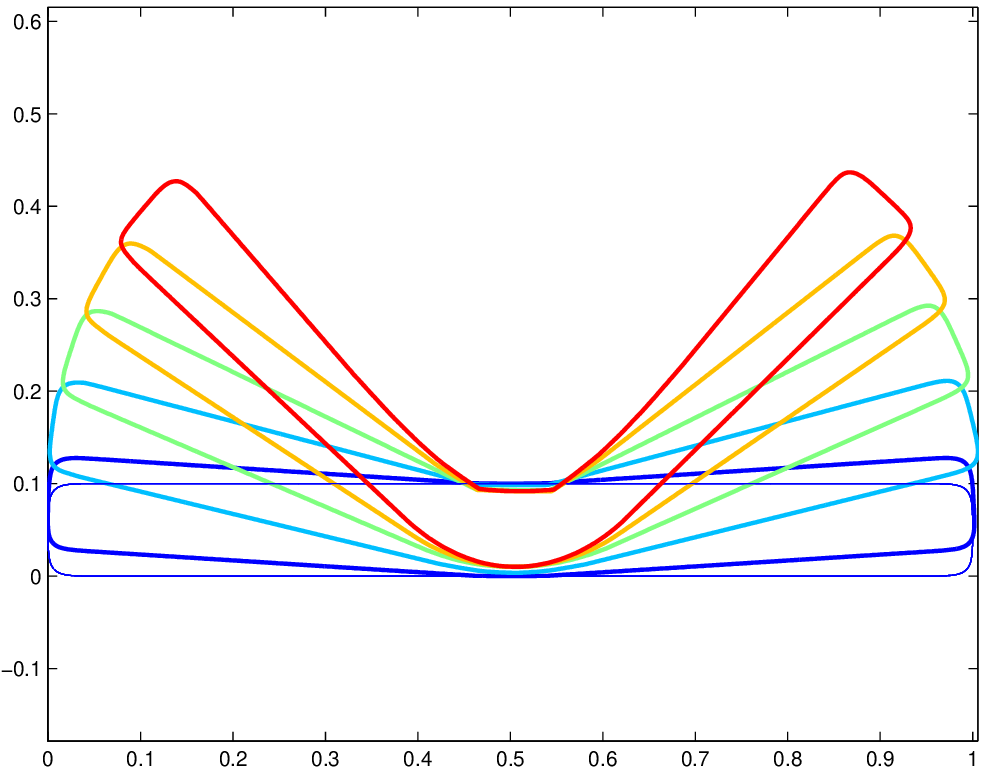}\\
$\la=.01$ & $\la=50$ \\
\includegraphics[width=.3\linewidth]{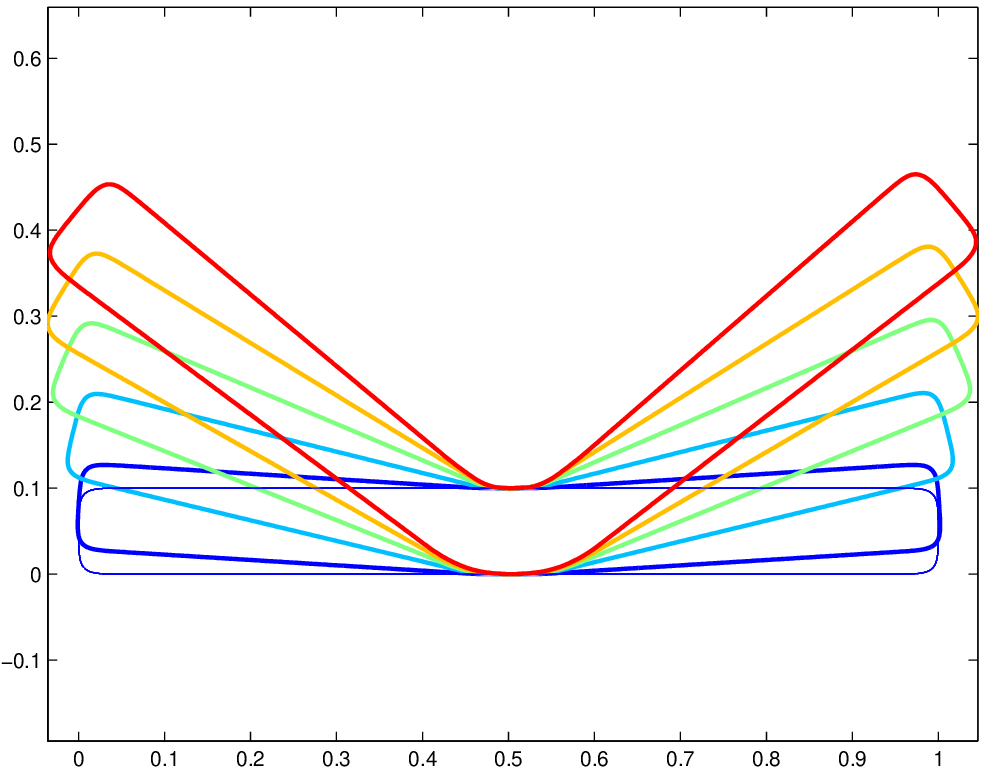}&
\includegraphics[width=.3\linewidth]{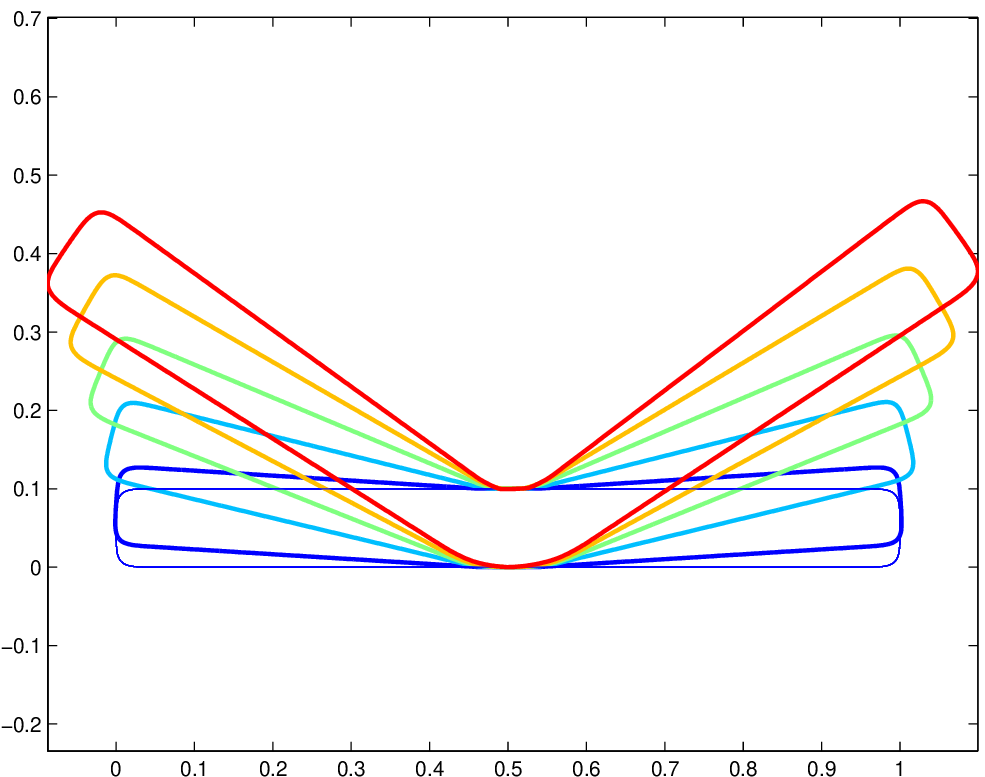}\\
 $\la=300$ & $\la=2000$
\end{tabular}
\caption{ Piecewise similarity Finsler flow evolutions for $\rho = 0.3$ and for different values of $\la$. }\label{evolutions2}
\end{figure}

\paragraph{Influence of  $\la$. }

We first re-use the synthetic example introduced in Section~\ref{sec-influ-rho} to illustrate the influence of the parameter $\la$. We thus use an evolution driven by the flow $F(\Ga) \in T_\Ga \Bb$ defined in~\eqref{field-const}. Figure~\ref{evolutions2} shows how $\la$ allows one to interpolate between the piecewise rigid model (when $\la=0$) to a piecewise similarity model when $\la$ increases.  For large value of $\la$, one clearly sees the global scaling introduced by the model which is helpful to better follow the flow of $F$.

Figure~\ref{evolutions3} compares the evolution obtained with the initial flow~\eqref{eq-synthetic-flow} (corresponding to $(\rho,\la)=(0,0)$, i.e. the Finsler gradient is equal to $F(\Ga)$), the piecewise rigid flow (corresponding to $\rho>0$ and $\la=0$) and the piecewise similarity flow (corresponding to $\rho>0$ and $\la>0$).
 
\begin{figure}[!h]
\centering
\begin{tabular}{@{}c@{\hspace{3mm}}c@{\hspace{3mm}}c@{}}
\includegraphics[width=.25\linewidth]{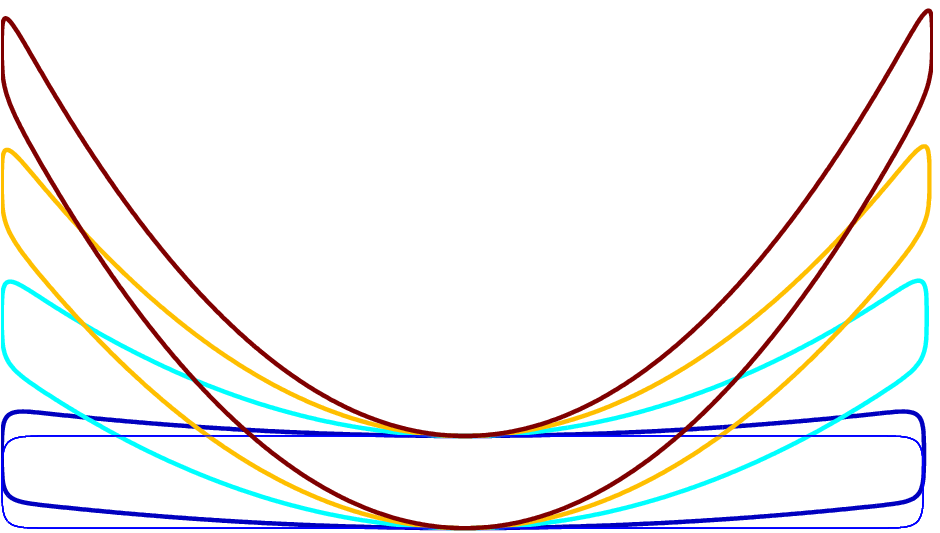}&
\includegraphics[width=.25\linewidth]{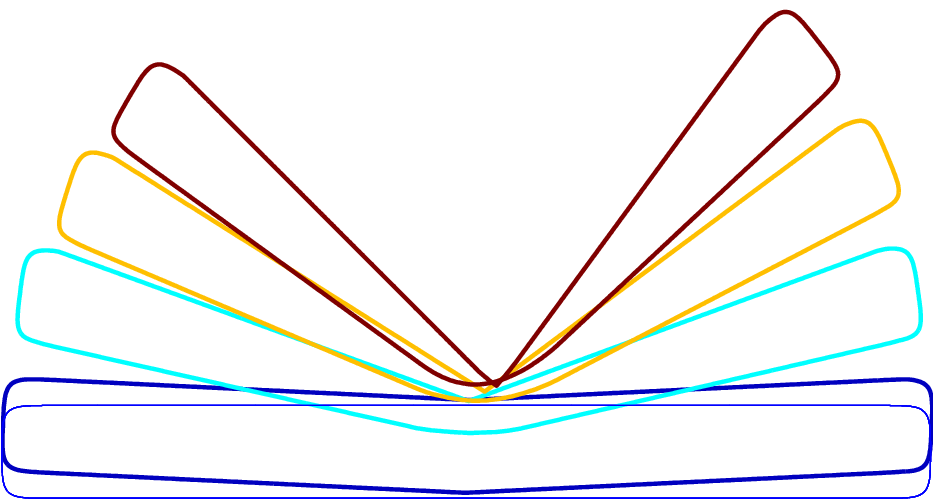}&
\includegraphics[width=.33\linewidth]{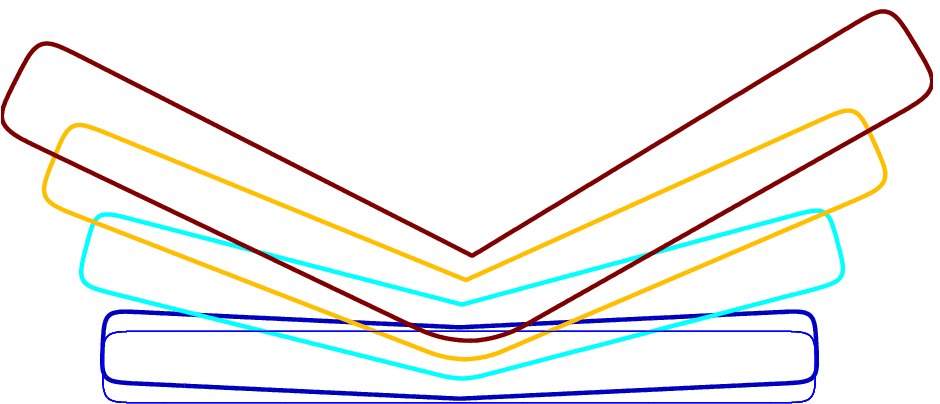} \\
$L^2$ & $(\rho,\la)=(0.5,0)$ & $(\rho,\la)=(0.5,200)$ 
\end{tabular}
\caption{From left to right: evolutions by using the $L^2$ gradient, piecewise rigid Finsler gradient, piecewise similarity Finsler gradient. }\label{evolutions3}
\end{figure}

\paragraph{Application to the matching problem.}
 
We now show an application of the Finsler descent method to the curve matching problem, by minimizing the energy $E$ defined in~\eqref{energy}. Figure~\ref{evolutions} shows the results obtained with the piecewise similarity penalty for well chosen parameters $(\si,\de,\la,\rho)$. These evolutions should be compared with the ones reported in Section~\ref{sec-numerics-registration}. Allowing the length of the curve to vary using a parameter $\la>0$ allows the evolutions to better capture the geometry of the target shape and thus leads to better matchings.

\begin{figure}[h]
\centering
\begin{tabular}{@{}c@{\hspace{1mm}}c@{\hspace{1mm}}c@{\hspace{1mm}}c@{\hspace{1mm}}c@{}}
\includegraphics[width=.19\linewidth]{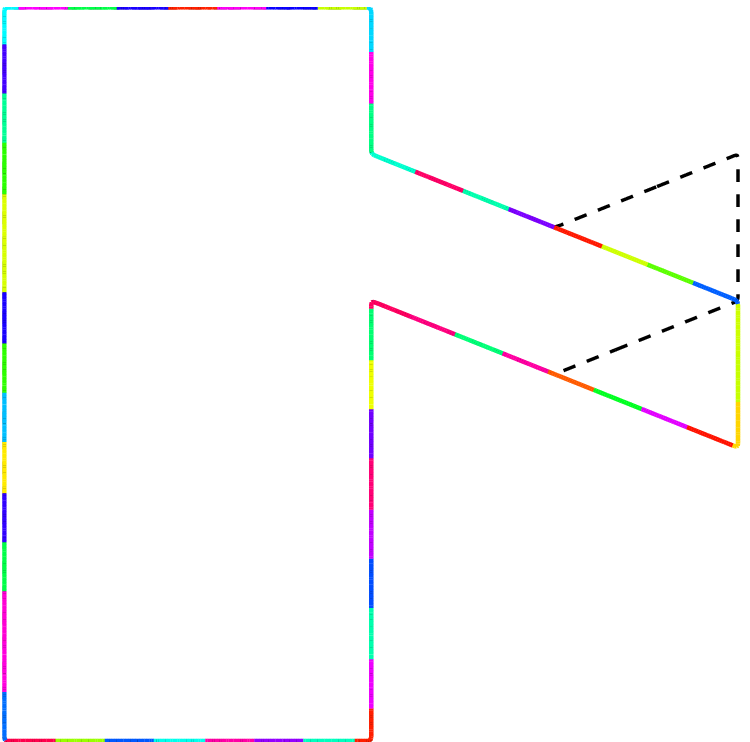}&
\includegraphics[width=.19\linewidth]{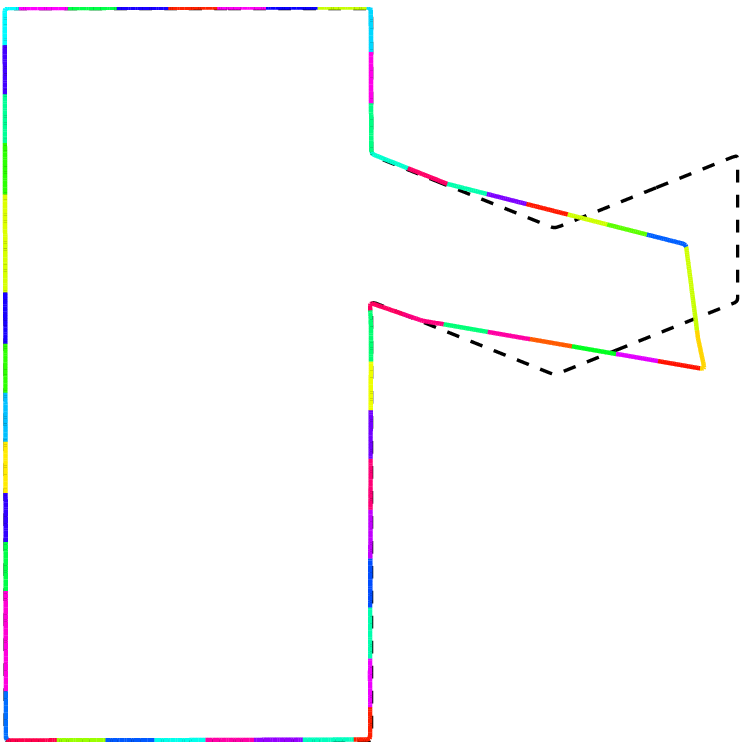}&
\includegraphics[width=.19\linewidth]{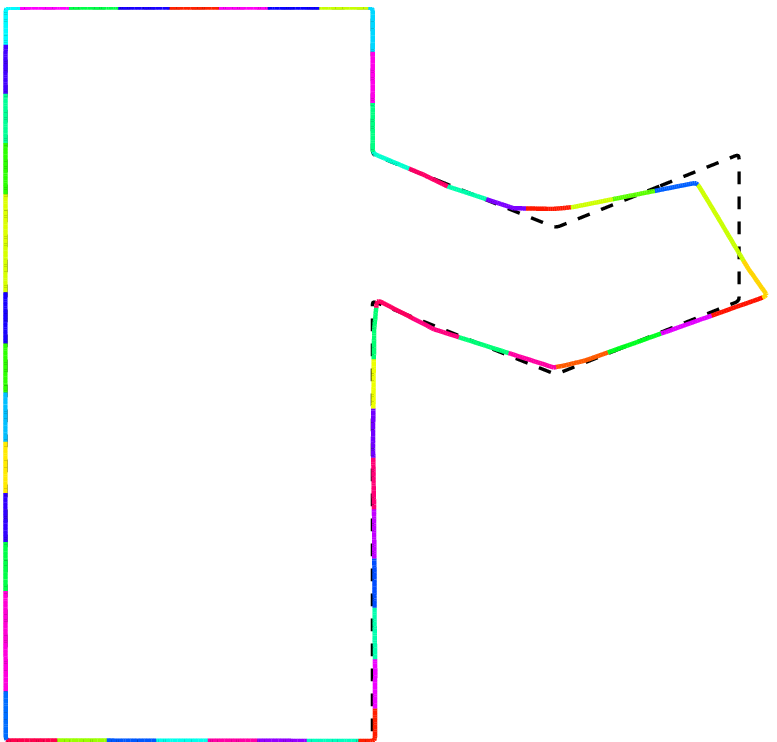}&
\includegraphics[width=.19\linewidth]{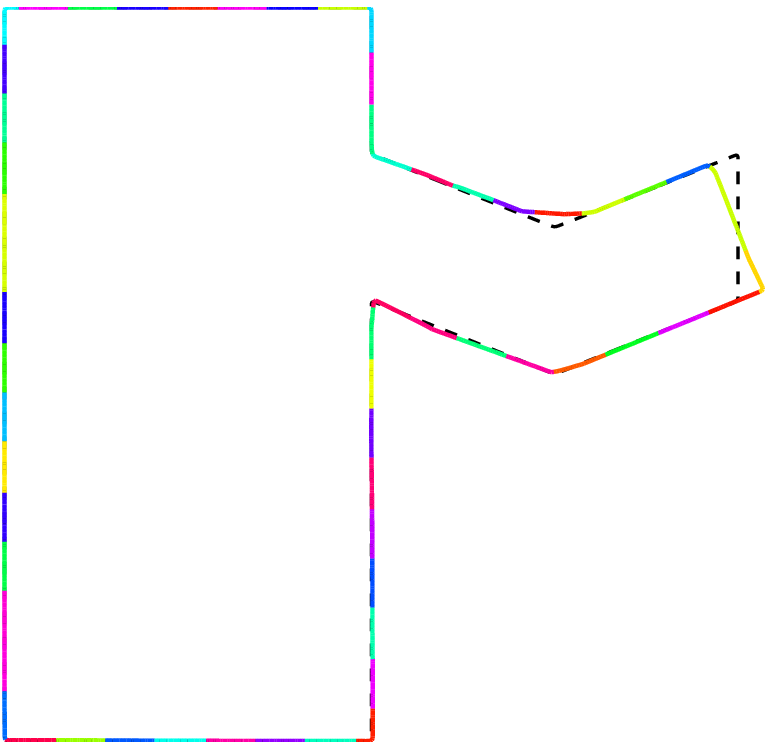}&
\includegraphics[width=.19\linewidth]{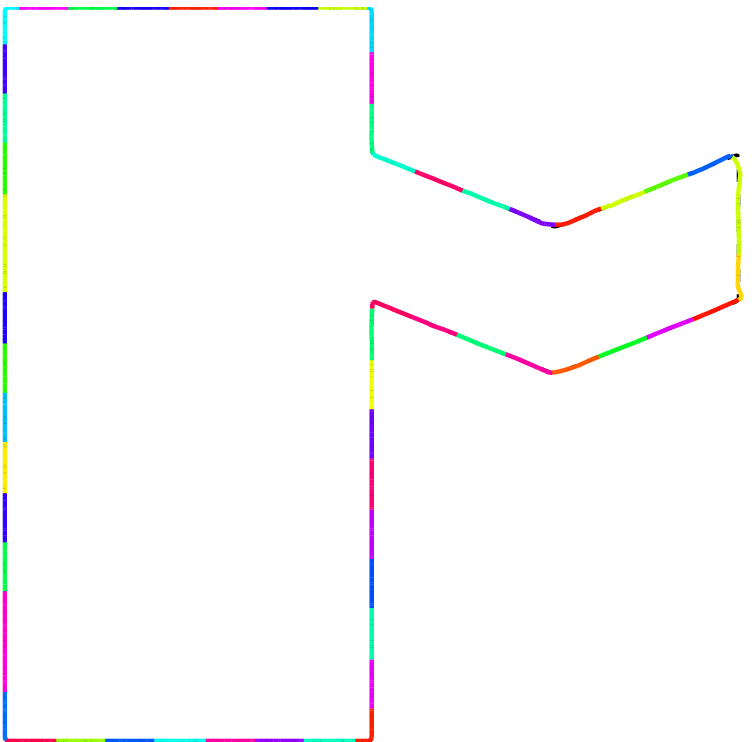}\\
\includegraphics[height=.19\linewidth,angle=270]{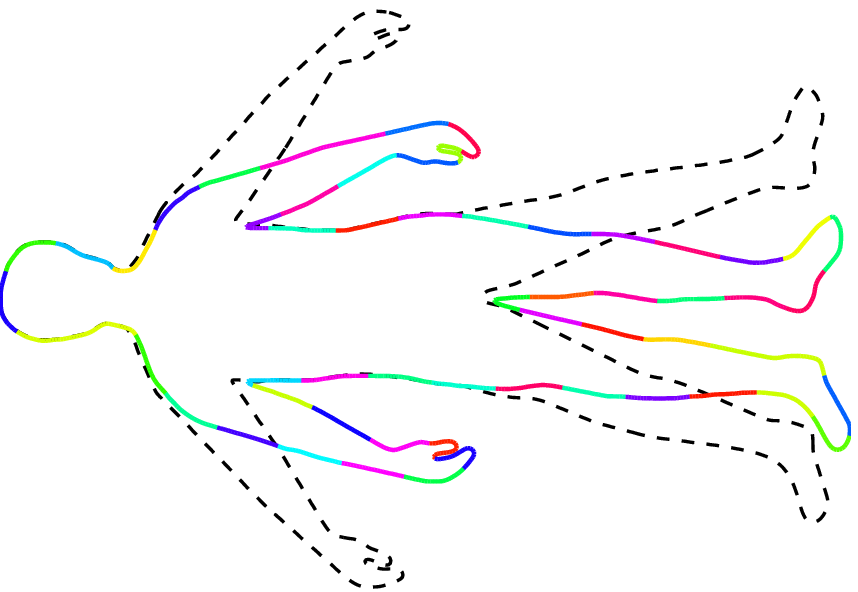}&
\includegraphics[height=.19\linewidth,angle=270]{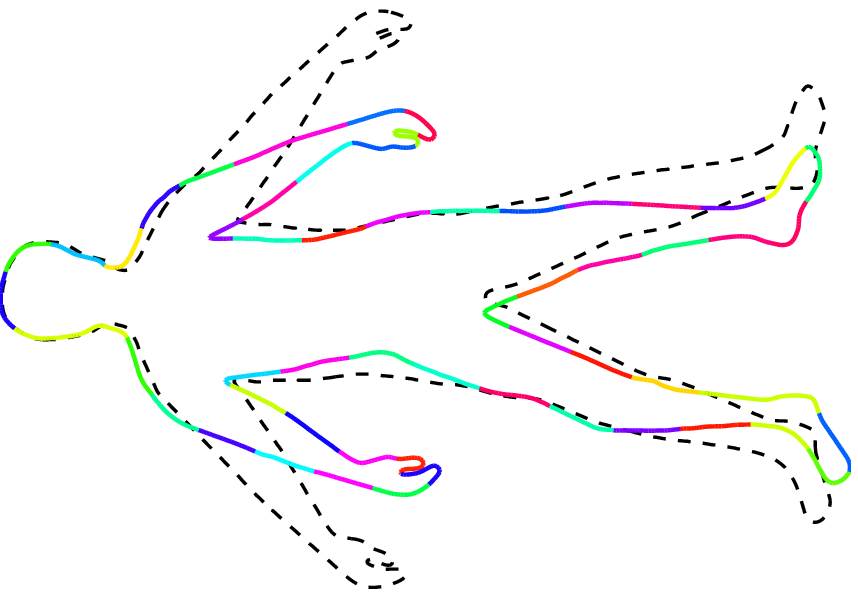}&
\includegraphics[height=.19\linewidth,angle=270]{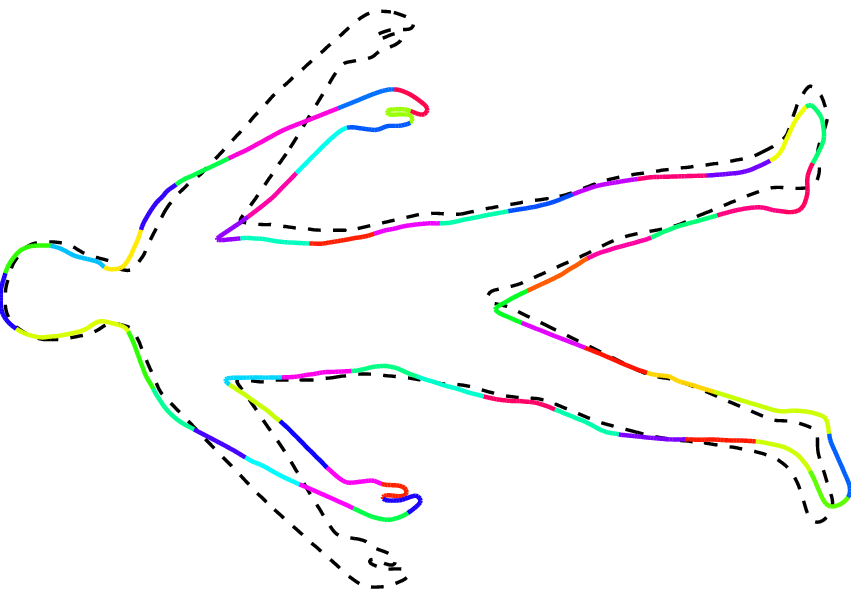}&
\includegraphics[height=.19\linewidth,angle=270]{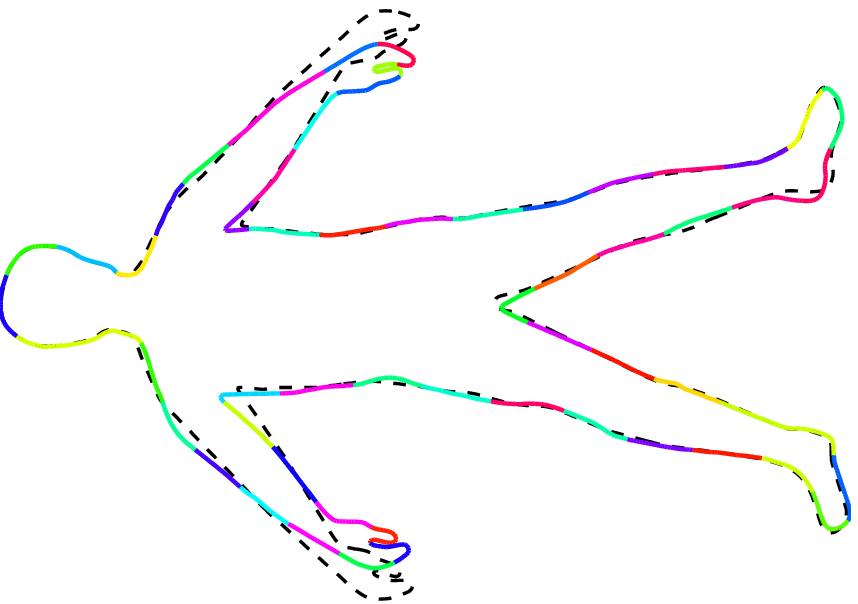}&
\includegraphics[height=.19\linewidth,angle=270]{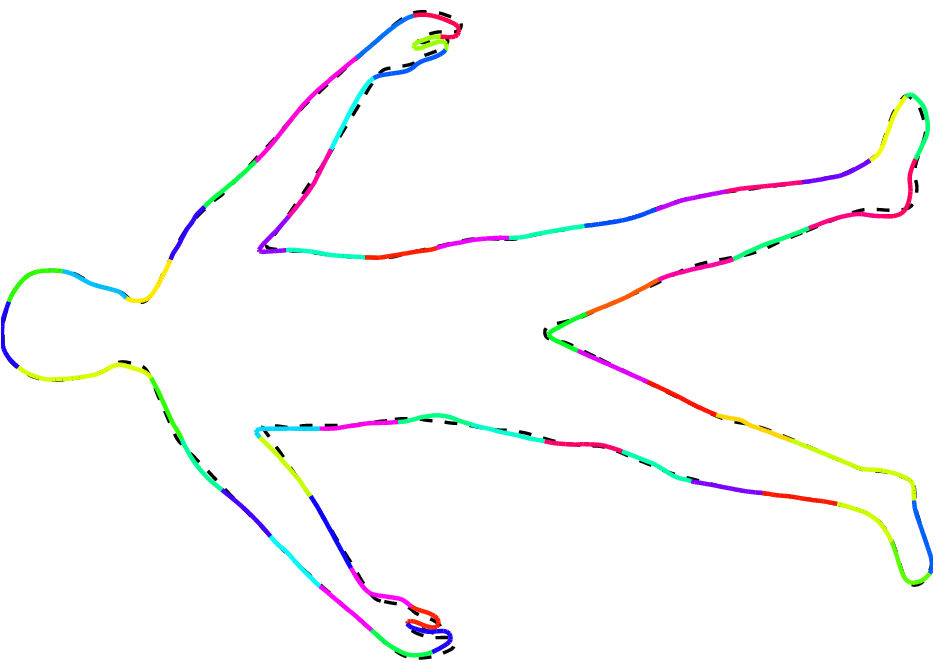}\\
\includegraphics[height=.19\linewidth,angle=270]{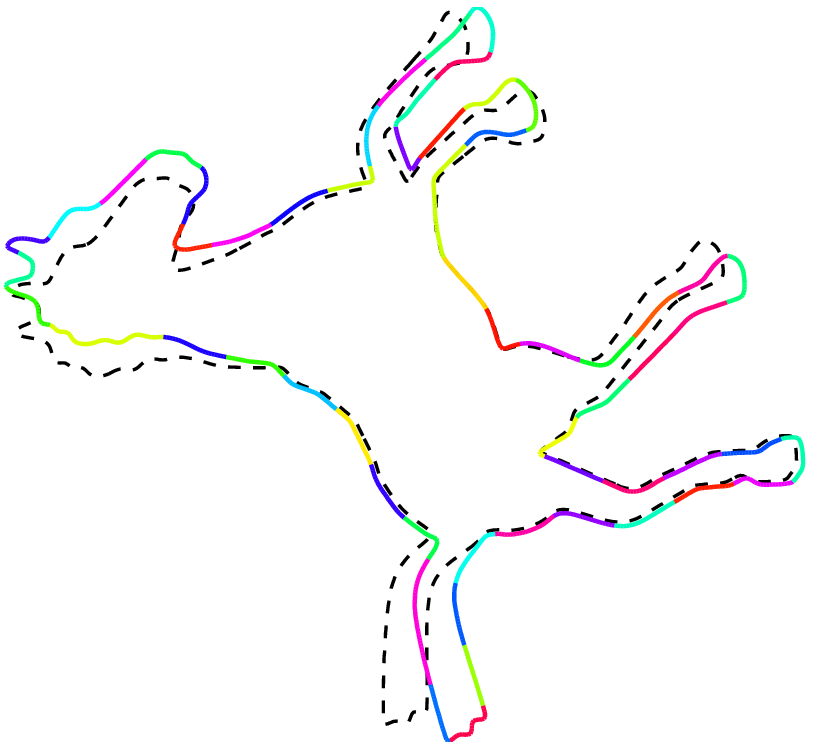}&
\includegraphics[height=.19\linewidth,angle=270]{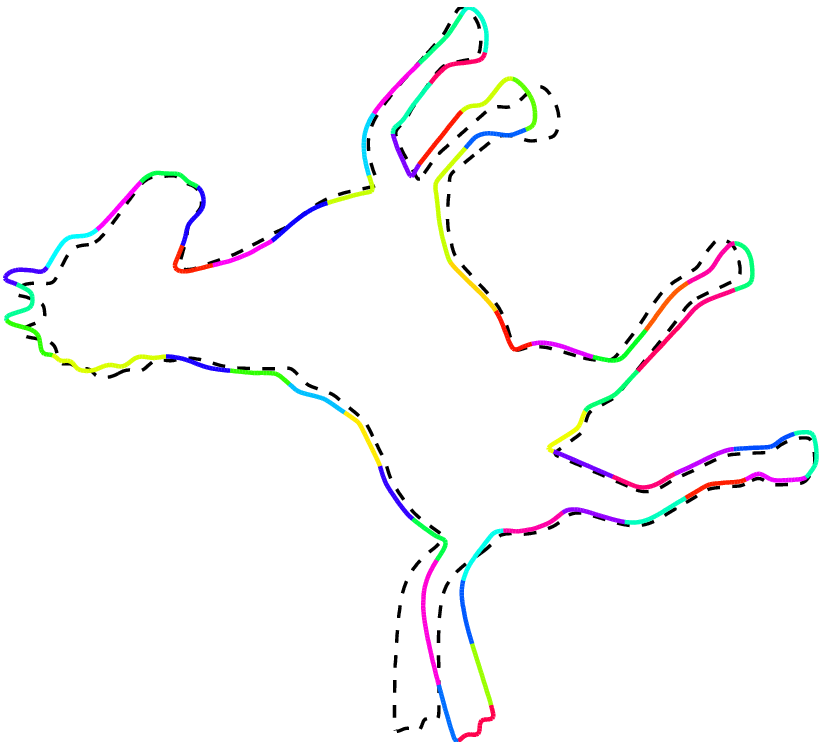}&
\includegraphics[height=.19\linewidth,angle=270]{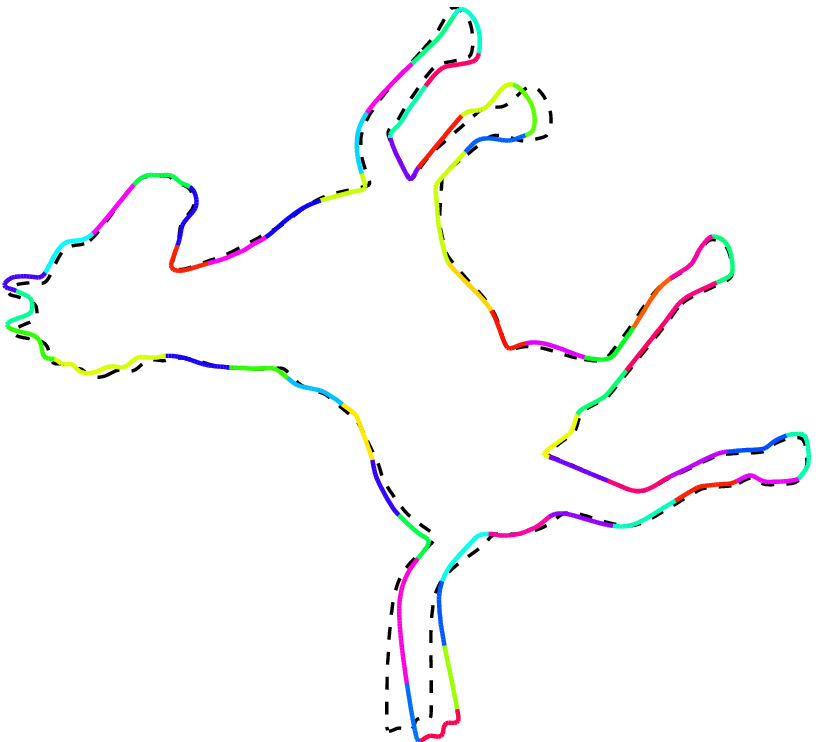}&
\includegraphics[height=.19\linewidth,angle=270]{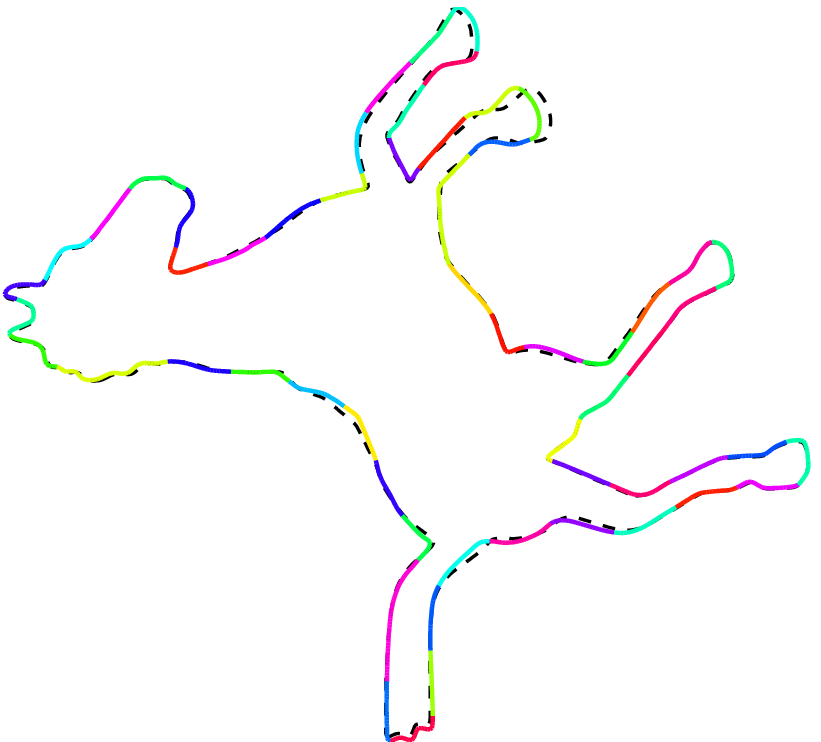}&
\includegraphics[height=.19\linewidth,angle=270]{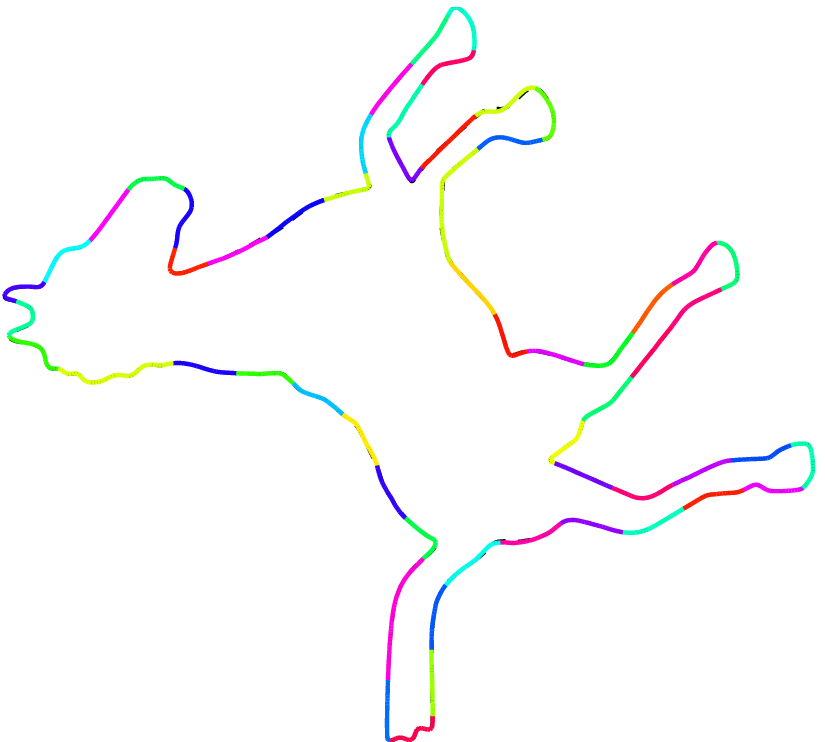}
\end{tabular}
\caption{\label{evolutions} Curve matching by piecewise similarity motions. Each image displays the target curve $\La$ (dash line) and the current curve $\Ga_k$ (solid line). We used the following parameters: 
	top row: $\sigma=0.8$, $\delta=0.04$, $\la=2000$, $\rho=0.85$ ; 
	middle row: $\sigma=0.8$, $\delta=0.08$, $\la=2000$, $\rho=0.95$ ; 
	bottom row: $\sigma=0.9$, $\delta=0.03$, $\la=2000$, $\rho=0.87$.\vspace{0.5cm} 
}
\end{figure}

\section*{Conclusion}

This paper has presented a novel way to encode piecewise regular constraints for curve evolutions. This is achieved by designing Finsler penalties in the tangent space of curves. This method offers a unifying treatment of this class of evolutions. A distinctive feature of this approach is that it uses a convex modeling of the geometric constraints. For the particular case of piecewise rigid and piecewise similarity transforms, this avoids the need to specify	the location of the articular points, which is a difficult problem. Instead, these articulations are obtained, at each iteration, through the resolution of a convex program. This novel method opens the doors to many fascinating theoretical questions, such as the definition of a continous flow when letting $\tau_k \rightarrow 0$, and the properties of Finslerian spaces of curves. 

\section*{Acknowledgment}

This work has been supported by the European Research Council (ERC project SIGMA-Vision).

\section{Appendix: $BV$ and $BV^2$ functions}\label{appendix}

In this section we remind the definition of $BV$ and $BV^2$ functions in dimension one.

\begin{defn}
 Let $u\in L^1([0,1],\RR)$. We say that u is a function of bounded variation in $[0,1]$ if
\begin{equation}\label{bv1}
|D u|([0,1])=\sup \enscond{
 	\int_0^1u\,g'\, \d x 
	}{
		g\in \mathrm{C}^\infty_c([0,1],\RR),\|g\|_{L^\infty([0,1],\RR)}\leq 1
	}< \infty\,.
\end{equation}
By  Riesz's representation theorem this is equivalent to state that there exists a unique finite Radon measure, denoted by $D u$, such that 
$$\int_0^1u\,g'\, \d x = - \int_0^1g\, \d D u \quad \foralls g\in \mathrm{C}^1_c({[0,1]})\,.$$
Clearly the total variation of the measure $D u$ on $[0,1]$, i.e., $|D u|([0,1])$, coincides with the quantity defined in~\eqref{bv1} and this justifies our notations. 
We denote the space of functions of bounded variation in $[0,1]$ by $\BV([0,1],\RR)$. 
The space $\BV([0,1],\RR)$ equipped with the norm 
$$\|u\|_{\BV} = \|u\|_{L^1} + |D u |([0,1])$$
is a Banach space. We say that $\{u_h\}$  weakly* converges in $\BV([0,1],\RR)$ to $u$ if
$$u_h \overset{L^1}{\longrightarrow} u \quad\mbox{and}\quad D u_h \overset{*}{\rightharpoonup}D u\,, \quad\mbox{as}\quad h \rightarrow \infty\,.$$
\end{defn}

We now define the set of $BV^2$-functions as the functions whose second derivative are Radon measures:
\begin{defn}
 Let $u\in W^{1,1}([0,1],\RR)$. 
 We say that $u$ belongs to $BV^2([0,1],\RR)$ if 
 \begin{equation}\label{2var}
 |D^2 u|([0,1]):=\sup \enscond{
 		\int_0^1u\,g''\, \d x 
	}{
		g\in \mathrm{C}^\infty_c({[0,1]},\mathbb{R} ),\|g\|_{L^{\infty}([0,1],\mathbb{R} )}\leq 1
	}
	< \infty\,.
\end{equation}
As for the first variation, the functional considered in \eqref{2var} can be represented by a measure $D^2u$ whose total variation coincides with  the quantity $|D^2 u|([0,1])$ previously defined.
\par $\BV^2([0,1],\RR)$ equipped with the norm 
\begin{equation}\label{normbv2}
\|u\|_{\BV^2} = \|u\|_{BV} + |D^2u |([0,1])
\end{equation}
is a Banach space. In particular we have $W^{2,1}([0,1],\RR)\subset BV^2([0,1],\RR)$.
 We say that $\{u_h\}$  weakly* converges in $\BV^2([0,1],\RR)$ to $u$ if
$$u_h \overset{W^{1,1}}{\longrightarrow} u \quad\mbox{and}\quad D^2 u_h \overset{*}{\rightharpoonup}D^2 u\,, \quad\mbox{as}\quad h \rightarrow \infty\,.$$
\end{defn}
We remind that if $\{u_h\}\subset BV^2([0,1],\RR)$ is such that  $\underset{h}{\sup}\,\|u_h\|_{BV^2}< M$ then there exists $u\in  BV^2([0,1],\RR)$ and a subsequence (not relabeled)  $\{u_h\}$  that weakly* converges in $\BV^2([0,1],\RR)$ toward $u$ and
$$|D^2u |([0,1]) \leq \, \underset{h \rightarrow \infty}{\liminf}\, |D^2 u_h |([0,1])\,.$$

Moreover we have the following proposition showing the link between $BV$ and $BV^2$ functions\begin{prop} \label{bv-bv2} A function $u$ belongs to $ BV^2([0,1],\RR)$ if and only if $u\in W^{1,1}([0,1],\RR)$ and $u' \in BV([0,1],\RR)$, for every $i=1,...,n$. Moreover
$$|D^2u|([0,1]) = \left|D u'\right|([0,1])\,.$$\end{prop}

\par We also remind that  
 $BV^2([0,1],\RR)$ is embedded in $ W^{1,\infty}([0,1],\RR)$ so $BV^2$ functions are  Lipschitz continuous (see Theorem 5,~\cite{EG} pag. 131). Then,  as $[0,1]\subset \R$ is bounded, $BV^2([0,1],\RR)$ is embedded $W^{1,p}([0,1],\RR)$, for every $p\geq 1$. In particular this implies that  $BV^2([0,1],\RR)$ is dense in $W^{1,p}([0,1],\RR)$, for every $p\geq 1$.
\par A vector field ${\bf u}$ belongs to $BV([0,1],\RR^2)$ ($BV^2([0,1],\RR^2)$ respectively)  if every component of ${\bf u}$ belongs to $BV([0,1],\RR)$ ($BV^2([0,1],\RR)$ respectively). 
We refer to~\cite{AFP} and~\cite{Demengel} for more properties of these spaces.

\bibliographystyle{plain}  
\bibliography{bibliography}	

\end{document}